\pdfoutput=1

\documentclass{amsart}

\newcommand{\myFigures}[1]{#1}

\usepackage[bookmarks=true,%
    colorlinks=true,%
    linkcolor=blue,%
    citecolor=blue,%
    filecolor=blue,%
    menucolor=blue,%
    urlcolor=blue,%
    breaklinks=true]{hyperref}

\usepackage{bm}
\usepackage{amsthm}
\usepackage{graphicx}

\usepackage[left=2.3cm,right=2.3cm,top=2cm,bottom=3cm]{geometry}

\usepackage{amssymb}

\newcommand\myLangle{\langle~}
\newcommand\myRangle{~\rangle}

\newcommand\grLangle{\langle~}
\newcommand\grRangle{~\rangle}

\newcommand\extraIEqual{}

\newcommand\overviewFigureSize{\footnotesize}
\newcommand\overviewFigureLinkSize{\normalsize}

\newcommand{\Id}{\mathrm{Id}}
\newcommand{\C}{\mathbb{C}}
\renewcommand{\H}{\mathbb{H}}
\newcommand{\Q}{\mathbb{Q}}
\newcommand{\Z}{\mathbb{Z}}
\newcommand{\PSL}{\mathrm{PSL}}
\newcommand{\SL}{\mathrm{SL}}
\newcommand{\PGL}{\mathrm{PGL}}

\newcommand{\homologyH}{H}
\newcommand{\sqrtMinusDSymbol}{\psi}
\newcommand{\sqrtMinusDSymbolTable}[1]{\sqrt{-#1}}

\newcommand{\idealNorm}{\mathcal{N}}
\newcommand{\myMod}{\!\!\mod}

\newcommand{\containedInI}{I \subset}

\newcommand{\markConjugate}[1]{\phantom{*}\hfill #1 \hfill*}
\newcommand{\explainConjugate}[1]{%
\bigskip%
\begin{minipage}{17.5cm}%
\rule{6.5cm}{0.4pt}\\%
$\phantom{L}^{\mbox{*}}$ Ideal conjugate to ideal in Figure~\ref{#1}.%
\end{minipage}}

\newtheorem{theorem}{Theorem}[section]
\newtheorem{lemma}[theorem]{Lemma}
\theoremstyle{definition}
\newtheorem{example}[theorem]{Example}

\makeatletter
\newsavebox\myboxA
\newsavebox\myboxB
\newlength\mylenA

\newcommand*\myOverline[1]{%
    \sbox{\myboxA}{$\m@th#1$}%
    \setbox\myboxB\null
    \ht\myboxB=\ht\myboxA%
    \dp\myboxB=\dp\myboxA%
    \ifdim\wd\myboxA<45pt%
        \wd\myboxB=\wd\myboxA
    \else
        \wd\myboxB=45pt
    \fi%
    \sbox\myboxB{$\m@th\overline{\copy\myboxB}$}
    \setlength\mylenA{\the\wd\myboxA}
    \addtolength\mylenA{-\the\wd\myboxB}%
    \ifdim\wd\myboxB<\wd\myboxA%
       \rlap{\hskip 0.5\mylenA\usebox\myboxB}{\usebox\myboxA}%
    \else
        \hskip -0.5\mylenA\rlap{\usebox\myboxA}{\hskip 0.5\mylenA\usebox\myboxB}%
    \fi}
\makeatother

\title{Technical Report: All Principal Congruence Link Groups}
\author{M. D. Baker}
\author{M. Goerner}
\author{A. W. Reid}

\address{\newline
IRMAR,\newline
Universit\'e de Rennes 1,\newline
35042 Rennes Cedex,\newline
France.}
\email{mark.baker@univ-rennes1.fr }
\address{\newline 
Pixar Animation Studios, \newline
1200 Park Avenue,\newline
Emeryville, CA 94608, USA.}
\email{enischte@gmail.com}
\address{\newline
Department of Mathematics,\newline
Rice University,\newline
Houston, TX 77005, USA}
\email{alan.reid@rice.edu}

\begin{document}

\begin{abstract}
This is a technical report accompanying the paper ``All Principal Congruence Link Groups'' classifying all principal congruence link complements in $S^3$ by the same authors. It provides a complete overview of all cases $(d,I)$ that had to be considered, as well as describes the necessary computations and computer programs written for the classification result.
\end{abstract}

\maketitle

\setcounter{tocdepth}{2}

\tableofcontents

\section{Introduction}

We follow the notation 
of \cite{bgr18:AllPrinCong}, in particular, let $d>0$ be a square-free integer and $I\subset O_d$ be an ideal in the ring of integers $O_d$ of $\Q(\sqrt{-d})$. The goal of this report is to give a complete proof of the following theorem stated in \cite{bgr18:AllPrinCong}:
\begin{theorem}
\label{main}
The following list of $48$ pairs $(d,I)$ describes all
principal congruence subgroups $\Gamma(I) < \PSL(2,O_d)$ such that
$\H^3/\Gamma(I)$ is a link complement in $S^3${\normalfont :}

\begin{enumerate}
\item {\bf $d=1$:}~$I=\myLangle 2 \myRangle $, $\extraIEqual\myLangle 2\pm i\myRangle $, $\extraIEqual\myLangle (1\pm i)^3\myRangle $, $\extraIEqual\myLangle 3\myRangle $, $\extraIEqual\myLangle 3\pm i\myRangle $, $\extraIEqual\myLangle 3\pm 2i\myRangle $, $\extraIEqual\myLangle 4\pm i\myRangle $.
\item {\bf $d=2$:}~$I = \myLangle 1\pm \sqrt{-2}\myRangle $, $\extraIEqual\myLangle 2\myRangle $, $\extraIEqual\myLangle 2\pm \sqrt{-2}\myRangle $, $\extraIEqual\myLangle 1\pm 2\sqrt{-2}\myRangle $, $\extraIEqual\myLangle 3\pm \sqrt{-2}\myRangle $.
\item {\bf $d=3$:}~$I = \myLangle 2\myRangle $,  $\extraIEqual \myLangle 3\myRangle $, $\extraIEqual \myLangle (5\pm \sqrt{-3})/2\myRangle $, $\extraIEqual \myLangle 3\pm \sqrt{-3}\myRangle $, $\extraIEqual \myLangle (7\pm \sqrt{-3})/2\myRangle $, $\extraIEqual \myLangle 4\pm \sqrt{-3}\myRangle $, $\extraIEqual \myLangle (9\pm \sqrt{-3})/2\myRangle $.
\item {\bf $d=5$:}~$I=\myLangle 3,(1\pm \sqrt{-5})\myRangle $.
\item {\bf $d=7$:}~$I = \myLangle (1\pm \sqrt{-7})/2\myRangle $, $\extraIEqual \myLangle 2\myRangle $, $\extraIEqual \myLangle (3\pm \sqrt{-7})/2\myRangle $, $\extraIEqual \myLangle \pm \sqrt{-7}\myRangle $, $\extraIEqual \myLangle 1\pm \sqrt{-7}\myRangle $, $\extraIEqual \myLangle (5\pm \sqrt{-7})/2\myRangle $, $\extraIEqual \myLangle 2\pm \sqrt{-7}\myRangle $, $\extraIEqual \myLangle (7\pm \sqrt{-7})/2\myRangle $, $\extraIEqual \myLangle (1\pm 3\sqrt{-7})/2\myRangle $.
\item {\bf $d=11$:}~$I = \myLangle (1\pm \sqrt{-11})/2\myRangle $, $\extraIEqual \myLangle (3\pm \sqrt{-11})/2\myRangle $, $\extraIEqual \myLangle (5\pm \sqrt{-11})/2\myRangle $.
\item {\bf $d=15$:}~$I = \myLangle 2,(1\pm \sqrt{-15})/2\myRangle $, $\extraIEqual\myLangle 3,(3\pm \sqrt{-15})/2\myRangle $, $\extraIEqual\myLangle (1\pm \sqrt{-15})/2\myRangle $, $\extraIEqual\myLangle 5,(5\pm \sqrt{-15})/2\myRangle $, $\extraIEqual\myLangle (3\pm \sqrt{-15})/2\myRangle $.
\item {\bf $d=19$:}~$I =\myLangle (1\pm \sqrt{-19})/2\myRangle $.
\item {\bf $d=23$:}~$I = \myLangle 2,(1\pm \sqrt{-23})/2\myRangle $,  $\extraIEqual \myLangle 3,(1\pm \sqrt{-23})/2\myRangle $, $\extraIEqual \myLangle 4,(3\pm \sqrt{-23})/2\myRangle $.
\item {\bf $d=31$:}~$I = \myLangle 2,(1\pm \sqrt{-31})/2\myRangle $, $\extraIEqual \myLangle 4,(1\pm \sqrt{-31})/2\myRangle $, $\extraIEqual \myLangle 5,(3\pm \sqrt{-31})/2\myRangle $.
\item {\bf $d=47$:}~$I = \myLangle 2,(1\pm \sqrt{-47})/2\myRangle $, $\extraIEqual \myLangle 3,(1\pm \sqrt{-47})/2\myRangle $, $\extraIEqual \myLangle 4,(1\pm \sqrt{-47})/2\myRangle $.
\item {\bf $d=71$:}~$I =\myLangle 2,(1\pm \sqrt{-71})/2\myRangle $.
\end{enumerate}
\end{theorem}

Recall that Theorem~2.1 and 2.2 in the Section ``Preliminaries and techniques'' of \cite{bgr18:AllPrinCong} reduce the proof to finitely many cases $(d,I)$ in which we need to decide whether $M=\H^3/\Gamma(I)$ is a link complement. This report is a self-contained treatise of all of these cases documenting the necessary data, computer programs, and computations. We have compiled a separate document \cite{BGR:links} containing link diagrams 

for all cases where a diagram is known and for which the complement in $S^3$ is a principal congruence manifold.

The first and third, respectively, the second author developed methods to decide whether $M=\H^3/\Gamma(I)$ is a link complement independently. To reflect this, the report is split into a preamble and two parts each illustrating the method by one set of authors. The preamble contains the three special cases 

$(1,\myLangle 4+3\sqrt{-1}\myRangle), (2,\myLangle 1+3\sqrt{-2}\myRangle),$ and $(3,\myLangle \frac{11+\sqrt{-3}}{2}\myRangle)$ where $M$ is homologically but not topologically a link complement.
 All other cases left by the aforementioned theorems in \cite{bgr18:AllPrinCong} are each covered by both methods. Thus, a complete proof of the classification result can be obtained in two ways, namely, by combining Section~\ref{sec:casesInsuffHom} with either Part~\ref{part:matthias} or Part~\ref{part:markAndAlan} of this report.

Of the three special cases, $(2,\myLangle 1+3\sqrt{-2} \myRangle)$ was particularly hard and required finding an automatic group structure using the program Monoid Automata Factory (MAF) \cite{MAF}. The other two special cases had already been addressed in \cite{goerner:regTessLinkComps}. In that paper, the principal congruence link complements of discriminant $d=1$ and $d=3$ were classified using a combinatorial argument which, unlike the proof of Theorem~2.2 in \cite{bgr18:AllPrinCong}, does not rely on the 6-Theorem or a bound on the systole. 

The methods in Part~\ref{part:matthias} and \ref{part:markAndAlan} both start with information about the Bianchi groups $\PSL(2,O_d)$ in question. For Part~\ref{part:matthias}, this is in the form of a triangulated Dirichlet domain of $\PSL(2,O_d)$ together with the face-pairing matrices, see Section~\ref{sec:diriBianchi}. The second author computed these using SageMath \cite{Sage} and they are available as Regina files, see Section~\ref{sec:fundPoly} for details. For Part~\ref{part:markAndAlan}, this is in the form of presentations of $\PSL(2,O_d)$ given by Swan and Page which we also list in Section~\ref{sec:presentationBianchi}. That section also includes the peripheral subgroups of $\PSL(2,O_d)$ which were not given by Page and had to be derived manually by the first and third authors.

Part~\ref{part:matthias} uses the fundamental domain of $\PSL(2,O_d)$ to build a triangulation of the principal congruence manifold $M=\H^3/\Gamma(I)$ which can be examined by the 3-manifold software SnapPy \cite{SnapPy} to determine whether $M$ is a link complement. The results are shown in diagrammatic form in Section~\ref{sec:overviewDiagrams}.

Part~\ref{part:markAndAlan} uses the presentation of $\PSL(2,O_d)$ to obtain a presentation of the quotient group $B(I)$ of $\PSL(2,O_d)$ by the subgroup $N(I)$ generated by the parabolic elements of $\Gamma(I)$. Using the computer algebra system Magma \cite{Magma}, $B(I)$ and $N(I)$ can be examined, the first step being to decide whether $|B(I)|=|\PSL(2,O_d/I)|$ which is a necessary condition for $M=\H^3/\Gamma(I)$ to be a link complement (see Section~\ref{sec:Prelims}). The results are tabulated in Section~\ref{sec:overviewTables}.

Both parts rely on Perelman's 
resolution of the Poincar\'e conjecture to show that $M=\H^3/\Gamma(I)$ is a link complement in $S^3$ by finding a peripheral curve for each cusp of $M$ such that the curves kill $\pi_1(M)$. If $|B(I)|=|\PSL(2,O_d/I)|$, this is equivalent to finding a parabolic element in $N(I)$ for each cusp of $M$ such that the elements kill $N(I)$ since $N(I)$ is isomorphic to $\Gamma(I)\cong \pi_1(M)$.

There is a large difference in the performance of the two methods and the first and third authors needed far less computer time than the second author to arrive at the classification result. The method of the second author requires at least $O(|\PSL(2,O_d/I)|)$ time to construct the triangulation and the difference in performance becomes more pronounced when $|\PSL(2,O_d/I)|$ is large. In these cases, we also observed that computing $N(I)$ with Magma's \texttt{NormalClosure} can take much more time than the computations with $B(I)$ to prove it large enough.

We are very grateful to the people who helped us with this work and refer the reader to \cite{bgr18:AllPrinCong} for detailed acknowledgments.

\section{Notation and preliminaries} \label{sec:Prelims}

Let $\psi=\sqrt{-d}$ and $\omega_d=(1+\sqrt{-d})/2$ (if $d\equiv 3\myMod 4$), respectively, $\omega_d=\sqrt{-d}$ (otherwise).

Besides the principal congruence group $\Gamma(I)=\ker(\pi)$, we will also consider another congruence group $\Gamma_1(I)=\pi^{-1}(P)$ where $P$ are the upper unit-triangular matrices in $\PSL(2,O_d/I)$ and $\pi:\PSL(2,O_d)\twoheadrightarrow \PSL(2,O_d/I)=\SL(2,O_d/I)/{\pm \mathrm{Id}}$.

Let $P_x$ be the parabolic elements in $\PSL(2,O_d)$ fixing $x\in\mathrm{P}^1(\Q(\sqrt{-d}))$ and $P_x(I)=P_x\cap \Gamma(I)$. Let $N(I)$ be the normal subgroup of $\PSL(2,O_d)$ obtained as the normal closure of all $P_x(I)$. Let $B(I)=\PSL(2,O_d)/N(I)$. Note that $N(I)\subset \Gamma(I)$ so there is an epimorphism $B(I)\twoheadrightarrow \PSL(2,O_d/I)$. As discussed in \cite[\S{}2.2]{bgr18:AllPrinCong}, the following statements are equivalent and necessary for $M$ to be a link complement:
\begin{itemize}
\item $|B(I)|=|\PSL(2,O_d/I)|$
\item $B(I)\twoheadrightarrow \PSL(2,O_d/I)$ is an isomorphism
\item $N(I)$ and $\Gamma(I)\cong \pi_1(M)$ are the same subgroup of $\PSL(2,O_d)$.
\end{itemize}

The size of $\PSL(2,O_d/I)$ is given in \cite{bakerReid:prinCong} by \label{sec:containsPslSizeEquation}
\begin{equation} \label{eqn:PslSizeEquation}
|\PSL(2,O_d/I)|=\frac{\idealNorm(I)^3}{\tau}\prod_{P|I} \left(1-\frac{1}{\idealNorm(P)^2}\right) \quad\mbox{where}\quad \tau= \begin{cases}
 1 &  \text{when } 2\in I\\
2 & \text{otherwise}
\end{cases}
\end{equation}
and where $P$ runs over the prime ideal divisors of $I$ and $\idealNorm(P)=|O_d/P|$ is the norm of $P$ (in particular, $|\PSL(2,O_d/I)|=6$ if $\idealNorm(I)=2$).

\section{Cases where homology was insufficient} \label{sec:casesInsuffHom}

\begin{lemma} \label{lemma:specialCases}
The principal congruence manifolds for $(1, \langle 4+3\sqrt{-1}\rangle)$ and $(3, \langle \frac{11+\sqrt{-3}}{2}\rangle)$ are not link complements.
\end{lemma}

\begin{proof}
Following Section~\ref{sec:Prelims}, it is sufficient to show that the group $B(I)=\PSL(2,O_d)/N(I)$, respectively, $B_{\PGL}(I)=\PGL(2,O_d)/N(I)$ is infinite. For $d=1$ and $d=3$, presentations for the latter group were given in \cite[Section~7.2]{goerner:regTessLinkComps}. The following Gap \cite{gap} code
\begin{center}
\begin{minipage}{15cm}
\begin{verbatim}
F := FreeGroup(3);;
P := F.1;; Q := F.2;; R := F.3;;
G := F/[ P^3, Q^4, R^4, (P*Q)^2, (Q*R)^2, (P*Q*R)^2, (Q*R^3)^4 * (R*Q*R^2)^3];;
G1:=DerivedSubgroup(G);;
for H in LowIndexSubgroupsFpGroup(G1,15) do
    if Index(G1, H) = 15 then
        Print(AbelianInvariants(Kernel(GQuotients(H, PSp(4,3))[1])));
    fi;
od;
\end{verbatim}
\end{minipage}
\end{center}
prints \verb"[ 0, 0, 0, 0, 0, 0, 0, 0, 0, 0, 2, 2," ...\verb", 2, 2, 3 ]" showing that $B_{\PGL}(\langle 4+3\sqrt{-1}\rangle)$ has a subgroup with Abelianization $\Z^{10} \oplus \Z_2^{22} \oplus \Z_3$. Similarly, the following code finds two subgroups of $B_{\PGL}(\myLangle \frac{11+\sqrt{-3}}{2}\myRangle)$ with Abelianization $\Z$:
\begin{center}
\begin{minipage}{15cm}
\begin{verbatim}
F := FreeGroup(3);;
P := F.1;; Q := F.2;; R := F.3;;
G := F/[ P^3, Q^3, R^6, (P*Q)^2, (Q*R)^2, (P*Q*R)^2, (Q*R^4)^5 * (R*Q*R^3)^1];;

G1:=DerivedSubgroup(G);;
for H in LowIndexSubgroupsFpGroup(G1,31) do
    if Index(G1, H) = 31 then
        Print(AbelianInvariants(DerivedSubgroup(H)));
    fi;
od;
\end{verbatim}
\end{minipage}
\end{center}
\end{proof}

The above proof for the case $(1, \langle 4+3\sqrt{-1}\rangle)$ has been taken from \cite[Section~10]{goerner:regTessLinkComps}. A different proof for $(3, \langle \frac{11+\sqrt{-3}}{2}\rangle)$ already appeared in \cite[Lemma~8.7]{goerner:regTessLinkComps}. For completeness, we also reproduce \cite[Lemma~4.1]{bgr18:AllPrinCong} and its proof:

\begin{lemma} \label{lemma:nasty}
The principal congruence manifold \,${\Bbb H}^3/\Gamma(\langle 1+3\sqrt{-2} \rangle)$\! is not homeomorphic to a link complement in $S^3$.
\end{lemma}

\begin{proof}
From \cite{Sw} (also see Section~\ref{sec:presentationBianchi}), we have the following presentation for 
$$\PSL(2,O_2)=\langle a,t,u | a^2=(ta)^3=(au^{-1}au)^2=tut^{-1}u^{-1}=1\rangle$$
where $$a=\begin{pmatrix} 0 & 1 \\ -1 & 0 \end{pmatrix},\quad t = \begin{pmatrix}1 & 1\\ 0 & 1\end{pmatrix},\quad \mbox{ and} \quad u = \begin{pmatrix}1 & \sqrt{-2}\\ 0 & 1\end{pmatrix}.$$
The group $P_\infty(\langle 1+3\sqrt{-2}\rangle)$ introduced in Section~\ref{sec:Prelims} is generated by the two parabolic elements $t^6u^{-1}$ and $t^{19}$. Since $\Q(\sqrt{-2})$ has class number $h_2=1$, the normal closure of these two elements in $\PSL(2,O_2)$ is $N_2(\langle 1+3\sqrt{-2}\rangle)$. Hence adding these two elements to the above presentation yields
$$B(\myLangle 1+3\sqrt{-2}\myRangle)=\frac{\PSL(2,O_2)}{N_2(\langle 1+3\sqrt{-2}\rangle)} = \langle a,t,u | a^2=(ta)^3=(au^{-1}au)^2=tut^{-1}u^{-1}=t^6u^{-1}=t^{19}=1\rangle.$$
We give this presentation of $B(\myLangle 1+3\sqrt{-2}\myRangle)$ to MAF \cite{MAF} in form of a file \texttt{myGroup}:
\begin{center}
\begin{minipage}{15cm}
\begin{verbatim}
_RWS := rec(
    isRWS := true,
    generatorOrder := [_g1,_g2,_g3,_g4,_g5],
    inverses :=       [_g1,_g3,_g2,_g5,_g4],
    ordering := "shortlex",
    equations := [
        [_g2*_g1*_g2,_g1*_g3*_g1],
        [_g1*_g5*_g1*_g4,_g5*_g1*_g4*_g1],
        [_g2*_g4,_g4*_g2],
        [_g2^4,_g4*_g3^2],
        [_g2^10,_g3^9] ]);
\end{verbatim}
\end{minipage}
\end{center} and then call (which takes about 2 hours of time on a MacBook pro with a 2.6Ghz Intel Core i5):
\begin{center}
\begin{minipage}{15cm}
\begin{verbatim}
$ automata -no-kb myGroup         # find automatic structure
$ gpaxioms myGroup                # verify automatic structure for correct group
[...]
Checking relation _g2*_g4=_g4*_g2
Checking relation _g2^4=_g4*_g3^2
Checking relation _g2^10=_g3^9
Axiom check succeeded.
$ fsacount myGroup.wa             # count words accepted by word aceptor
The accepted language is infinite
\end{verbatim}
\end{minipage}
\end{center}
Since \texttt{automata} always finds a word acceptor automaton that accepts exactly one word for any group element, this proves $B(\myLangle 1+3\sqrt{-2}\myRangle)$ to be infinite. From the discussion in Section~\ref{sec:Prelims}, we deduce that $\Gamma(\langle 1+3\sqrt{-2}\rangle)$ is not a link group. 
\end{proof}

We are very grateful to Alun Williams who helped us implement the above proof.

\clearpage

\part{Using the method by the second author} \label{part:matthias}

\section{Introduction}

In order to obtain a triangulation of a principal congruence manifold $M=\H^3/\Gamma(I)$, the second author has written two computer programs available at \cite{goerner:data}. As described in Section~\ref{sec:diriBianchi}, the first of 
these programs can compute a Dirichlet domain for $\PSL(2,O_d)$. This serves as input to the second program described in Section~\ref{sec:prinCongTrig} to construct the triangulation of $M$ given an ideal $I$. 

This triangulation can be given to SnapPy \cite{SnapPy} for further examination. To show that particular $M$ is a link complement, we find Dehn-fillings trivializing the fundamental group of $M$, see Section~\ref{sec:linkCompCert}. We can use the homology $H_1(M)$ to show that $M$ is not a link complement, see Section~\ref{sec:homOfCov}. Since computing $H_1(M)$ can be prohibitively expensive in some cases, we also enabled the program to construct a triangulation of the congruence manifold $M_1=\H^3/\Gamma_1(I)$.

We conclude this part with some remarks on how the triangulations were simplified before computing their homology in Section~\ref{sec:technicalRemarks}. We also want to point out that many of the homology groups could only be determined due to SnapPy's efficient implementation of homology.

\section{Overview diagrams} \label{sec:overviewDiagrams}

This section shows the overview diagrams for the finitely many cases we need to consider. The diagrams for class number $h_d=1$ and for higher class numbers $h_d>1$ are slightly different. Hence, we split them up into two sections, each beginning with a brief explanation how to read the diagrams.

\subsection{Class number one}

Recall that the ideal $\myLangle x\myRangle$ is the same when multiplying $x$ by a unit and that complex conjugation only flips the orientation of the principal congruence manifold. Hence, we only need to consider generators $x$ lying in the first quadrant or a $\pi/4$-, respectively, $\pi/6$-wedge for $d=1$ or $3$. Furthermore, by \cite[Theorem~2.2]{bgr18:AllPrinCong}, we only need to consider those generators lying strictly within the circle of radius 6.

For each generator $x$, the diagram either indicates
\begin{itemize}
\item that $M=\H^3/\Gamma(\myLangle x\myRangle)$ is a link complement or
\item gives the reason why $M$ is not a link complement which can be
\begin{itemize}
\item that $M$ is an orbifold
\item that the homology $H_1(M)/\imath_*(H_1(\partial M))$ (where $\imath:\partial M \to M$ is the inclusion of the boundary) is non-trivial, e.g., $\Z^5$ for $I=\myLangle 3+3\sqrt{-1}\myRangle$, or
\item one of the lemmas in Section~\ref{sec:casesInsuffHom}.
\end{itemize}
\end{itemize}

\myFigures{

\begin{figure}[h]
\begin{center}
\overviewFigureSize
\input{overviewFigures_gen/discriminant1}
\end{center}
\caption{$d=1$ (compare to Table~\ref{tbl:results1}). \label{fig:disc1}}
\end{figure}

\begin{figure}
\begin{center}
\overviewFigureSize
\input{overviewFigures_gen/discriminant2}
\end{center}
\caption{$d=2$ (compare to Table~\ref{tbl:results2}). \label{fig:disc2}}
\end{figure}

\begin{figure}
\begin{center}
\overviewFigureSize
\input{overviewFigures_gen/discriminant3}
\end{center}
\caption{$d=3$ (compare to Table~\ref{tbl:results3}). \label{fig:disc3}}
\end{figure}

\begin{figure}
\begin{center}
\overviewFigureSize
\input{overviewFigures_gen/discriminant7}
\end{center}
\caption{$d=7$ (compare to Table~\ref{tbl:results7}). \label{fig:disc7}}
\end{figure}

\begin{figure}
\begin{center}
\overviewFigureSize
\input{overviewFigures_gen/discriminant11}
\end{center}
\caption{$d=11$ (compare to Table~\ref{tbl:results11}). \label{fig:disc11}}
\end{figure}

\begin{figure}
\begin{center}
\overviewFigureSize
\input{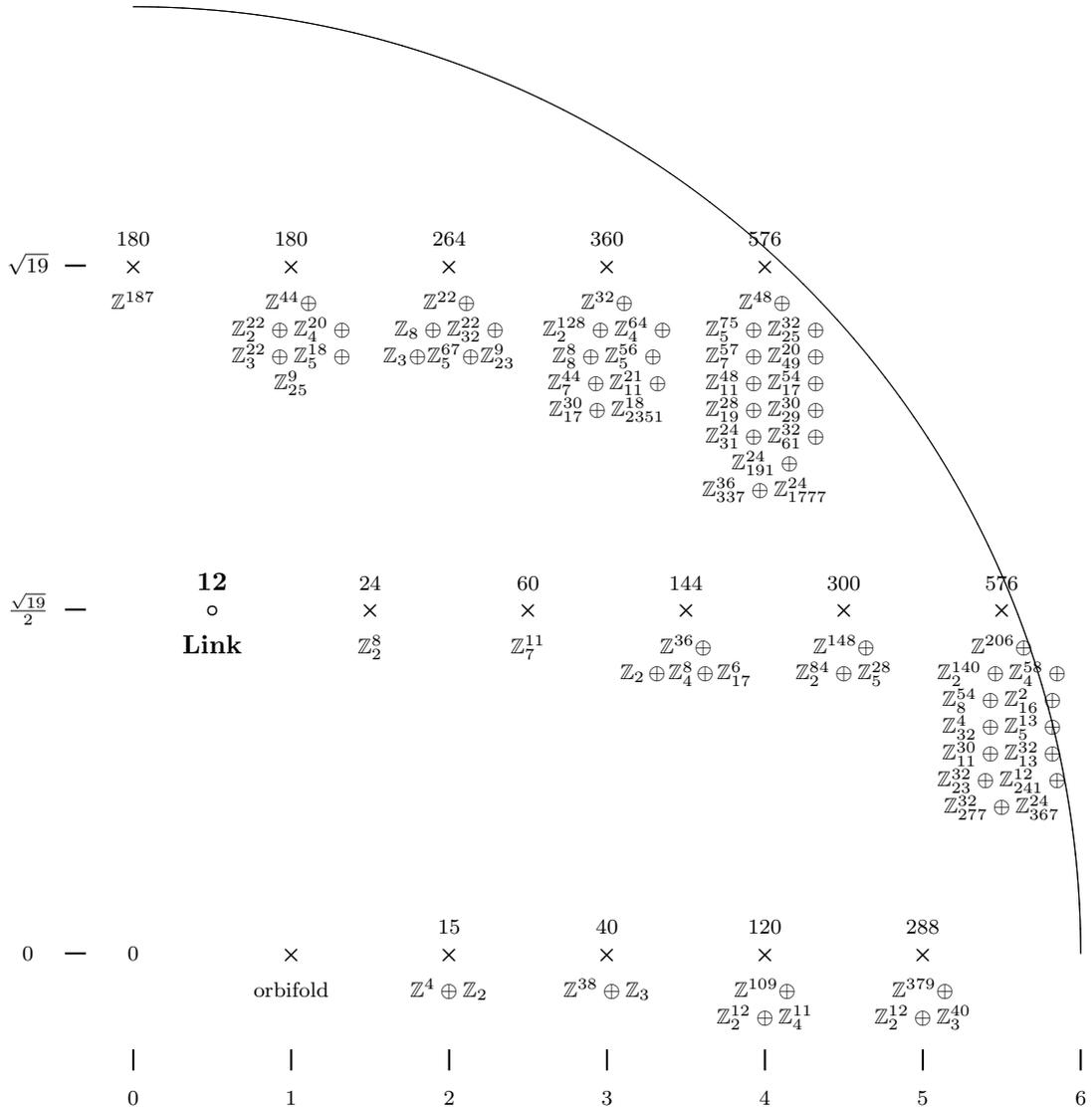}
\end{center}
\caption{$d=19$ (compare to Table~\ref{tbl:results19}). \label{fig:disc19}}
\end{figure}

}

\clearpage

\subsection{Higher class numbers}
For higher class numbers $h_d>1$, there are non-principal ideals $I$ requiring two generators $I=\langle x, y\rangle$. For such an ideal $I$, we always pick as primary generator an element $x\in I$ with the smallest absolute value $|x|$. Considerations similar to $h_d=1$ apply and we can pick $x$ to be in the first quadrant lying strictly within the circle of radius $\sqrt{39}$. The secondary generator is the element $y$ with the smallest absolute value generating $I$ together with $x$, preferable in the first quadrant, but always in the first or second quadrant. In the diagrams, each ideal corresponds to a box. Boxes for ideals having the same primary generator $x$ are grouped together.

Such a box either indicates that 
\begin{itemize}
\item $M=\H^3/\Gamma(I)$ is a link complement (also giving the number of components) or
\item gives the reason why $M$ is not a link complement which can be
\begin{itemize}
\item that $M$ is an orbifold
\item that the homology $H_1(M)/\imath_*(H_1(\partial M))$ is non-trivial or
\item an argument from Section~\ref{sec:homOfCov} as illustrated by Examples~\ref{example:gamma1Degree}, \ref{example:coverIncl}, and \ref{example:coverInclAndGamma1}.
\end{itemize}
\end{itemize}

The arguments from Section~\ref{sec:homOfCov} are needed here, since, unlike for $h_d=1$, computing the homology for all cases in question was infeasible.

\myFigures{

\begin{figure}[h]
\begin{center}
\overviewFigureSize
\input{overviewFiguresHigher_gen/discriminant5}
\end{center}
\caption{$d=5$ (compare to Table~\ref{tbl:results5}). \label{fig:disc5}}
\end{figure}

\begin{figure}[h]
\begin{center}
\overviewFigureSize
\input{overviewFiguresHigher_gen/discriminant6}
\end{center}
\caption{$d=6$ (compare to Table~\ref{tbl:results6}). \label{fig:disc6}}
\end{figure}

\begin{figure}[h]
\begin{center}
\overviewFigureSize
\input{overviewFiguresHigher_gen/discriminant15}
\end{center}
\caption{$d=15$ (compare to Table~\ref{tbl:results15}). \label{fig:disc15}}
\end{figure}

\begin{figure}[h]
\begin{center}
\overviewFigureSize
\input{overviewFiguresHigher_gen/discriminant23}
\end{center}
\caption{$d=23$ (compare to Table~\ref{tbl:results23}). \label{fig:disc23}}
\end{figure}

\begin{figure}[h]
\begin{center}
\overviewFigureSize
\input{overviewFiguresHigher_gen/discriminant31}
\end{center}
\caption{$d=31$ (compare to Table~\ref{tbl:results31}). \label{fig:disc31}}
\end{figure}

\begin{figure}[h]
\begin{center}
\overviewFigureSize
\input{overviewFiguresHigher_gen/discriminant39}
\end{center}
\caption{$d=39$ (compare to Table~\ref{tbl:results39}). \label{fig:disc39}}
\end{figure}

\begin{figure}[h]
\begin{center}
\overviewFigureSize
\input{overviewFiguresHigher_gen/discriminant47}
\end{center}
\caption{$d=47$ (compare to Table~\ref{tbl:results47}). \label{fig:disc47}}
\end{figure}

\begin{figure}[h]
\begin{center}
\overviewFigureSize
\input{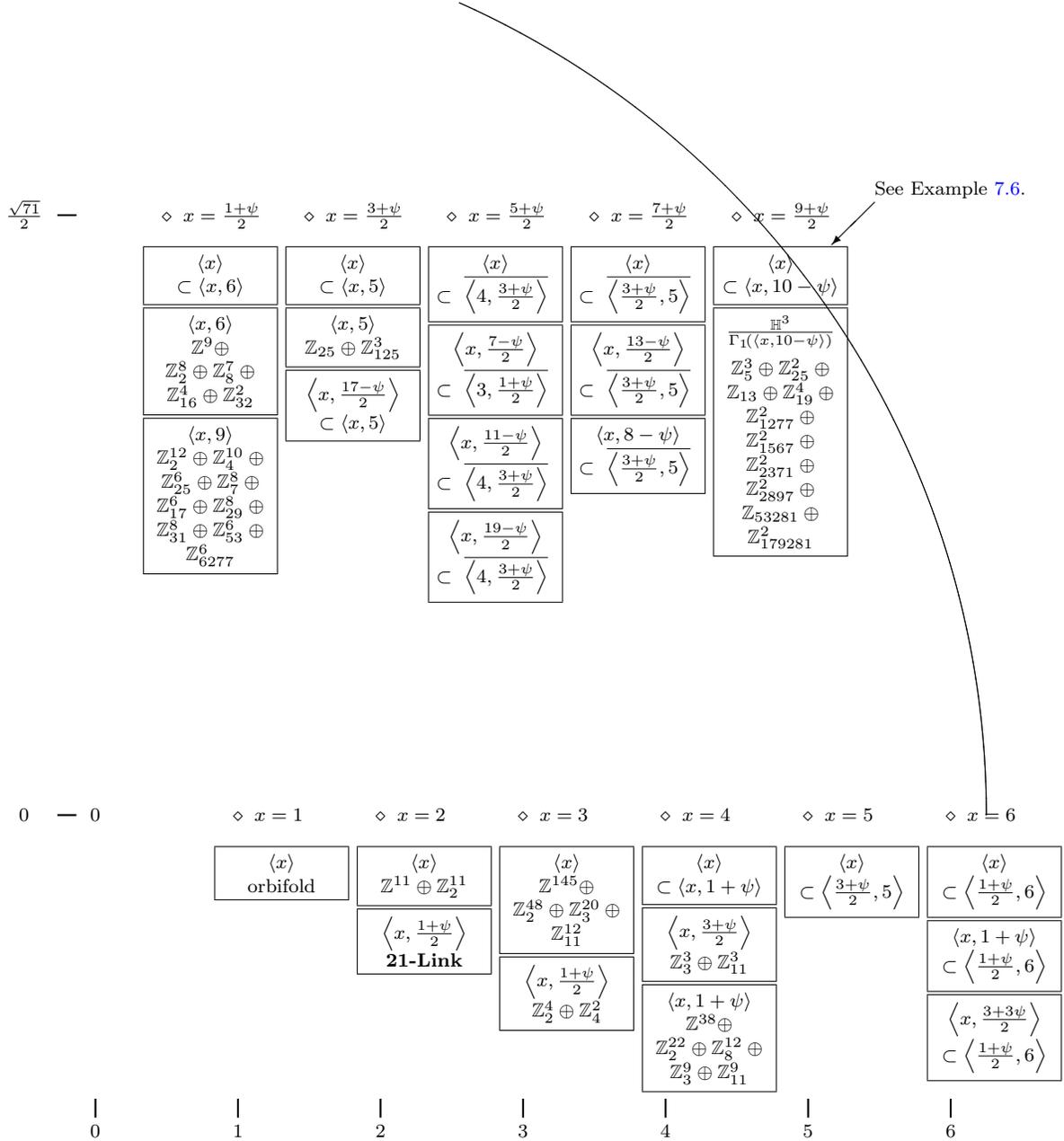}
\end{center}
\caption{$d=71$ (compare to Table~\ref{tbl:results71}).\label{fig:disc71}}
\end{figure}

}

\clearpage

\section{Link complement certificates} \label{sec:linkCompCert}

To prove $M=\H^3/\Gamma$ to be a link complement in the 48 necessary cases, we produce a SnapPy triangulation of each $M$ with meridians set in such a way that filling along each meridian trivializes the fundamental group. These peripheral curves were found using techniques similar to \cite[Section 7.3.2]{goerner:regTessLinkComps} and we provide the SnapPy triangulations at \cite[\href{http://unhyperbolic.org/prinCong/prinCong/LinkComplementCertificates/}{\texttt{prinCong/LinkComplementCertificates/}}]{goerner:data}. Thus, the reader can verify the result by $(1,0)$-filling each cusp and checking that SnapPy's simplified presentation of the resulting fundamental group has no generators. The same directory also contains the script \href{http://unhyperbolic.org/prinCong/prinCong/LinkComplementCertificates/proveLinkComplement.py}{\texttt{proveLinkComplement.py}} to do this automatically for all 48 cases.

\section{Homology and covering spaces} \label{sec:homOfCov}

Let $M$ be a compact, orientable $3$-manifold with boundary consisting of disjoint tori. Let $\imath:\partial M\to M$ be the inclusion of the boundary. We often work with $H_1(M)/\imath_*(H_1(\partial M))$ which can easily be computed from $H_1(M)$ since \cite[Lemma~9.6]{goerner:regTessLinkComps} states that $H_1(M)=H_1(M)/\imath_*(H_1(\partial M)) \oplus \Z^c$ where $c$ is the number of boundary components of $M$. If $H_1(M)/\imath_*(H_1\partial(M))\not=0$, then $M$ cannot be a link complement. Furthermore, small covers of $M$ often cannot be link complements either:

\begin{lemma} \label{lemma:minCovDegree}
A cover $N\to M$ with degree less than $|H_1(M)/\imath_*(H_1(\partial M))|$ has non-trivial $H_1(N)/\imath_*(H_1(\partial N))$.
\end{lemma}

\begin{proof}
Consider the map $$f_M:\pi_1(M)\twoheadrightarrow H_1(M)\twoheadrightarrow \frac{H_1(M)}{\imath_*(H_1(\partial M))}.$$ A cover $N\to M$ corresponds to a subgroup $\Gamma\subset\pi_1(M)$ and we have an analogous map
$$f_N:\Gamma\cong \pi_1(N)\twoheadrightarrow H_1(N)\twoheadrightarrow \frac{H_1(N)}{\imath_*(H_1(\partial N))}.$$
Note that $[\pi_1(M):\ker(f_M)]=|H_1(M)/\imath_*(H_1(\partial M))|$ and that the degree of $N\to M$ is given by $[\pi_1(M):\Gamma]$, so if $\Gamma\subset\ker(f_M)$, then the degree of $N\to M$ would be at least $|H_1(M)/\imath_*(H_1(\partial M))|$. Hence, there is a $\gamma\in\Gamma$ with $\gamma\not\in\ker(f_M)$.\\
Assume that $H_1(N)/{\imath_*(H_1(\partial N))}=0$. Note that the $\ker(f_M)$ is the normal closure of all peripheral curves and commutators in $\pi_1(M)$. Similarly, $\ker(f_N)$ is the normal closure of all peripheral curves and commutators in $\Gamma\cong\pi_1(N)$ and, since $\Gamma\subset\pi_1(M)$, $\ker(f_N)\subset \ker(f_M)$. Hence $\gamma\not\in\ker(f_N)$. A contradiction.
\end{proof}

Let $\idealNorm(I)=|O_d/I|$ denote the norm of an ideal $I$.

\begin{lemma} \label{lemma:gamma1CoverDegree}
Let $M_1=\H^3/\Gamma_1(I)$ be a congruence manifold (i.e., $\Gamma_1(I)$ is torsion-free). If $$\left| \frac{H_1(M_1)}{\imath_*(H_1(\partial M_1))}\right| > \idealNorm(I),$$
then $\Gamma(I)$ is not a link group.
\end{lemma}

\begin{proof}
The degree of the cover $M\to M_1$ is $|P|=\idealNorm(I)$.
\end{proof}

\begin{example} \label{example:gamma1Degree}
For $(31, \myLangle\sqrt{-31}\myRangle)$, computing $\homologyH_1(M)/\imath_*(\homologyH_1(\partial M))$ was infeasible. However, it can be proven that $M$ is not a link complement by computing $|\homologyH_1(M_1)/\imath_*(\homologyH_1(\partial M_1))|$ which is infinite (see Figure~\ref{fig:disc31}) and thus larger than $\idealNorm(I)=31$. A similar argument can be used for $(47, \myLangle 5\myRangle)$, see Figure~\ref{fig:disc47}, and $(71, \myLangle (9+\sqrt{-71})/2, 10 - \sqrt{-71}\myRangle)$, see Figure~\ref{fig:disc71}.
\end{example}

\begin{lemma} \label{lemma:prinCongCoverDegree}
Let $M=\H^3/\Gamma(I)$ be a principal congruence manifold (i.e., $\Gamma(I)$ is torsion-free). If $J\subset I$ is an ideal such that $$\left| \frac{H_1(M)}{\imath_*(H_1(\partial M))}\right| > \frac{|\PSL(2,O_d/J)|}{|\PSL(2,O_d/I)|},$$
then $N=\H^3/\Gamma(J)$ is not a principal congruence link complement.
\end{lemma}

\begin{proof}
$M$ is a cover of the Bianchi orbifold $Q_d=\H^3/\PSL(2,O_d)$ with covering group $\PSL(2,O_d/I)$. Thus, the degree of the cover $\H^3/\Gamma(J)\to \H^3/\Gamma(I)$ is given by the right-hand side of the inequality.
\end{proof}

We can compute the right hand side of the inequality in Lemma~\ref{lemma:prinCongCoverDegree} using Equation~\ref{eqn:PslSizeEquation} in Section~\ref{sec:containsPslSizeEquation}.

\begin{example} \label{example:coverIncl}
Let $d=5$, $J=\myLangle 4+2\sqrt{-5}\myRangle$ and $I=\myLangle 2\myRangle$. The right hand side of the inequality of Lemma~\ref{lemma:prinCongCoverDegree} is $15552/48=324$, but $\left| H_1(M)/\imath_*(H_1(\partial M))\right|$ is infinite (see Figure~\ref{fig:disc5}), so $N$ is not a link complement.
\end{example}

\begin{example} \label{example:coverInclAndGamma1}
Let $d=71$, $I=\myLangle (9+\sqrt{-71})/2, 10-\sqrt{-71}\myRangle$ and $J=\myLangle (9+\sqrt{-71})/2\myRangle$. $\H^3/\Gamma(J)$ is not a principal congruence link complement by Lemma~\ref{lemma:minCovDegree} since $\H^3/\Gamma(J)\to\H^3/\Gamma(I)\to\H^3/\Gamma_1(I)=M_1$ has degree $6\cdot 19=114$ which is smaller than $\left| H_1(M_1)/\imath_*(H_1(\partial M_1))\right|$, see Figure~\ref{fig:disc71}.
\end{example}

\section{Computing Dirichlet domains of Bianchi orbifolds} \label{sec:diriBianchi}

For constructing triangulations of (principal) congruence manifolds in Section~\ref{sec:prinCongTrig}, we need a triangulated fundamental domain for $\PSL(2,O_d)$ with the extra information specified in Section~\ref{sec:fundPoly}. Computer programs to compute fundamental domains have been written before, notably by Riley \cite{Riley:PoincareTheoremFundamental} and more recently by Page \cite{page:ArithKlein}. Unfortunately, the published data (e.g., \cite{Sw}) do not include all the information we needed, at least not in a form that was easily computer parsable.

Hence, the second author implemented his own program (using SageMath \cite{Sage}) to produce a Regina \cite{Regina} triangulation of a Bianchi orbifold with suitable $\PSL(2,O_d)$-matrix annotations. The program could, in theory, produce a non-trivial covering space of the Bianchi orbifold $Q_d=\H^3/\PSL(2,O_d)$ instead of $Q_d$ itself. We can, however, easily rule this out since the volume of the Bianchi orbifold $Q_d$ is known for the $d$ we need to consider.

During the implementation, we noticed that the Dirichlet domain in the Klein model can be scaled by $\sqrt{d}$ in one direction such that all coordinates of the vertices become rational. Proving this is the motivation for discussing the computation of Dirichlet domains in the detail we do here.

\subsection{Data for a fundamental domain for a Bianchi group} \label{sec:fundPoly}

The data we need about the fundamental domain for $\PSL(2,O_d)$ is the combinatorics of a fundamental polyhedron $P$ for the Bianchi group $\PSL(2,O_d)$ together with the following information for each face $f$ of $P$:
\begin{itemize}
\item another face $f'$ of $P$ called the mate face
\item the face-pairing matrix $g_f\in\PSL(2,O_d)$ such that $g_f f'=f$
\item for each (finite or ideal) vertex $v$ of $f$ the corresponding vertex $v'$ of $f'$ with $g_f v'=v$,
\item for each edge of $P$, the singular order the edge has in the Bianchi orbifold $Q_d=\H^3/\PSL(2,O_d)$ (1 if the edge is non-singular in $Q_d$).
\end{itemize}

For simplicity, we triangulate $P$ by taking the barycentric subdivision. We index the vertices of the resulting simplices such that vertex $i$ of a simplex corresponds to the center of an $i$-cell of $P$. This results in a triangulation where the gluing permutations are always the identity.

Each simplex $\Delta_j$ of the barycentric subdivision has a ``mate'' simplex and a mating matrix $g_j\in\PSL(2,O_d)$ that takes face $3$ of the simplex to face $3$ of the mate simplex. We obtain a triangulation of $Q_d$ by gluing each simplex to its mate along face $3$. This is the triangulation we store and the fundamental domain can be easily obtained by just ungluing each face $3$. Along the triangulation, we store the mating matrices $g_j\in\PSL(2,O_d)$ in a separate array. Note that the singular locus of $Q_d$ falls onto the edges of the faces with index $3$ and for face 3 of a simplex $\Delta_j$ we obtain three numbers describing the singular orders of its three edges in $Q_d$. We store these triples of natural numbers for all simplices $\Delta_j$ in a separate array as well.

We provide this information as Regina \cite{Regina} readable files at \cite[\href{http://unhyperbolic.org/prinCong/prinCong/src/bianchiOrbifold/data/}{\texttt{prinCong/src/bianchiOrbifold/data/}}]{goerner:data} with further details explained at 
\cite[\href{http://unhyperbolic.org/prinCong/prinCong/src/bianchiOrbifold/orbifoldData.py}{\texttt{prinCong/src/bianchiOrbifold/orbifoldData.py}}]{goerner:data}.

\subsection{Dirichlet domains} We use a Dirichlet domain for the Bianchi group $\PSL(2,O_d)$ as a fundamental domain. Fix a base point $p_0$ in hyperbolic space. Associate to a matrix $m$ the half space containing $p_0$ that is limited by the plane bisecting $p_0$ and the image of $p_0$ under the action of $m$. From a sample of matrices in $\PSL(2,O_d)$, we obtain a candidate polyhedron $P$ for the Dirichlet domain by intersecting the half spaces associated to the matrices. Each face $f$ of $P$ comes from the intersection with a half space associated with a matrix $m$ which will become the face-pairing matrix $g_f$. If we can consistently recover the information described in Section~\ref{sec:fundPoly}, $P$ is the fundamental domain for a (hopefully trivial) cover of the Bianchi orbifold $Q_d$. In other words, we need to check that for each face $f$ of $P$, there is a face $f'$ with matrix $g_{f'}={g_{f}}^{-1}$ which will become the mate face. Furthermore, for each face $f$ an
 d each vertex $v$ of $f$, we need to find a vertex $v'$ of $P$ with $g_fv'=v$.

It is convient to let $p_0$ be the origin $0$ in the Klein or Poincar\'e ball model. Unfortunately, there are matrices in $\PSL(2,O_d)$ that fix the origin. But we can pick a suitable matrix $l\in\PSL(2,\Q(\sqrt{-d}))$ and let $m'=l^{-1}ml$ instead of $m\in\PSL(2,O_d)$ act on ${\Bbb H}^3$ or $B^3$.

\subsection{Poincar\'e extension for the Poincar\'e ball}

Let $\bf H $ denote Hamilton's quaternions and ${\Bbb H}^3=\{z+tj : z\in\C, t>0\}\subset {\bf H}$ and $B^3=\{x+yj+zk:x^2+y^2+z^2<1\}$ be the upper half space, respectively, Poincar\'e ball model of hyperbolic 3-space. There is an action of suitable $2\times 2$ matrices with quaternions as coefficients on ${\bf H} \cup\{\infty\}$ given by
$$T:\left(\begin{array}{cc}a&b\\c&d\end{array}\right) \mapsto \left(w \mapsto (aw+b)\cdot(cw+d)^{-1}\right).$$
If $m\in\PSL(2,\C)$, then $T(m)_{|{\Bbb H}^3}$ is an isometry of ${\Bbb H}^3$. Furthermore, letting
$$m_{{\Bbb H}^3\to B^3}=\left(\begin{array}{rr}1&-j\\1&j\end{array}\right)\quad\mbox{and}\quad m_{B^3\to{\Bbb H^3}}=\left(\begin{array}{rr}1&1\\j&-j\end{array}\right),$$
$T(m_{{\Bbb H}^3\to B^3})$ and $T(m_{B^3\to{\Bbb H}^3})$ convert between ${\Bbb H}^3$ and $B^3$.
Thus, $T(m_{{\Bbb H}^3\to B^3}\cdot m \cdot m_{B^3\to{\Bbb H}^3})_{|B^3}$ is the isometry of the Poincar\'e ball model $B^3$ corresponding to $m$.

\subsection{Hyperbolic midpoint and conversion to Klein model}

It is convenient to work in the Klein model since hyperbolic half spaces become Euclidean half spaces (intersected with the unit ball).

When converting between the Klein and the Poincar\'e ball model, we do so such that the origin and the boundary of the unit ball are fixed.

\begin{figure}[h]
\begin{center}
\includegraphics[width=6cm]{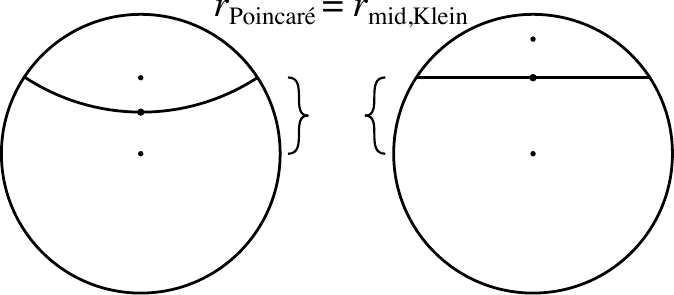}
\end{center}
\caption{Taking the midpoint in the Poincar\'e model and converting it to the Klein model gives the same Euclidean point.\label{fig:PoincareKlein}}
\end{figure}

\begin{lemma}
Let $p$ be the point in the Poincar\'e ball model with Euclidean coordinates $(x_p, y_p, z_p)$. The result of taking the hyperbolic midpoint between $p$ and the origin and then converting that midpoint to the Klein model also has coordinates $(x_p,y_p,z_p)$, see Figure~\ref{fig:PoincareKlein}. Thus, the plane bisecting $p$ and the origin has equation $x_px+y_py+z_pz=x_p^2+y_p^2+z_p^2$ in the Klein model.\label{lemma:KleinPoincareMidpoint}
\end{lemma}

\begin{proof}
Let $r_{\mathrm{Klein}}$ and $r_{\mathrm{Poincar\acute{e}}}$ be the Euclidean distance of the origin to a point in the Klein model, respectively, the corresponding   point in the Poincar\'e ball model.
We have $$r_{\mathrm{Poincar\acute{e}}}=\frac{r_{\mathrm{Klein}}}{1+\sqrt{1-r_{\mathrm{Klein}}^2}}.$$
Note that this is the same relationship we have between the Euclidean distance $r_{\mathrm{Poincar\acute{e}}}$ of a point in Poincar\'e ball model and $r_{\mathrm{mid, Poincar\acute{e}}}$ of the hyperbolic midpoint between that point and the origin:
$$r_{\mathrm{mid, Poincar\acute{e}}}=\frac{r_{\mathrm{Poincar\acute{e}}}}{1+\sqrt{1-r_{\mathrm{Poincar\acute{e}}}}}.$$
Thus, we have $r_{\mathrm{mid,Klein}}=r_{\mathrm{Poincar\acute{e}}}$.
\end{proof}

\subsection{Rational plane equation}

\begin{lemma}
Let $m\in \PSL(2,\Q(\sqrt{-d}))$. Let $(x_p,y_p,z_p)$ be the Euclidean coordinates of the image of the origin $0$ in the Poincar\'e ball model $B^3$ under the action of $m$. Then, $x_p, y_p\in\Q$ and $z_p\in \sqrt{d}\,\Q$. Thus, in the Klein model, the equation for the plane associated to $m$ has rational coefficients when replacing $z$ by $\sqrt{d}\,z'$ in Lemma~\ref{lemma:KleinPoincareMidpoint}.
\end{lemma}

\begin{proof} Let $$m=\left(\begin{array}{cc}a&b\\c&d\end{array}\right).$$
The image of the origin in $B^3$ is given by $(x_p,y_p,z_p)=T(m_{{\Bbb H}^3\to B^3}\cdot m \cdot m_{B^3\to{\Bbb H}^3})(0)$. Note that the origin in $B^3$ corresponds to $j$ in ${\Bbb H}^3$ and a standard
calculation gives:
\begin{equation} \label{eqn:PoincareExt}
T(m)(z+tj)=\left((az+b)(\overline{cz+d})+a\overline{c}t^2+tj\right)/|c(z+tj)+d|^2,
\end{equation}
 so $$T(m)(j)=(b\overline{d}+a\overline{c}+j)/|cj+d|^2\in \Q +i\sqrt{d} \Q + j\Q.$$ Applying the conversion $T(m_{{\Bbb H}^3\to B^3})$ now gives the result.
\end{proof}

Unfortunately, we do need to deal with a further quadratic extension of $\Q(\sqrt{d})$ when verifying the correspondences between the vertices $v$ and $v'$ of a face $f$ and its mate face $f'$.

\section{Triangulations of (principal) congruence manifolds} \label{sec:prinCongTrig}

Let $I$ be an ideal in $O_d$. In this section we describe how to construct a triangulation of $M=\H^3/\Gamma(I)$, respectively, $M_1=\H^3/\Gamma_1(I)$ using copies of the triangulated fundamental polyhedron $P$ of $\PSL(2,O_d)$ from the data in Section~\ref{sec:fundPoly}.

\subsection{Principal congruence manifolds}

We label each copy of $P$ by a matrix $m \in \PSL(2, O_d/I )$. We use the following algorithm:
\begin{enumerate}
\item Start with a ``base'' copy $P_{\Id}$.
\item While there is a copy $P_m$ with an unglued face $f$:
\begin{enumerate}
\item Compute $m'=mg_f\in \PSL(2, O_d/I )$.
\item If there is no copy $P_{m'}$ yet, create one.
\item Glue face $f$ of $P_m$ to the mate face $f'$ of $P_{m'}$ such that the vertices are matching as described in the information about the fundamental polyhedron.
\end{enumerate}
\end{enumerate}
In the implementation, we store the $P_m$ in an array and use a dictionary mapping matrices $m\in\PSL(2,O_d/I)$ to an index in the array for fast lookups.

To determine whether $\H^3/\Gamma(I)$ is an orbifold, we can compare the degrees of the edges of the resulting triangulation with the degrees of the edges of the triangulation of $Q_d$ multiplied by the respective orbifold orders stored with the triangulation of $Q_d$.

\subsection{Congruence manifolds $\H^3/\Gamma_1(I)$}

Note that $\H^3/\Gamma_1(I)$ is obtained from $\H^3/\Gamma(I)$ as quotient by the action of upper, respectively, lower unit-triangular matrices. In other words, each copy of $P$ in the triangulation of $\H^3/\Gamma_1(I)$ is labeled by a vector $v\in (O_d/I)^2/\pm 1$ which corresponds to the first row of a matrix $m \in \PSL(2, O_d/I )$. When computing the label $v'\in (O_d/I)^2/\pm 1$ for a neighboring copy, we need a lift of $v$ to $m \in \PSL(2, O_d/I )$ so that $v'$ is given as first row of $mg_f$. For efficiency, we remember such a matrix $m\in \PSL(2,O_d/I)$ for each copy $P_v$ of $P$. In other words, we store pairs $(P_v,m)$ in an array and use a dictionary mapping $v\in (O_d/I)^2/\pm 1$ to an index in the array for fast lookups.

Thus, the algorithm becomes:
\begin{enumerate}
\item Start with a ``base'' copy $P_{(1,0)}$.
\item While there is a copy $P_v$ with an unglued face $f$:
\begin{enumerate}
\item Let $m$ be the matrix stored with $P_v$. Compute $m'=mg_f\in \PSL(2, O_d/I )$ and let $v'$ be the first row of $m'$.
\item If there is no copy $P_{v'}$ yet, create a copy $P_{v'}$ and store $m'$ with it.
\item Glue face $f$ of $P_v$ to the mate face $f'$ of $P_{v'}$ such that the vertices are matching as described in the information about the fundamental polyhedron.
\end{enumerate}
\end{enumerate}

\subsection{Calculations in $O_d/I$} Conceptually, a polyhedron is labeled by an elements in $m\in\PSL(2,O_d/I)$. In the implementation, we label a polyhedron by a canonical representative, i.e., by a canonical matrix with coefficients in $O_d$ representing $m$. For this, we need a procedure to reduce any representative in $O_d$ of an element in $O_d/I$ to a canonical representative.

Given a 2-vector $v$, the reduced form of $v$ with respect to the vectors $v_1$ and $v_2$ is the element in $v+\Z v_1+\Z v_2$ in the parallelogram spanned by $v_1$ and $v_2$. In other words, $v$ is reduced with respect to $v_1$ and $v_2$ if $$v \left(\begin{array}{c}v_1\\v_2\end{array}\right)^{-1}\in [0,1)^2.$$
Let us associate the vector $(a,b)$ to an element $a+b\sqrt{-d}\in O_d$. Let us fix two vectors $v_1$ and $v_2$ that span $I$ as a lattice. We can then reduce a representative in $O_d$ by reducing the associated vector $(a,b)$ by $v_1$ and $v_2$.

It is left to find such $v_1$ and $v_2$ given generators $x_1, \dots, x_k\in O_d$ of the ideal. As a lattice, $I$ is spanned by the vectors associated to $x_1, x_1\omega_d, \dots, x_k, x_k\omega_d$. We need a procedure that takes such a set of vectors and returns two vectors spanning the same lattice as the input vectors. By iterating, it suffices to have a method that produces two vectors spanning the same lattice as three given vectors. This can be done by repeatedly reducing one vector by the other two vectors until one of them is zero.

\section{Technical remarks about computing homologies} \label{sec:technicalRemarks}

The triangulation of a (principal) congruence manifold produced in Section~\ref{sec:prinCongTrig} has both finite and ideal vertices and can be quiet large. Even though the result is the same, computing the homology of the unsimplified triangulation
with finite and ideal vertices is much slower than first simplifying the triangulation and then computing the homology of the simplified triangulation. For example, the largest triangulation we encountered has 1843200 simplices in the case $(31,\myLangle 5\myRangle)$). Removing finite vertices reduced the number of tetrahedra to 122704 simplices.

SnapPy \cite{SnapPy} has a procedure to remove all finite vertices of a triangulation of a cusped manifold. However, this procedure does not scale to large triangulations. Thus, we implemented our own method to simplify the triangulation:
\begin{enumerate}
\item Perform a coarsening of the barycentric subdivision: there is a group of four simplices about each edge from vertex 1 to 2; collapse all these groups to a single simplex each simultaneously.
\item Collapse edges (preferring edges with high order) for as long as there is an edge which can be collapsed without changing the topology -- similar to what the method \verb"collapseEdge" in Regina \cite{Regina} does.
\end{enumerate}
Even though this procedure might not in general remove all finite vertices, it does so for all the triangulations we needed to consider here.

Unfortunately, Regina's \verb"collapseEdge" invalidates and recomputes the entire skeleton of the triangulation each time an edge is collapsed. Hence, the above procedure would not scale to large triangulations using Regina's implementation. Therefore, we reimplemented Regina's method so that it performs a more targeted invalidation: only the edge classes near the collapsed edge recomputed and the two vertices at the ends of the collapsed edge are merged.

We use SnapPy to compute the homology of the triangulation simplified this way. For large triangulations, Dunfield and Culler implemented the homology as follows:
\begin{enumerate}
\item Using a sparse-matrix representation, simplify the matrix performing row operations as long as there is a $\pm 1$ in the matrix.
\item Compute the Smith normal form using algorithms described in \cite[Chapter~2.4]{cohen:compNumbThy} (and implemented in Pari \cite{pari}).
\end{enumerate}

\clearpage

\part{Using the method by the first and third authors} \label{part:markAndAlan}

\section{Introduction}
The main computational tool used by the first and third authors to establish when a candidate principal congruence manifold is, or is not, homeomorphic to a link complement in $S^3$ is Magma \cite{Magma}.  We refer the reader to
the papers \cite{bakerReid:prinCong} and \cite{bakerReid:higherPrinCong} for more on the background to these methods.  However, we note that, thanks to the second author, many of the Magma routines have now been
automated, and this is what is included in this report. In addition, we take the opportunity to correct some mistakes in some of the entries in the tables of \cite{bakerReid:higherPrinCong} that were uncovered whilst checking our calculations.  This did not affect the outcome of whether the principal congruence manifold was a link complement.

All files needed to reproduce the results in this part are available at \cite[\href{http://unhyperbolic.org/prinCong/prinCong/magma/}{\texttt{prinCong/magma/}}]{goerner:data}.

\section{Presentations for the Bianchi groups} \label{sec:presentationBianchi}
 
\subsection{Class number one}
The following presentations are from \cite{Sw}:
\begin{eqnarray*}
\PSL(2,O_1) & = & \grLangle a,\ell,t,u~|~ \ell^2 = (t\ell)^2 = (u\ell)^2 =
(a\ell)^2 = a^2 = (ta)^3 = (ua\ell)^3 = 1, [t,u]=1\grRangle\\
\PSL(2,O_2) & = & \grLangle a,t,u~|~a^2=(ta)^3=(au^{-1}au)^2=1, [t,u]=1\grRangle\\
\PSL(2,O_3) & = & \grLangle a,\ell,t,u~|~\ell^3=a^2=(a\ell)^2=(ta)^3=(ua\ell)^3=1,
\ell^{-1}t\ell=t^{-1}u^{-1},\ell^{-1}u\ell=t, \\
& & [t,u]=1\grRangle\\
\PSL(2,O_7) & = & \grLangle a,t,u~|~a^2=(ta)^3=(atu^{-1}au)^2=1,  [t,u]=1\grRangle\\
\PSL(2,O_{11}) & = & \grLangle a,t,u~|~a^2=(ta)^3=(atu^{-1}au)^3=1, [t,u]=1\grRangle\\
\PSL(2,O_{19}) & =  &
\grLangle a,b,t,u~|~a^2=(ta)^3=b^3=(bt^{-1})^3=(ab)^2=(at^{-1}ubu^{-1})^2=1,  
[t,u]=1\grRangle.\\
\end{eqnarray*}
The matrices are given by
$$a=\begin{pmatrix}0 &  -1\\ 1 & 0\end{pmatrix},\quad b=\begin{pmatrix}1-\omega_{19} & 2\\ 2 & \omega_{19}\end{pmatrix},\quad
 \quad
 t=\begin{pmatrix}1 &  1\\ 0 & 1\end{pmatrix},$$
 
 $$u=\begin{pmatrix}1 &  \omega_d\\ 0 & 1\end{pmatrix}~~\mbox{for } d\not=3 \quad \mbox{respectively} \quad u=\begin{pmatrix}1 &  \omega_d^2\\ 0 & 1\end{pmatrix}~~\mbox{for } d=3,$$
 and
  $$\ell = \begin{pmatrix}-i & 0 \\ 0 & i\end{pmatrix}
 ~~ \mbox{for } d=1 
  \quad \mbox{respectively} \quad
  \ell = \begin{pmatrix}1/\omega_3^2 & 0 \\ 0 & \omega_3^2\end{pmatrix}
 ~~ \mbox{for } d=3.$$
In each case, the parabolic group fixing $\infty$ corresponding to the peripheral subgroup of the one-cusped Bianchi orbifold $Q_d=\H^3/\PSL(2,O_d)$ is given by:
$$P_\infty = \myLangle t,u\myRangle.$$

\subsection{Swan's presentations for higher class numbers ($d=5, 6, 15$)}

The following presentations are from \cite{Sw}:
\begin{eqnarray*}
\PSL(2,O_5) &=& \grLangle a,t,u,b,c~|~a^2=b^2=(ta)^3=(ab)^2=1, \\
& &
(aubu^{-1})^2=acatc^{-1}t^{-1}=
ubu^{-1}cbtc^{-1}t^{-1}=1, [t,u]=1\grRangle\\
\PSL(2,O_6)&=&\grLangle a,t,u,b,c~|~a^2=b^2=(ta)^3=(atb)^3=1,\\
& &
(atubu^{-1})^3=
t^{-1}ctubu^{-1}c^{-1}b^{-1}=1, [t,u]=[a,c]=1\grRangle\\
\PSL(2,O_{15}) &=& \grLangle a,t,u,c~|~a^2=(ta)^3=1,\\
& &
ucuatu^{-1}c^{-1}u^{-1}a^{-1}t^{-1}=1, [t,u]=[a,c]=1\grRangle
\end{eqnarray*}

The matrices are given by
$$a=\begin{pmatrix}0 &  -1\\ 1 & 0\end{pmatrix}, \quad
 t=\begin{pmatrix}1 &  1\\ 0 & 1\end{pmatrix}, \quad
 u=\begin{pmatrix}1 &  \omega_d\\ 0 & 1\end{pmatrix}$$

and

$$\left.\begin{array}{c}
b=\begin{pmatrix}-\omega_5 &  2\\ 2 & \omega_5\end{pmatrix},\quad
c=\begin{pmatrix}-\omega_5-4 &  -2\omega_5\\ 2\omega_5 & \omega_5-4\end{pmatrix}
\end{array}\right\} ~~\mbox{for } d= 5$$
$$\left.\begin{array}{c}
b=\begin{pmatrix}-1-\omega_6 & 2-\omega_6\\ 2 &  1+\omega_6\end{pmatrix},\quad
c=\begin{pmatrix}5 & -2\omega_6\\ 2\omega_6 & 5\end{pmatrix}
\end{array}\right\} ~~\mbox{for } d= 6$$
$$\left.\begin{array}{c}
c=\begin{pmatrix}4 & 1-2\omega_{15}\\  2\omega_{15}-1 & 4\end{pmatrix}
\end{array}\right\} ~~\mbox{for } d= 15.$$
Up to conjugacy, all parabolic subgroups $P_x$ fixing $x$ are given by
$$\left.\begin{array}{c}
P_\infty=\myLangle t,~u\myRangle,\quad P_\frac{1-\sqrt{-5}}{2}=\myLangle tb,~tu^{-1}ct^{-1}\myRangle
\end{array}\right\} ~~\mbox{for } d= 5$$

$$\left.\begin{array}{c}
P_\infty=\myLangle t,~u\myRangle,\quad P_\frac{-\sqrt{-6}}{2}=\myLangle tb,~u^{-1}c^{-1}\myRangle
\end{array}\right\} ~~\mbox{for } d= 6$$

$$\left.\begin{array}{c}
P_\infty=\myLangle t,~u\myRangle,\quad P_\frac{1+\sqrt{-15}}{4}=\myLangle uca,~c^{-1}au^{-1}c^{-1}u^{-1}ta\myRangle
\end{array}\right\} ~~\mbox{for } d= 15$$

\subsection{Page's presentations for higher class numbers}

The presentations for the remaining class numbers were done by A.~Page using a suite of computer packages he recently developed (see  
\cite{page:ArithKlein}) to study arithmetic
Kleinian groups.  The presentations given here were communicated to the first and third authors by A.~Page.

\begin{eqnarray*}
\PSL(2,O_{23}) & = & \grLangle g_1,g_2,g_3,g_4,g_5~|~g_3^3=(g_3g_2)^2=1, \\
& &
g_5g_2^{-1}g_3^{-1}g_5^{-1}g_1^{-1}g_2^{-1}g_3^{-1}g_1=
g_4^{-1}g_5g_3g_2g_5^{-1}g_2g_4g_3=1, [g_1,g_2]=[g_4,  
g_5]=1 \grRangle \\
\PSL(2,O_{31}) & = & \grLangle g_1,g_2,g_3,g_4,g_5~|~g_2^3=(g_2g_1^{-1})^2=1, \\
& &
g_4g_1^{-1}g_3^{-1}g_2g_3g_4^{-1}g_2g_4g_3^{-1}g_1^{-1}g_2g_3g_4^{-1}g_2=  1, \\
& &
g_5g_3^{-1}g_2g_3g_4^{-1}g_2g_1^{-1}g_5^{-1}g_2^{-1}g_4g_3^{-1}g_2^{-1}g_3g_1=1, \\
& &
g_2g_3g_4^{-1}g_2g_1^{-1}g_4g_3^{-1}g_2g_3g_4^{-1}g_1g_2^{-1}g_4g_3^{-1}=1,
[g_1,g_3]=[g_4,g_5]=1 \grRangle \\
\PSL(2,O_{39}) & = & \grLangle g_1,g_2,g_3,g_4,g_5,g_6,g_7~|~g_3^3=1,\\
& &
(g_3g_5)^2=
(g_1^{-1}g_3^{-1})^2=(g_5^{-1}g_1)^3=(g_7g_5^{-1}g_7^{-1}g_1)^3= 1, \\
& &
g_5^{-1}g_1g_6^{-1}g_4^{-1}g_5g_4g_1^{-1}g_6 =
g_4^{-1}g_5g_4g_2^{-1}g_7g_5^{-1}g_7^{-1}g_2 = 1, \\
& &
g_6g_1^{-1}g_5g_6^{-1}g_4^{-1}g_5g_4g_1^{-1}g_4^{-1}g_5g_4g_1^{-1}=1,
[g_2, g_1]=[g_3^{-1}, g_7^{-1}]=[g_4, g_6]=1 \grRangle \\
\PSL(2,O_{47}) & = & \grLangle g_1,g_2,g_3,g_4,g_5,g_6,g_7~|~g_1^3=
(g_2^{-1}g_1)^2=1, \\
& &
g_2^{-1}g_1g_6g_1^{-1}g_2g_6^{-1}=
g_6g_2^{-1}g_4^{-1}g_5g_3^{-1}g_6^{-1}g_4g_2g_3g_5^{-1}= 1,\\
& &
g_7^{-1}g_2^{-1}g_5^{-1}g_4g_1g_4^{-1}g_2g_7g_4g_1^{-1}g_4^{-1}g_5=
g_3g_5^{-1}g_4g_1g_4^{-1}g_2g_5g_3^{-1}g_2^{-1}g_4^{-1}g_1^{-1}g_4=1,\\
& &
g_5^{-1}g_4g_1g_4^{-1}g_7^{-1}g_2^{-1}g_4g_1^{-1}g_4^{-1}g_5g_3^{-1}g_2g_3g_7  
= 1,
[g_5, g_7]=[g_3, g_2]=1 \grRangle \\
\PSL(2,O_{71}) & = & 
\grLangle g_1,g_2,g_3,g_4,g_5,g_6,g_7,g_8,g_9~|~g_8^3=(g_8g_7^{-1})^2=1,\\
& &
g_1^{-1}g_3g_7g_3^{-1}g_1g_7^{-1} =
g_6g_3g_6^{-1}g_7g_9^{-1}g_3^{-1}g_9g_7^{-1} =1,\\
& &
g_7^{-1}g_6g_3g_6^{-1}g_5^{-1}g_2g_7g_5g_6g_3^{-1}g_6^{-1}g_2^{-1} = 1,\\
& &
g_8g_7^{-1}g_1g_5g_6g_3^{-1}g_1g_5g_7g_8^{-1}g_5^{-1}g_1^{-1}g_3g_6^{-1}g_5^{-1}g_1^{-1}  = 1, \\
& &
g_4^{-1}g_7^{-1}g_5^{-1}g_2g_1^{-1}g_3g_7g_9g_4g_1g_7^{-1}g_2^{-1}g_5g_7g_9^{-1}g_3^{-1}  = 1,\\
& &
g_5g_8g_7^{-1}g_5^{-1}g_1^{-1}g_7g_9g_6g_1g_5g_8g_7^{-1}g_5^{-1}g_1^{-1}g_3g_6^{-1}g_9^{-1
}g_7^{-1}g_3^{-1}g_1 = 1,\\
& &
g_2g_6g_1g_5g_7g_8^{-1}g_5^{-1}g_1^{-1}g_3g_6^{-1}g_7g_8^{-1}g_5^{-1}g_2^{-1}g_5g_7g_8^{-1
}g_5^{-1}g_1^{-1}g_7g_8^{-1}g_1g_5g_6g_3^{-1}g_6^{-1} = 1,\\
& &
[g_8^{-1},g_4]=1 \grRangle
\end{eqnarray*}

The matrices (or more precisely a representative of the generator of $\PSL$) are given by 
$$
 \left.\begin{array}{c}
 g_1=\begin{pmatrix}1 &-1+\omega_{23}\\ 0 & 1\end{pmatrix},\quad  
g_2= \begin{pmatrix}1 & 1\\ 0 & 1\end{pmatrix},\quad
g_3= \begin{pmatrix}0 & 1\\ -1 & 1\end{pmatrix}\\ [\bigskipamount]
g_4=\begin{pmatrix}3+\omega_{23} & -4+\omega_{23}\\ -2+\omega_{23} &  
-1-\omega_{23}\end{pmatrix},\quad
g_5=\begin{pmatrix}5-\omega_{23} & 1+2\omega_{23}\\
            2+\omega_{23} & -3+\omega_{23}\end{pmatrix}
\end{array}\right\} ~~\mbox{for } d=23
$$

$$
\left.\begin{array}{c}
g_1 = \begin{pmatrix}1 & -1\\ 0 & 1\end{pmatrix},\quad
g_2 = \begin{pmatrix}0 & 1\\ -1 & 1\end{pmatrix},\quad
g_3 = \begin{pmatrix}1 & \omega_{31}\\ 0 & 1\end{pmatrix},\\  [\bigskipamount]
 g_4 = \begin{pmatrix}3 &  -2+2\omega_{31}\\ \omega_{31} & -5\end{pmatrix},\quad
g_5 = \begin{pmatrix}3-2\omega_{31} & 7+\omega_{31}\\ 4 &  -1+2\omega_{31}\end{pmatrix}
\end{array}\right\} ~~\mbox{for } d=31
$$

$$
\left.\begin{array}{c}
g_1= \begin{pmatrix}1 & 1\\ 0 & 1\end{pmatrix},\quad
g_2= \begin{pmatrix}1  & \omega_{39}\\ 0 & 1\end{pmatrix},\quad
g_3= \begin{pmatrix}0 & 1\\ -1 & 1\end{pmatrix}\\  [\bigskipamount]
g_4 =  \begin{pmatrix}-3-\omega_{39} & 7-2\omega_{39}\\ 2-\omega_{39} &  5+\omega_{39}\end{pmatrix},\quad
g_5 = \begin{pmatrix}3-\omega_{39} & 2+\omega_{39}\\ 3 &  -1+\omega_{39}\end{pmatrix}\\  [\bigskipamount]
g_6 = \begin{pmatrix}7-\omega_{39} & 2+3\omega_{39}\\ 2+\omega_{39} &  -5+\omega_{39}\end{pmatrix},\quad
g_7 = \begin{pmatrix}6-\omega_{39} & -1+2\omega_{39}\\ 1-2\omega_{39} &  5+\omega_{39}\end{pmatrix}
\end{array}\right\} ~~\mbox{for } d=39
$$

$$
\left.\begin{array}{c}
g_1= \begin{pmatrix}-1 & 1\\ -1 & 0\end{pmatrix},\quad
g_2 = \begin{pmatrix}1 & 1\\ 0 & 1\end{pmatrix},\quad
g_3 = \begin{pmatrix}1 & -1+\omega_{47}\\ 0 & 1\end{pmatrix}\\  [\bigskipamount]
g_4 =  \begin{pmatrix}-2+\omega_{47} & 5\\ -3 & 1+\omega_{47}\end{pmatrix},\quad
g_5 = \begin{pmatrix}5  & -3+3\omega_{47}\\ \omega_{47} & -7\end{pmatrix}\\  [\bigskipamount]
g_6 = \begin{pmatrix}-4+\omega_{47} & 3+\omega_{47}\\ -3-\omega_{47} &  -4+\omega_{47}\end{pmatrix},\quad
g_7 = \begin{pmatrix}1-2\omega_{47} & 11+\omega_{47}\\ 4 &  -3+2\omega_{47}\end{pmatrix}
\end{array}\right\} ~~\mbox{for } d=47
$$

$$
\left.\begin{array}{c}
g_1 = \begin{pmatrix}-5 & 5-3\omega_{71}\\  -1+\omega_{71} & -10-\omega_{71}\end{pmatrix},\quad g_2=\begin{pmatrix}-3+2\omega_{71} & -17-\omega_{71}\\ -4 & 1-2\omega_{71}\end{pmatrix}\\  [\bigskipamount]
g_3= \begin{pmatrix}5 & -2\omega_{71}\\ 1-\omega_{71} & -7\end{pmatrix},\quad
g_4 = \begin{pmatrix}-5 & 2+\omega_{71}\\ -2-\omega_{71} & -3+\omega_{71}\end{pmatrix}\\  [\bigskipamount]
g_5 = \begin{pmatrix}-6-3\omega_{71} & 13-2\omega_{71}\\ 5-\omega_{71} &  4+\omega_{71}\end{pmatrix},\quad
g_6 = \begin{pmatrix}-1+2\omega_{71} & 12\\ -6 & -1+2\omega_{71}\end{pmatrix}\\  [\bigskipamount]
g_7 = \begin{pmatrix}1 & -1\\ 0 & 1\end{pmatrix},\quad
g_8 = \begin{pmatrix}0 & -1\\ 1 & -1\end{pmatrix},\quad
g_9 = \begin{pmatrix}1+\omega_{71} & -7\\ 3 & -2+\omega_{71}\end{pmatrix}
\end{array}\right\} ~~\mbox{for } d=71.
$$

All (up to conjugacy) parabolic subgroups $P_x$ fixing $x$ are given by
$$\left.\begin{array}{c}
P_\infty=\myLangle g_2,~g_1\myRangle,\quad
P_\frac{1-\sqrt{-23}}{4}=\myLangle g_4,~g_5\myRangle,\quad
P_\frac{-1-\sqrt{-23}}{4}=\myLangle g_4g_3g_2,~g_2^{-1}g_5g_3g_2\myRangle
\end{array}\right\} ~~\mbox{for } d= 23$$
$$\left.\begin{array}{c}
P_\infty=\myLangle g_1,~g_3\myRangle,\quad
P_\frac{1-\sqrt{-31}}{4}=\myLangle g_4,~g_5\myRangle,\quad
P_\frac{-1-\sqrt{-31}}{4}=\myLangle g_1g_5,~g_3^{-1}g_2g_3g_4^{-1}g_2g_5\myRangle
\end{array}\right\} ~~\mbox{for } d= 31$$

$$\left.\begin{array}{c}
P_\infty=\myLangle g_1,~g_2\myRangle,\quad
P_\frac{1-\sqrt{-39}}{4}=\myLangle g_4,~g_6\myRangle,\quad\\ [\bigskipamount]
P_\frac{1-\sqrt{-39}}{5}=\myLangle g_5^{-1}g_6,~g_4g_1^{-1}g_6\myRangle,\quad
P_\frac{3-\sqrt{-39}}{6}=\myLangle g_5,~g_4g_2^{-1}g_7\myRangle
\end{array}\right\} ~~\mbox{for } d= 39$$

$$\left.\begin{array}{c}
P_\infty=\myLangle g_2,~g_3\myRangle,\quad
P_\frac{1-\sqrt{-47}}{4}=\myLangle g_5,~g_7\myRangle,\quad
P_\frac{3-\sqrt{-47}}{4}=\myLangle g_2g_7,~g_4g_1^{-1}g_4^{-1}g_5\myRangle,\\  [\bigskipamount]
P_\frac{1-\sqrt{-47}}{6}=\myLangle g_6g_2^{-1}g_4^{-1},~g_5g_3^{-1}g_2^{-1}g_4^{-1}\myRangle,\quad
P_\frac{1+\sqrt{-47}}{6}=\myLangle g_6^{-1}g_1^{-1}g_4,~g_3g_5^{-1}g_4g_1\myRangle
\end{array}\right\} ~~\mbox{for } d= 47$$

$$\left.\begin{array}{c}
P_\infty=\myLangle g_7,~g_1^{-1}g_3\myRangle,\quad
P_\frac{1-\sqrt{-71}}{4}=\myLangle g_2,~g_6g_1g_5g_7g_8^{-1}g_5^{-1}g_1^{-1}g_3g_6^{-1}g_7g_8^{-1}g_5^{-1}\myRangle,\\  [\bigskipamount]
P_\frac{1+\sqrt{-71}}{6}=\myLangle g_3,~g_6^{-1}g_7g_9^{-1}\myRangle,\quad
P_\frac{-1-\sqrt{-71}}{4}=\myLangle g_7g_2,~g_6g_3g_6^{-1}g_5^{-1}g_7^{-1}\myRangle,\\  [\bigskipamount]
P_\frac{-1+\sqrt{-71}}{6}=\myLangle g_7g_9g_6,~g_3^{-1}g_1g_5g_8g_7^{-1}g_5^{-1}g_1^{-1}\myRangle,\quad
P_\frac{3+\sqrt{-71}}{8}=\myLangle g_3g_9g_4,~g_4^{-1}g_7^{-1}g_5^{-1}g_2g_7g_1^{-1}\myRangle,\\  [\bigskipamount]
P_\frac{3-\sqrt{-71}}{8}=\myLangle g_4g_1g_7^{-1}g_2^{-1}g_5g_7g_9^{-1}g_3^{-1}g_8^{-1}g_4^{-1},~g_6g_3^{-1}g_1g_5g_8g_7^{-1}g_5^{-1}g_1^{-1}g_6^{-1}g_9^{-1}g_8^{-1}g_4^{-1}\myRangle
\end{array}\right\} ~~\mbox{for } d= 71.$$

\section{Finding $N(I)$ and $B(I)$}

Given an ideal $I$, we need to find words in the generators of $\PSL(2,O_d)$ that generate the $P_x(I)$. Let $(n,k,l)$ be a triple of integers such that $n$ and $k+l\omega_d$ generate the ideal $I$ as a lattice: $I=n\Z+(k+l\omega_d)\Z$. The functions to generate $B(I)$ are given in the file \href{http://unhyperbolic.org/prinCong/prinCong/magma/GroupB.m}{\texttt{GroupB.m}}.

\subsection{Class number one} \label{sec:findBIClassNumberOne} $P_\infty$ is generated by $t^n$ and $t^k u^l$. The following Magma code gives us the $N(I)$ as defined in Section~\ref{sec:Prelims} for the case $d=2$ and $I=\myLangle 1+\sqrt{-2}\myRangle = 3\Z + (1+\sqrt{-2})\Z$ giving rise to $(n,k,l)=(3,1,1)$:
\begin{center}
\begin{minipage}{13cm}
\begin{verbatim}
Bianchi2<a,t,u> := Group<a,t,u|a^2,(t*a)^3,(a*u^-1*a*u)^2,(t,u)>;

function N2(n, k, l)
    P := sub<Bianchi2|t^n, t^k * u^l>;
    N := NormalClosure(Bianchi2, P);
    N := Rewrite(Bianchi2, N); // For optimization
    return N;
end function;

N2(3,1,1);
\end{verbatim}
\end{minipage}
\end{center}
A presentation for $B(I)$ can be obtained by simply adding the generators of $P_\infty$ as relations to $\PSL(2,O_d)$. Thus, using the function in Table~\ref{tbl:magma2}, $B(I)$ for the same case is obtained by the Magma code \texttt{B2(3,1,1)}.

Note that for $d=3$, we need to pick the triple $(n,k,l)$ such that $I=n\Z + (k+l\omega_3^2)\Z$.

For the cases $(d,I)$ that needed to be considered, the tables in Section~\ref{sec:tablesClassNumberOne} show the arguments $n, k, l$ passed to the Magma functions to generate the $B(I)$.

\subsection{Higher class numbers $h_d$ with $d\not\equiv 3\myMod 4$} Let us denote the parabolic subgroups and their generators given in Section \ref{sec:presentationBianchi} by $P_{(1)}=\myLangle p_{(1),1}, p_{(1),2}\myRangle,\dots,P_{(h_d)}=\myLangle p_{(h_d),1}, p_{(h_d),2}\myRangle$. For $d=5$ and $d=6$, the generators were chosen so that $P_{(i)}(I)$ is generated by $p_{(i),1}^n$ and $p_{(i),1}^k p_{(i),2}^l$. Thus, the Magma functions to generate $N(I)$ and $B(I)$ similarly only need the triple $(n,k,l)$ as input, see Tables~\ref{tbl:magma5} and ~\ref{tbl:magma6}.

\subsection{Higher class numbers $h_d$ with $d\equiv 3\myMod 4$} \label{sec:findBIHigherClassNumber} Unfortunately, we cannot find generators $p_{(i),j}$ for the parabolic subgroups of $\PSL(2,O_d)$ such that each $P_{(i)}(I)$ is again always generated by $p_{(i),1}^n$ and $p_{(i),1}^k p_{(i),2}^l$. Instead, the Magma functions (e.g., in Tables~\ref{tbl:magma15} and \ref{tbl:magma39}) generating $N(I)$ and $B(I)$  take  triples $(n_1,k_1,l_1)$, \dots, $(n_{h_d},k_{h_d},l_{h_d})$ that must be chosen such that $\myLangle p_{(i),1}^{n_i}, p_{(i),1}^{k_i} p_{(i),2}^{l_i}\myRangle=P_{(i)}(I)$. These triples are shown in, e.g., Tables~\ref{tbl:results15} and \ref{tbl:results39}.

These triples were determined by computing the matrices for the $p_{(i),j}$. In \cite{bakerReid:prinCong} and \cite{bakerReid:higherPrinCong} these were computed manually.  However, this has been automated, and 
we provide code for SageMath \cite{Sage} to verify that the triples given in these tables are valid. Given a triple $(n_i,k_i,l_i)$ with $n_i, l_i>0$ and an ideal $I$, the code will check that
\begin{enumerate}
\item $p_{(i),1}^{n_i} \equiv \pm \Id \myMod I$ and $p_{(i),1}^{k_i} p_{(i),2}^{l_i}\equiv \pm \Id \myMod I$ and
\item for any $(s,t)\in \Z^2\setminus \{(0,0)\}$ with $0 \leq t < l_i$ and $t k_i  / l_i \leq s < n_i + t k_i  / l_i$, we have $$p_{(i),1}^{s} p_{(i),2}^{t}\not\equiv \pm \Id\myMod I.$$
\end{enumerate}
For a triple $(n_i,k_i,l_i)$ with $n_i>0$ and $l_i<0$, the code will flip the signs of $k_i$ and $l_i$ before performing the above checks.

\begin{example}
Running the command \texttt{sage -python \href{http://unhyperbolic.org/prinCong/prinCong/magma/checkMatricesAndPeripherals15.py}{checkMatricesAndPeripherals15.py}} will output \texttt{True} 15 times to indicate that the triples in Table~\ref{tbl:results15} for the 15 relevant ideals $I\subset O_{15}$  are valid. 
\end{example}

\clearpage

\section{Tables} \label{sec:overviewTables}

Recall from Section~\ref{sec:Prelims} that it is sufficient to show $|B(I)|>|\PSL(2,O_d/I)|$ to rule out $M=\H^3/\Gamma(I)$ as a link complement. The same section also contained Equation~\ref{eqn:PslSizeEquation} allowing us to compute $|\PSL(2,O_d/I)|$. To obtain $|B(I)|$ or a lower bound for it, we use Magma's builtin \texttt{Order} or one of the functions in Table~\ref{tbl:magmaHelpers}. This works for all cases except the special ones treated in Section~\ref{sec:casesInsuffHom}.

\begin{table}[h]
\caption{Magma functions to give a lower bound for the size of a group.\label{tbl:magmaHelpers}}
\begin{center}
\begin{minipage}{14cm}
\begin{verbatim}
function LowerBound1(G, index)
    L := LowIndexNormalSubgroups(G, index);
    m := Max([ Index(G,H`Group) * Order(AbelianQuotient(H`Group)) : H in L]); 
    return m;
end function;

function LowerBound2(G, index, generator)
    g := G.generator;
    G := ReduceGenerators(G);       // for optimization
    L := LowIndexNormalSubgroups(G, index);
    for N in L do
        if Index(G,N`Group) eq index then
            H := sub<G | N`Group, g>;
            return Index(G, H) * Order(AbelianQuotient(H));
        end if;
    end for;
end function;
\end{verbatim}
\end{minipage}
\end{center}
\end{table}

Given a subgroup $H$ of a given group $G$, a lower bound on $|G|$ is given by the product of the index $[G:H]$ and the order of $H$'s Abelianization. \texttt{LowerBound1} returns the best lower bound obtained this way when considering all normal subgroups $H$ up to a given index. \texttt{LowerBound2} considers the subgroup $H$ generated by $g$ and the normal subgroup of given index that Magma found first\footnote{The order in which Magma's \texttt{LowIndexNormalSubgroups} returns the groups is not guaranteed to be deterministic or stable between different Magma versions. However, in all our invocations of \texttt{LowerBound2}, there is exactly one normal subgroup of the given index, so our results are reproducable.}. Here, $g$ is a generator of $G$ specified by its index.

We give some examples how the use of these functions is encoded in the following tables:
\begin{itemize}
\item For $(7,\myLangle 3\myRangle)$, the Magma code \texttt{Order(B7(3,0,3))} shows that $|B(I)|=1080$ which is larger than $|\PSL(2,O_d/I)|=360$, so $\H^3/\Gamma(I)$ is not a link complement. See Table~\ref{tbl:results7}.
 \item For $(1, \myLangle 4+2\sqrt{-1}\myRangle)$, the Magma code \texttt{LowerBound1(B1(10,4,2), 2880)} returns 92160 as lower bound for $|B(I)|$ which is larger than $|\PSL(2,O_d/I)|=2880$, so $\H^3/\Gamma(I)$ is not a link complement either. If the index given to $\texttt{LowerBound1}$ is $|\PSL(2,O_d/I)|$, we do not specify the arguments to the function in the table. See Table~\ref{tbl:results1}. 
 \item For $(15,\myLangle 5\myRangle)$, the Magma code \texttt{LowerBound1(B5(5,0,5, 5,0,5), 5)} returned \texttt{Infinity} showing this case not to be a link complement. See Table~\ref{tbl:results15}.
 \item For $(15,\myLangle 4 + \sqrt{-15}\myRangle)$, the Magma code \texttt{LowerBound2(B5(31,-14,1, 31,13,1), 14880, 2)} was used. The group $H$ used was generated the normal subgroup of index 14880 and the second generator of $B(I)$, namely $t$. See Table~\ref{tbl:results15}.
 \item For $(1,\myLangle 4+4\sqrt{-1}\myRangle)$, we use that $|B(J)|\geq |B(I)|$ if $J\subset I$. See Table~\ref{tbl:results1}.
 \end{itemize}

For each discriminant, we provide a Magma file to check all the necessary cases at \cite[\href{http://unhyperbolic.org/prinCong/prinCong/magma/}{\texttt{prinCong/magma/}}]{goerner:data}. For example, the following shell command
\begin{center}
\begin{minipage}{14cm}
\begin{verbatim}
magma GroupB.m NotLinkComplementHelpers.m NotLinkComplement7.m
\end{verbatim}
\end{minipage}
\end{center}
or Magma commands
\begin{center}
\begin{minipage}{14cm}
\begin{verbatim}
load "GroupB.m";
load "NotLinkComplementHelpers.m";
load "NotLinkComplement7.m";
\end{verbatim}
\end{minipage}
\end{center}
replay the necessary computations for $d=7$:
\begin{center}
\begin{minipage}{14cm}
\begin{verbatim}
> Order(B7(3,0,3));
1080
> LowerBound1(B7(8,2,2), 1152);
Infinity
> LowerBound1(B7(4,0,4), 1152);
Infinity
> LowerBound1(B7(6,0,3), 2160);
836828256240
...
\end{verbatim}
\end{minipage}
\end{center}

\clearpage

\subsection{Class number one} \label{sec:tablesClassNumberOne}
Note that $|B(I)|=|\PSL(2,O_d/I)|$ for $(1, \myLangle1+\sqrt{-1}\myRangle)$, $(2, \myLangle \sqrt{-2}\myRangle)$, and $(3,\myLangle\frac{3+\sqrt{-3}}{2}\myRangle)$, but $\Gamma(I)$ contains torsion elements so we obtain an orbifold.

\begin{table}[h]
\caption{Magma function to produce $B(I)$ for $d=1$.\label{tbl:magma1}}
\begin{center}
\begin{minipage}{12cm}
\begin{verbatim}
function B1(n1,k1,l1)
    return Group<a,l,t,u|a^2,l^2,(t*l)^2,(u*l)^2,(a*l)^2,
                         (t*a)^3,(u*a*l)^3,(t,u),
    
                         t^n1,t^k1*u^l1>;
end function;
\end{verbatim}
\end{minipage}
\end{center}
\end{table}

\begin{table}[h]
\caption{Computations for $d=1$ (compare to Figure~\ref{fig:disc1}).\label{tbl:results1}}
\begin{center}
\begin{tabular}{| c || r |r|r||c|c|c|}
\hline
Ideal & \multicolumn{1}{c|}{$n$} & \multicolumn{1}{c|}{$k$} & \multicolumn{1}{c||}{$l$} & Method & $|B(I)|$ & $|\PSL(2,O_d/I)|$ \\ \hline \hline
$\myLangle 1+\sqrtMinusDSymbolTable{1}\myRangle$ &
2 & 1 & 1 &
Orbifold & 6 &
6 \\ \hline
$\myLangle 2\myRangle$ &
2 & 0 & 2 &
\multicolumn{2}{c|}{6-Link} &
48 \\ \hline
$\myLangle 2+1 \sqrtMinusDSymbolTable{1}\myRangle$ &
5 & 2 & 1 &
\multicolumn{2}{c|}{6-Link} & 
60 \\ \hline
$\myLangle 2+2 \sqrtMinusDSymbolTable{1}\myRangle$ &
4 & 2 & 2 &
\multicolumn{2}{c|}{12-Link} &
192 \\ \hline
$\myLangle 3\myRangle$ &
3 & 0 & 3 &
\multicolumn{2}{c|}{20-Link} &
360 \\ \hline
$\myLangle 3+1 \sqrtMinusDSymbolTable{1}\myRangle$ &
10 & 3 & 1 &
\multicolumn{2}{c|}{18-Link} &
360 \\ \hline
$\myLangle 3+2 \sqrtMinusDSymbolTable{1}\myRangle$ &
13 & -5 & 1 &
\multicolumn{2}{c|}{42-Link} &
1092 \\ \hline
$\myLangle 4 \myRangle$ &
4 & 0 & 4 &
\texttt{LowerBound1} & $\infty$ &
1536 \\ \hline
$\myLangle 3+3 \sqrtMinusDSymbolTable{1}\myRangle$ &
6 & 3 & 3 &
\texttt{LowerBound1} & $\infty$ &
2160 \\ \hline
$\myLangle 4+1 \sqrtMinusDSymbolTable{1}\myRangle$ &
17 & 4 & 1 &
\multicolumn{2}{c|}{72-Link} &
2448 \\ \hline
$\myLangle 4+2 \sqrtMinusDSymbolTable{1}\myRangle$ &
10 & 4 & 2 &
\texttt{LowerBound1} & $\geq 92160$ &
2880 \\ \hline
$\myLangle 5+1 \sqrtMinusDSymbolTable{1}\myRangle$ &
26 & 5 & 1 &
\texttt{LowerBound1} & $\geq 107347968$ &
6552 \\ \hline
$\myLangle 5 \myRangle$ &
5 & 0 & 5 &
\texttt{LowerBound1} & $\infty$ &
7200 \\ \hline
$\myLangle 4+3 \sqrtMinusDSymbolTable{1}\myRangle$ &
 &  &  &
Lemma~\ref{lemma:specialCases} & $\infty$ &
7500 \\ \hline
$\myLangle 5+2 \sqrtMinusDSymbolTable{1}\myRangle$ &
29 & 17 & 1 &
\texttt{LowerBound1} & $\infty$ &
12180 \\ \hline
$\myLangle 4+4 \sqrtMinusDSymbolTable{1}\myRangle$ &
 &  &  &
$\containedInI\myLangle 4\myRangle$ & $\infty$ &
12288 \\ \hline
$\myLangle 5+3 \sqrtMinusDSymbolTable{1}\myRangle$ &
34 & 13 & 1 &
\texttt{LowerBound1} & $\infty$ &
14688 \\ \hline
\end{tabular}
\end{center}
\end{table}

\clearpage

\begin{table}[h]
\caption{Magma function to produce $B(I)$ for $d=2$.\label{tbl:magma2}}
\begin{center}
\begin{minipage}{12cm}
\begin{verbatim}
function B2(n,k,l)
    return Group<a,t,u | a^2, (t*a)^3, (a*u^-1*a*u)^2, (t,u),
    
                         t^n, t^k*u^l>;
end function;
\end{verbatim}
\end{minipage}
\end{center}
\end{table}

\begin{table}[h]
\caption{Computations for $d=2$ (compare to Figure~\ref{fig:disc2}).\label{tbl:results2}}
\begin{center}
\begin{tabular}{| c || r |r|r||c|c|c|}
\hline
Ideal & \multicolumn{1}{c|}{$n$} & \multicolumn{1}{c|}{$k$} & \multicolumn{1}{c||}{$l$} & Method & $|B(I)|$ & $|\PSL(2,O_d/I)|$ \\ \hline \hline
$\myLangle \sqrt{-2}\myRangle$ & 2 & 0 & 1 & Orbifold & 6 & 6 \\ \hline
$\myLangle 1+\sqrt{-2}\myRangle$ & 3 & 1 & 1 & \multicolumn{2}{c|}{4-Link} & 12 \\ \hline
$\myLangle 2\myRangle$ & 2 & 0 & 2 & \multicolumn{2}{c|}{12-Link} & 48 \\ \hline
$\myLangle 2+\sqrt{-2}\myRangle$ & 6 & 2 & 1 & \multicolumn{2}{c|}{12-Link} & 72 \\ \hline
$\myLangle 2\sqrt{-2}\myRangle$ & 4 & 0 & 2 & \texttt{LowerBound1} & $\infty$ & 192 \\ \hline
$\myLangle 3\myRangle$ & 3 & 0 & 3 & \texttt{LowerBound1} & $\infty$ & 288 \\ \hline
$\myLangle 1+2\sqrt{-2}\myRangle$ & 9 & -4 & 1 & \multicolumn{2}{c|}{36-Link} & 324 \\ \hline
$\myLangle 2+2\sqrt{-2}\myRangle$ & 6 & 2 & 2 & \texttt{LowerBound1}& $\geq 18432 $ & 576 \\ \hline
$\myLangle 3+\sqrt{-2}\myRangle$ & 11 & 3 & 1 & \multicolumn{2}{c|}{60-Link} & 660 \\ \hline
$\myLangle 4\myRangle$ & 4 & 0 & 4 & \texttt{LowerBound1} & $\infty$ & 1536\\ \hline
$\myLangle 3\sqrt{-2}\myRangle$ & & & & $\containedInI\myLangle 3\myRangle$ & $\infty$ & 1728 \\ \hline
$\myLangle 4+\sqrt{-2}\myRangle$ & 18 & 4 & 1 & \texttt{LowerBound1} & $\infty$ & 1944 \\ \hline
$\myLangle 3+2\sqrt{-2}\myRangle$ & 17 & -7 & 1 & \texttt{LowerBound1} & $\geq 1253376$ & 2448 \\ \hline
$\myLangle 1+3\sqrt{-2}\myRangle$ & 19 & 6 & -1 & Lemma~\ref{lemma:nasty} & $\infty$ & 3420 \\ \hline
$\myLangle 2+3\sqrt{-2}\myRangle$ & 22 & 8 & 1 & \texttt{LowerBound1} & $\infty$ & 3960 \\ \hline
$\myLangle 4+2\sqrt{-2}\myRangle$ & &&& $\containedInI\myLangle 2-2\sqrt{-2}\myRangle$ & $\geq 18432$ & 4608 \\ \hline
$\myLangle 3+3\sqrt{-2}\myRangle$ & &&& $\containedInI\myLangle 3\myRangle$ & $\infty$ & 7776 \\ \hline
$\myLangle 5\myRangle$ & 5 & 0 & 5 & \texttt{LowerBound1} & $\infty$ & 7800 \\ \hline
$\myLangle 5+\sqrt{-2}\myRangle$ & 27 & 5 & 1 & \texttt{LowerBound1} & $\geq 3.462\cdot 10^{40}$ & 8748 \\ \hline
$\myLangle 4\sqrt{-2}\myRangle$ & &&& $\containedInI\myLangle 2\sqrt{-2}\myRangle$ & $\infty$ & 12288 \\ \hline
$\myLangle 4+3\sqrt{-2}\myRangle$ & &&& $\containedInI\myLangle 3-2\sqrt{-2} \myRangle$ & $\geq 1253376$ & 14688 \\ \hline
$\myLangle 1+4\sqrt{-2}\myRangle$ & 33 & -8 & 1 & \texttt{LowerBound1} & $\infty$ & 15840 \\ \hline
$\myLangle 5+2\sqrt{-2}\myRangle$ & 33 & -14 & 1  & \texttt{LowerBound1} & $\infty$ & 15840 \\ \hline
\end{tabular}
\end{center}
\end{table}

\clearpage

\begin{table}[h]
\caption{Magma function to produce $B(I)$ for $d=3$.\label{tbl:magma3}}
\begin{center}
\begin{minipage}{12cm}
\begin{verbatim}
function B3(n1,k1,l1)
    return Group<a,l,t,u|(t,u),a^2,(a*l)^2,(t*a)^3,l^3,(u*a*l)^3,
                         l^-1*t*l*u*t,l^-1*u*l*t^-1,
                         
                         t^n1,t^k1*u^l1>;
end function;
\end{verbatim}
\end{minipage}
\end{center}
\end{table}

For convenience, we remind the reader that in the case of $d=3$, we use $\omega_3^2 = (-1+\sqrt{-3})/2$ in the matrix for $u$.

\begin{table}[h]
\caption{Computations for $d=3$ (compare to Figure~\ref{fig:disc3}).\label{tbl:results3}}
\begin{center}
\begin{tabular}{| c || r |r|r||c|c|c|}
\hline
Ideal & \multicolumn{1}{c|}{$n$} & \multicolumn{1}{c|}{$k$} & \multicolumn{1}{c||}{$l$} & Method & $|B(I)|$ & $|\PSL(2,O_d/I)|$ \\ \hline \hline
$\myLangle \frac{3+\sqrt{-3}}{2}\myRangle$ &
3 & -1 & 1 &
Orbifold & 12 &
12 \\ \hline
$\myLangle 2\myRangle$ &
2 & 0 & 2 &
\multicolumn{2}{c|}{5-Link} &
60 \\ \hline
$\myLangle \frac{5+ \sqrtMinusDSymbolTable{3}}{2}\myRangle$ &
7 & 3 & 1 &
\multicolumn{2}{c|}{8-Link} &
168 \\ \hline
$\myLangle 3\myRangle$ &
3 & 0 & 3 &
\multicolumn{2}{c|}{12-Link} &
324 \\ \hline
$\myLangle 3 + \sqrtMinusDSymbolTable{3}\myRangle$ &
6 & 4 & 2 &
\multicolumn{2}{c|}{20-Link} &
720 \\ \hline
$\myLangle \frac{7+\sqrt{-3}}{2}\myRangle$ &
13 & 4 & 1 &
\multicolumn{2}{c|}{28-Link} &
1092 \\ \hline
$\myLangle 4\myRangle$ &
4 & 0 & 4 &
\texttt{Order} & 3840 & 
1920 \\ \hline
$\myLangle 4+\sqrt{-3} \myRangle$ &
19 & 12 & 1 &
\multicolumn{2}{c|}{60-Link} &
3420 \\ \hline
$\myLangle \frac{9+ \sqrtMinusDSymbolTable{3}}{2}\myRangle$ &
21 & 5 & 1 &
\multicolumn{2}{c|}{64-Link} &
4032 \\ \hline
$\myLangle 5\myRangle$ &
5 & 0 & 5 &
\texttt{LowerBound1} & $\infty$ &
7800 \\ \hline
$\myLangle \frac{9+3\sqrt{-3}}{2}\myRangle$ &
9 & 6 & 3 &
\texttt{LowerBound1} & $\infty$ &
8748 \\ \hline
$\myLangle 5+ \sqrtMinusDSymbolTable{3}\myRangle$ &
14 & 6 & 2 &
\texttt{LowerBound1} & $\geq 2580480$ &
10080 \\ \hline
$\myLangle \frac{11+\sqrt{-3}}{2}\myRangle$ &
 &  &  &
Lemma~\ref{lemma:specialCases} & $\infty$ &
14880 \\ \hline
\end{tabular}
\end{center}
\end{table}

\clearpage

\begin{table}[h]
\caption{Magma function to produce $B(I)$ for $d=7$.\label{tbl:magma7}}
\begin{center}
\begin{minipage}{12cm}
\begin{verbatim}
function B7(n,k,l)
    return Group<a,t,u|a^2,(t*a)^3,(a*t*u^-1*a*u)^2,(t,u),
    
                       t^n,t^k*u^l>;
end function;
\end{verbatim}
\end{minipage}
\end{center}
\end{table}

\begin{table}[h]
\caption{Computations for $d=7$ (compare to Figure~\ref{fig:disc7}).\label{tbl:results7}}
\begin{center}
\begin{tabular}{| c || r |r|r||c|c|c|}
\hline
Ideal & \multicolumn{1}{c|}{$n$} & \multicolumn{1}{c|}{$k$} & \multicolumn{1}{c||}{$l$} & Method & $|B(I)|$ & $|\PSL(2,O_d/I)|$ \\ \hline \hline
$\myLangle \frac{1+\sqrtMinusDSymbolTable{7}}{2}\myRangle$ &
2 & 0 & 1 &
\multicolumn{2}{c|}{3-Link} &
6 \\ \hline
$\myLangle \frac{3+\sqrtMinusDSymbolTable{7}}{2}\myRangle$ &
4 & 1 & 1 &
\multicolumn{2}{c|}{6-Link} &
24 \\ \hline
$\myLangle 2\myRangle$ &
2 & 0 & 2 &
\multicolumn{2}{c|}{9-Link} &
36 \\ \hline
$\myLangle 1+\sqrtMinusDSymbolTable{7}\myRangle$ &
4 & 0 & 2 &
\multicolumn{2}{c|}{18-Link} &
144 \\ \hline
$\myLangle \sqrtMinusDSymbolTable{7}\myRangle$ &
7 & 3 & 1 &
\multicolumn{2}{c|}{24-Link} &
168 \\ \hline
$\myLangle \frac{5+\sqrtMinusDSymbolTable{7}}{2}\myRangle$ &
8 & 2 & 1 &
\multicolumn{2}{c|}{24-Link} &
192 \\ \hline
$\myLangle 3\myRangle$ &
3 & 0 & 3 &
\texttt{Order} & 1080 &
360 \\ \hline
$\myLangle 2+\sqrtMinusDSymbolTable{7}\myRangle$ &
11 & 5 & -1 &
\multicolumn{2}{c|}{60-Link} &
660 \\ \hline
$\myLangle \frac{7+\sqrtMinusDSymbolTable{7}}{2}\myRangle$ &
14 & 3 & 1 &
\multicolumn{2}{c|}{72-Link} &
1008 \\ \hline
$\myLangle 3+\sqrtMinusDSymbolTable{7}\myRangle$ &
8 & 2 & 2 &
\texttt{LowerBound1} & $\infty$ &
1152 \\ \hline
$\myLangle 4\myRangle$ &
4 & 0 & 4 &
\texttt{LowerBound1} & $\infty$ &
1152 \\ \hline
$\myLangle \frac{1+3 \sqrtMinusDSymbolTable{7}}{2}\myRangle$ &
16 & 5 & 1 &
\multicolumn{2}{c|}{96-Link} &
1536 \\ \hline
$\myLangle \frac{3+3 \sqrtMinusDSymbolTable{7}}{2}\myRangle$ &
6 & 0 & 3 &
\texttt{LowerBound1} & $\geq 8.368\cdot 10^{11}$ &
2160 \\ \hline
$\myLangle \frac{5+3 \sqrtMinusDSymbolTable{7}}{2}\myRangle$ &
22 & -7 & 1 &
\texttt{LowerBound1} & $\infty$ &
3960 \\ \hline
$\myLangle \frac{9+\sqrtMinusDSymbolTable{7}}{2}\myRangle$ &
22 & 4 & 1 &
\texttt{LowerBound1} & $\infty$ &
3960 \\ \hline
$\myLangle 2 \sqrtMinusDSymbolTable{7}\myRangle$ &
14 & 6 & 2 &
\texttt{LowerBound1} & $\infty$ &
6048 \\ \hline
$\myLangle 4+\sqrtMinusDSymbolTable{7}\myRangle$ &
23 & 13 & 1 &
\texttt{LowerBound1} & $\geq 2.546 \cdot 10^{10}$ &
6072 \\ \hline
$\myLangle 5\myRangle$ &
5 & 0 & 5 &
\texttt{LowerBound1} & $\infty$ &
7800 \\ \hline
$\myLangle \frac{7+3 \sqrtMinusDSymbolTable{7}}{2}\myRangle$ &
28 & 10 & 1 &
\texttt{LowerBound1} & $\infty$ &
8064 \\ \hline
$\myLangle 2+2 \sqrtMinusDSymbolTable{7}\myRangle$ &
 &  &  &
$\containedInI\myLangle 4\myRangle$ & $\infty$ &
9216 \\ \hline
$\myLangle 5+\sqrtMinusDSymbolTable{7}\myRangle$ &
 & & &
$\containedInI\myLangle 3-\sqrtMinusDSymbolTable{7}\myRangle$ & $\infty$ &
9216 \\ \hline
$\myLangle 1+2 \sqrtMinusDSymbolTable{7}\myRangle$ &
29 & 7 & 1 &
\texttt{LowerBound1} & $\geq 3.717 \cdot 10^{14}$ &
12180 \\ \hline
$\myLangle \frac{11+\sqrtMinusDSymbolTable{7}}{2}\myRangle$ &
32 & 5 & 1 &
\texttt{LowerBound1} & $\geq 3.494 \cdot 10^{49}$ &
12288 \\ \hline
\end{tabular}
\end{center}
\end{table}

\clearpage

\begin{table}[h]
\caption{Magma function to produce $B(I)$ for $d=11$.\label{tbl:magma11}}
\begin{center}
\begin{minipage}{12cm}
\begin{verbatim}
function B11(n,k,l)
    return Group<a,t,u|a^2,(t*a)^3,(a*t*u^-1*a*u)^3,(t,u),
    
                       t^n,t^k*u^l>;
end function;
\end{verbatim}
\end{minipage}
\end{center}
\end{table}

\begin{table}[h]
\caption{Computations for $d=11$ (compare to Figure~\ref{fig:disc11}).\label{tbl:results11}}
\begin{center}
\begin{tabular}{| c || r |r|r||c|c|c|}
\hline
Ideal & \multicolumn{1}{c|}{$n$} & \multicolumn{1}{c|}{$k$} & \multicolumn{1}{c||}{$l$} & Method & $|B(I)|$ & $|\PSL(2,O_d/I)|$ \\ \hline \hline
$\myLangle \frac{1+\sqrtMinusDSymbolTable{11}}{2}\myRangle$ &
3 & 0 & 1 &
\multicolumn{2}{c|}{4-Link} &
12 \\ \hline
$\myLangle 2\myRangle$ &
2 & 0 & 2 &
\texttt{Order} & 120 &
60 \\ \hline
$\myLangle \frac{3+\sqrtMinusDSymbolTable{11}}{2}\myRangle$ &
5 & 1 & 1 &
\multicolumn{2}{c|}{12-Link} &
60 \\ \hline
$\myLangle 3\myRangle$ &
3 & 0 & 3 &
\texttt{LowerBound1} & $\infty$ &
288 \\ \hline
$\myLangle \frac{5+\sqrtMinusDSymbolTable{11}}{2}\myRangle$ &
9 & 2 & 1 &
\multicolumn{2}{c|}{36-Link} &
324 \\ \hline
$\myLangle \sqrtMinusDSymbolTable{11}\myRangle$ &
11 & 5 & 1 &
\texttt{LowerBound1} & $\infty$ &
660 \\ \hline
$\myLangle 1+\sqrtMinusDSymbolTable{11}\myRangle$ &
6 & 0 & 2 &
\texttt{LowerBound1} & $\geq 248832$ &
720 \\ \hline
$\myLangle 2+\sqrtMinusDSymbolTable{11}\myRangle$ &
15 & 8 & 1 &
\texttt{LowerBound1} & $\geq 5.446 \cdot 10^{13}$ &
1440 \\ \hline
$\myLangle \frac{7+\sqrtMinusDSymbolTable{11}}{2}\myRangle$ &
15 & 3 & 1 &
\texttt{LowerBound1} & $\infty$ &
1440 \\ \hline
$\myLangle 4\myRangle$ &
4 & 0 & 4 &
\texttt{LowerBound1} & $\infty$ &
1920 \\ \hline
$\myLangle 3+\sqrtMinusDSymbolTable{11}\myRangle$ &
10 & 2 & 2 &
\texttt{LowerBound1} & $\infty$ &
3600 \\ \hline
$\myLangle \frac{9+\sqrtMinusDSymbolTable{11}}{2}\myRangle$ &
23 & 4 & 1 &
\texttt{LowerBound1} & $\geq 3.152 \cdot 10^{37}$ &
6072 \\ \hline
$\myLangle 5\myRangle$ &
5 & 0 & 5 &
\texttt{LowerBound1} & $\infty$ &
7200 \\ \hline
$\myLangle \frac{1+3 \sqrtMinusDSymbolTable{11}}{2}\myRangle$ &
25 & 8 & 1 &
\texttt{LowerBound1} & $\geq 1.418 \cdot 10^{48}$ &
7500 \\ \hline
$\myLangle \frac{3+3 \sqrtMinusDSymbolTable{11}}{2}\myRangle$ &
 &  &  &
$\containedInI\myLangle 3\myRangle$ & $\infty$ &
7776 \\ \hline
$\myLangle 4+\sqrtMinusDSymbolTable{11}\myRangle$ &
27 & 12 & -1 &
\texttt{LowerBound1} & $\geq 4.117 \cdot 10^{76}$ &
8748 \\ \hline
$\myLangle \frac{5+3 \sqrtMinusDSymbolTable{11}}{2}\myRangle$ &
31 & -10 & 1 &
\texttt{LowerBound1} & $\geq 3.785 \cdot 10^{145}$ &
14880 \\ \hline
$\myLangle \frac{11+\sqrtMinusDSymbolTable{11}}{2}\myRangle$ &
 &  &  &
$\containedInI\myLangle \sqrtMinusDSymbolTable{11}\myRangle$ & $\infty$ &
15840 \\ \hline
\end{tabular}
\end{center}
\end{table}

\clearpage

\begin{table}[h]
\caption{Magma function to produce $B(I)$ for $d=19$.\label{tbl:magma19}}
\begin{center}
\begin{minipage}{12cm}
\begin{verbatim}
function B19(n,k,l)
    return Group<a,b,t,u|a^2,(t*a)^3,b^3,(b*t^-1)^3,(a*b)^2,
                         (a*t^-1*u*b*u^-1)^2,(t,u),
                         
                         t^n,t^k*u^l>;
end function;
\end{verbatim}
\end{minipage}
\end{center}
\end{table}

\begin{table}[h]
\caption{Computations for $d=19$ (compare to Figure~\ref{fig:disc19}).\label{tbl:results19}}
\begin{center}
\begin{tabular}{| c || r |r|r||c|c|c|}
\hline
Ideal & \multicolumn{1}{c|}{$n$} & \multicolumn{1}{c|}{$k$} & \multicolumn{1}{c||}{$l$} & Method & $|B(I)|$ & $|\PSL(2,O_d/I)|$ \\ \hline \hline
$\myLangle 2\myRangle$ &
2 & 0 & 2 &
\texttt{LowerBound1} & $\infty$ &
60 \\ \hline
$\myLangle \frac{1+\sqrtMinusDSymbolTable{19}}{2}\myRangle$ &
5 & 0 & 1 &
\multicolumn{2}{c|}{12-Link} &
60 \\ \hline
$\myLangle \frac{3+\sqrtMinusDSymbolTable{19}}{2}\myRangle$ &
7 & 1 & 1 &
\texttt{LowerBound1} & $\geq 43008$ &
168 \\ \hline
$\myLangle 3\myRangle$ &
3 & 0 & 3 &
\texttt{LowerBound1} & $\infty$ &
360 \\ \hline
$\myLangle \frac{5+\sqrtMinusDSymbolTable{19}}{2}\myRangle$ &
11 & 2 & 1 &
\texttt{LowerBound1} & $\geq 1.305 \cdot 10^{12}$ &
660 \\ \hline
$\myLangle 4\myRangle$ &
 &  &  &
$\containedInI\myLangle 2\myRangle$ & $\infty$ &
1920 \\ \hline
$\myLangle \frac{7+\sqrtMinusDSymbolTable{19}}{2}\myRangle$ &
17 & 3 & 1 &
\texttt{LowerBound1} & $\infty$ &
2448 \\ \hline
$\myLangle \sqrtMinusDSymbolTable{19}\myRangle$ &
19 & -10 & 1 &
\texttt{LowerBound1} & $\infty$ &
3420 \\ \hline
$\myLangle 1+\sqrtMinusDSymbolTable{19}\myRangle$ &
 &  &  &
$\containedInI\myLangle 2\myRangle$ & $\infty$ &
3600 \\ \hline
$\myLangle 2+\sqrtMinusDSymbolTable{19}\myRangle$ &
23 & -11 & 1 &
\texttt{LowerBound1} & $\infty$ &
6072 \\ \hline
$\myLangle 5\myRangle$ &
5 & 0 & 5 &
\texttt{LowerBound1(B19(5,0,5), 300)} & $\infty$ &
7200 \\ \hline
$\myLangle \frac{9+\sqrtMinusDSymbolTable{19}}{2}\myRangle$ &
25 & 4 & 1 &
\texttt{LowerBound1} & $\infty$ &
7500 \\ \hline
$\myLangle 3+\sqrtMinusDSymbolTable{19}\myRangle$ &
 &  &  &
$\containedInI\myLangle 2\myRangle$ & $\infty$ &
10080 \\ \hline
$\myLangle 4+\sqrtMinusDSymbolTable{19}\myRangle$ &
 &  &  &
$\containedInI\myLangle \frac{3-\sqrtMinusDSymbolTable{19}}{2}\myRangle$ & $\geq 43008$ &
20160 \\ \hline
$\myLangle \frac{11+\sqrtMinusDSymbolTable{19}}{2}\myRangle$ &
 &  &  &
$\containedInI\myLangle \frac{3-\sqrtMinusDSymbolTable{19}}{2}\myRangle$ & $\geq 43008$ &
20160 \\ \hline
\end{tabular}
\end{center}
\end{table}

\clearpage

\subsection{Higher class numbers} Note that the orbifold cases
$(5,\myLangle 2, 1+\sqrt{-5}\myRangle)$, $(6, \myLangle 2, \sqrt{-6}\myRangle)$, and $(39, \myLangle 3, \frac{3+\sqrt{-39}}{2}\myRangle)$ are already ruled out as link complements by the fact that $|B(I)|>|\PSL(2,O_d/I)|$. This is in contrast to the orbifold cases with $h_d=1$.

\mbox{} 

\begin{table}[h]
\caption{Magma function to produce $B(I)$ for $d=5$.\label{tbl:magma5}}
\begin{center}
\begin{minipage}{12cm}
\begin{verbatim}
function B5(n,k,l)
    return Group<a,b,c,t,u | (t,u), a^2, b^2, (t*a)^3, (a*b)^2,
                             (a*u*b*u^-1)^2, a*c*a*t*c^-1*t^-1,
                             u*b*u^-1*c*b*t*c^-1*t^-1,
                             
                             t^n, t^k*u^l,
                             (t*b)^n,(t*b)^k*(t*u^-1*c*t^-1)^l>;
end function;
\end{verbatim}
\end{minipage}
\end{center}
\end{table}

\begin{table}[h]
\caption{Computations for $d=5$ (compare to Figure~\ref{fig:disc5}).\label{tbl:results5}}
\begin{center}
\begin{tabular}{| c || r |r|r||c|c|c|}
\hline
Ideal & \multicolumn{1}{c|}{$n$} & \multicolumn{1}{c|}{$k$} & \multicolumn{1}{c||}{$l$} & Method & $|B(I)|$ & $|\PSL(2,O_d/I)|$ \\ \hline \hline
$\myLangle 2,1+\sqrtMinusDSymbolTable{5}\myRangle$ &
2 & 1 & 1 &
\texttt{Order} & 12 &
6 \\
&&&& (Orbifold) & & \\ \hline
$\myLangle 3,1+\sqrtMinusDSymbolTable{5}\myRangle$ &
3 & 1 & 1 &
\multicolumn{2}{c|}{8-Link} &
12 \\ \hline
$\myLangle 2\myRangle$ &
2 & 0 & 2 &
\texttt{LowerBound1} & $\infty$ &
48 \\ \hline
$\myLangle \sqrtMinusDSymbolTable{5}\myRangle$ &
5 & 0 & 1 &
\texttt{LowerBound1} & $\infty$ &
60 \\ \hline
$\myLangle 1+\sqrtMinusDSymbolTable{5}\myRangle$ &
6 & 1 & 1 &
\texttt{Order} & 144 &
72 \\ \hline
$\myLangle 7,3+\sqrtMinusDSymbolTable{5}\myRangle = $ &
7 & 3 & 1 &
\texttt{LowerBound1} & $\geq 122472$ &
168 \\
$\myLangle 3+\sqrtMinusDSymbolTable{5}, 1-2 \sqrtMinusDSymbolTable{5}\myRangle$  & & & & & &   \\ \hline
$\myLangle 4,2+2 \sqrtMinusDSymbolTable{5}\myRangle$ &
&&&
$\containedInI\myLangle 2\myRangle$ & $\infty$ &
192 \\ \hline
$\myLangle 3\myRangle$ &
3 & 0 & 3 &
\texttt{LowerBound1} & $\infty$ &
288 \\ \hline
\markConjugate{$\myLangle 2-\sqrtMinusDSymbolTable{5}\myRangle$} &
9 & -2 & 1 &
\texttt{LowerBound1} & $\geq 1327104$ &
324 \\ \hline
$\myLangle 10,5+\sqrtMinusDSymbolTable{5}\myRangle = $ &
&&& $\containedInI\myLangle \sqrtMinusDSymbolTable{5}\myRangle $ & $\infty$ &
360 \\
$\myLangle 2 \sqrtMinusDSymbolTable{5}, 5+\sqrtMinusDSymbolTable{5}\myRangle$  &  & & & & &\\ \hline
$\myLangle 6,2+2 \sqrtMinusDSymbolTable{5}\myRangle$ &
&&&$\containedInI\myLangle 2\myRangle$ & $\infty$ &
576 \\ \hline
$\myLangle 3+\sqrtMinusDSymbolTable{5}\myRangle$ &
&&&$\containedInI\myLangle 7, 3+\sqrtMinusDSymbolTable{5}\myRangle$ & $\geq 122472$ &
1008 \\ \hline
$\myLangle 15,5+\sqrtMinusDSymbolTable{5}\myRangle = $ &
&&&$\containedInI\myLangle \sqrtMinusDSymbolTable{5}\myRangle $ & $\infty$ &
1440 \\
$\myLangle 5+\sqrtMinusDSymbolTable{5}, 3 \sqrtMinusDSymbolTable{5}\myRangle$  & &&& & &  \\ \hline
$\myLangle 4\myRangle$ &
&&&$\containedInI\myLangle 2\myRangle$ & $\infty$ &
1536 \\ \hline
$\myLangle 6,3+3 \sqrtMinusDSymbolTable{5}\myRangle$ &
&&&$\containedInI\myLangle 3\myRangle$ & $\infty$ &
1728 \\ \hline
$\myLangle 18,11+\sqrtMinusDSymbolTable{5}\myRangle = $ &
&&&$\containedInI\myLangle 2+\sqrtMinusDSymbolTable{5}\myRangle$ & $\geq 1327104$ &
1944 \\
$\myLangle 4+2 \sqrtMinusDSymbolTable{5}, 3-3 \sqrtMinusDSymbolTable{5}\myRangle$  & &&& & &  \\ \hline
$\myLangle 2 \sqrtMinusDSymbolTable{5}\myRangle$ &
&&&$\containedInI\myLangle 2\myRangle$ & $\infty$ &
2880 \\ \hline
$\myLangle 1+2 \sqrtMinusDSymbolTable{5}\myRangle$ &
&&&$\containedInI\myLangle 7,3-\sqrtMinusDSymbolTable{5}\myRangle$ & $\geq 122472$ &
4032 \\ \hline
$\myLangle 4+\sqrtMinusDSymbolTable{5}\myRangle$ &
&&&$\containedInI\myLangle 7,3-\sqrtMinusDSymbolTable{5}\myRangle$ & $\geq 122472$ &
4032 \\ \hline
$\myLangle 2+2 \sqrtMinusDSymbolTable{5}\myRangle$ &
&&&$\containedInI\myLangle 2\myRangle$ & $\infty$ &
4608 \\ \hline
$\myLangle 5\myRangle$ &
&&&$\containedInI\myLangle \sqrtMinusDSymbolTable{5}\myRangle $ & $\infty$ &
7500 \\ \hline
$\myLangle 5+\sqrtMinusDSymbolTable{5}\myRangle$ &
&&&$\containedInI\myLangle \sqrtMinusDSymbolTable{5}\myRangle $ & $\infty$ &
8640 \\ \hline
$\myLangle 3+2 \sqrtMinusDSymbolTable{5}\myRangle$ &
29 & -13 & 1 & \!\!\texttt{LowerBound2(B5(29,-13,1),}\!\! & $\infty$ &
12180 \\
&&&& \!\!\texttt{\phantom{LowerBound2(}12180,4)\phantom{LLLLL}}\!\! & & \\
 \hline
$\myLangle 6\myRangle$ &
&&&$\containedInI\myLangle 2\myRangle$ & $\infty$ &
13824 \\ \hline
$\myLangle 4+2 \sqrtMinusDSymbolTable{5}\myRangle$ &
&&&$\containedInI\myLangle 2\myRangle$ & $\infty$ &
15552 \\ \hline
\end{tabular}
\end{center}
\explainConjugate{fig:disc5}
\end{table}

\clearpage

\begin{table}[h]
\caption{Magma function to produce $B(I)$ for $d=6$.\label{tbl:magma6}}
\begin{center}
\begin{minipage}{12cm}
\begin{verbatim}
function B6(n,k,l)
    return Group<a,t,u,b,c|a^2,b^2,(t,u),(t*a)^3,(a,c),
                           t^-1*c*t*u*b*u^-1*c^-1*b^-1,
                           (a*t*b)^3,(a*t*u*b*u^-1)^3,

                           t^n,t^k*u^l,
                           (t*b)^n,(t*b)^k*(c*u)^-l>;
end function;
\end{verbatim}
\end{minipage}
\end{center}
\end{table}

\begin{table}[h]
\caption{Computations for $d=6$ (compare to Figure~\ref{fig:disc6}).\label{tbl:results6}}
\begin{center}
\begin{tabular}{| c || r |r|r||c|c|c|}
\hline
Ideal & \multicolumn{1}{c|}{$n$} & \multicolumn{1}{c|}{$k$} & \multicolumn{1}{c||}{$l$} & Method & $|B(I)|$ & $|\PSL(2,O_d/I)|$ \\ \hline \hline
$\myLangle 2,\sqrtMinusDSymbolTable{6}\myRangle$ &
2 & 0 & 1 &
\texttt{Order} & 24 &
6 \\
&&&& (Orbifold) && \\
\hline
$\myLangle 3,\sqrtMinusDSymbolTable{6}\myRangle$ &
3 & 0 & 1 &
\texttt{LowerBound1} & $\infty$ &
12 \\ \hline
$\myLangle 2\myRangle$ &
2 & 0 & 2 &
\texttt{LowerBound1} & $\infty$ &
48 \\ \hline
$\myLangle 5,2+\sqrtMinusDSymbolTable{6}\myRangle = $ &
5 & 2 & 1 &
\texttt{Order} & $1966080$ &
60 \\
$\myLangle 2+\sqrtMinusDSymbolTable{6}, 3-\sqrtMinusDSymbolTable{6}\myRangle$  & & & & & &  \\ \hline
$\myLangle \sqrtMinusDSymbolTable{6}\myRangle$ &
 &  &  &
$\containedInI\myLangle 3, \sqrtMinusDSymbolTable{6}\myRangle$ & $\infty$ &
72 \\ \hline
$\myLangle 1+\sqrtMinusDSymbolTable{6}\myRangle$ &
7 & 1 & 1 &
\texttt{LowerBound1} & $\infty$ &
168 \\ \hline
$\myLangle 4,2 \sqrtMinusDSymbolTable{6}\myRangle$ &
 &  &  &
$\containedInI\myLangle 2\myRangle$ & $\infty$ &
192 \\ \hline
$\myLangle 3\myRangle$ &
 &  &  &
$\containedInI\myLangle 3, \sqrtMinusDSymbolTable{6}\myRangle$ & $\infty$ &
324 \\ \hline
$\myLangle 2+\sqrtMinusDSymbolTable{6}\myRangle$ &
 &  &  &
$\containedInI\myLangle 5,2+\sqrtMinusDSymbolTable{6}\myRangle$ & $\geq 1966080$ &
360 \\ \hline
$\myLangle 6,2 \sqrtMinusDSymbolTable{6}\myRangle$ &
 &  &  &
$\containedInI\myLangle 2\myRangle$ & $\infty$ &
576 \\ \hline
$\myLangle 11,4+\sqrtMinusDSymbolTable{6}\myRangle = $ &
11 & 4 & 1 &
\texttt{LowerBound1} & $\geq 1.634 \cdot 10^{24}$ &
660 \\
$\myLangle 4+\sqrtMinusDSymbolTable{6}, 3-2 \sqrtMinusDSymbolTable{6}\myRangle$  & & & & & &  \\ \hline
$\myLangle 14,8+\sqrtMinusDSymbolTable{6}\myRangle = $ &
 &  &  &
$\containedInI\myLangle 1+\sqrtMinusDSymbolTable{6}\myRangle$ & $\infty$ &
1008 \\
$\myLangle 2+2 \sqrtMinusDSymbolTable{6}, 6-\sqrtMinusDSymbolTable{6}\myRangle$  & & & & & &  \\ \hline
$\myLangle 3+\sqrtMinusDSymbolTable{6}\myRangle$ &
 &  &  &
$\containedInI\myLangle 3, \sqrtMinusDSymbolTable{6}\myRangle$ & $\infty$ &
1440 \\ \hline
$\myLangle 4\myRangle$ &
 &  &  &
$\containedInI\myLangle 2\myRangle$ & $\infty$ &
1536 \\ \hline
$\myLangle 6,3 \sqrtMinusDSymbolTable{6}\myRangle$ &
 &  &  &
$\containedInI\myLangle 3, \sqrtMinusDSymbolTable{6}\myRangle$ & $\infty$ &
1944 \\ \hline
$\myLangle 4+\sqrtMinusDSymbolTable{6}\myRangle$ &
 &  & &
$\containedInI\myLangle 11,4+\sqrtMinusDSymbolTable{6}\myRangle$ & $\geq 1.634 \cdot 10^{24}$ &
3960 \\ \hline
$\myLangle 2 \sqrtMinusDSymbolTable{6}\myRangle$ &
 &  &  &
$\containedInI\myLangle 2\myRangle$ & $\infty$ &
4608 \\ \hline
$\myLangle 5\myRangle$ &
 &  &  &
$\containedInI\myLangle 5,2+\sqrtMinusDSymbolTable{6}\myRangle$ & $\geq 1966080$ &
7200 \\ \hline
$\myLangle 1+2 \sqrtMinusDSymbolTable{6}\myRangle$ &
 &  &  &
$\containedInI\myLangle 5,2-\sqrtMinusDSymbolTable{6}\myRangle$ & $\geq 1966080$ &
7500 \\ \hline
$\myLangle 2+2 \sqrtMinusDSymbolTable{6}\myRangle$ &
 &  &  &
$\containedInI\myLangle 2\myRangle$ & $\infty$ &
8064 \\ \hline
$\myLangle 5+\sqrtMinusDSymbolTable{6}\myRangle$ &
31 & 5 & 1 &
\!\!\texttt{LowerBound2(...,14880,2)}\!\! & $\geq 4.424\cdot 10^{22}$ &
14880 \\ \hline
$\myLangle 6\myRangle$ &
 &  &  &
$\containedInI\myLangle 3, \sqrtMinusDSymbolTable{6}\myRangle$ & $\infty$ &
15552 \\ \hline
$\myLangle 3+2 \sqrtMinusDSymbolTable{6}\myRangle$ &
 &  &  &
$\containedInI\myLangle 3, \sqrtMinusDSymbolTable{6}\myRangle$ & $\infty$ &
15840 \\ \hline
\end{tabular}
\end{center}
\end{table}

\clearpage

\begin{table}[h]
\caption{Magma function to produce $B(I)$ for $d=15$.\label{tbl:magma15}}
\vspace{-0.2cm}
\begin{center}
\begin{minipage}{12cm}
\begin{verbatim}
function B15(n1,k1,l1,  n2,k2,l2)
    return Group<a,c,t,u|
        (t,u),(a,c),a^2,(t*a)^3,u*c*u*a*t*u^-1*c^-1*u^-1*a*t^-1,
        
        t^n1, t^k1*u^l1,
        (u*c*a)^n2,(u*c*a)^k2*(c^-1*a*u^-1*c^-1*u^-1*t*a)^l2>;
end function;
\end{verbatim}
\end{minipage}
\end{center}
\end{table}

\begin{table}[h]
\caption{Computations for $d=15$ (compare to Figure~\ref{fig:disc15}).\label{tbl:results15}}
\vspace{-0.3cm}
\begin{center}
\begin{tabular}{| c || r |r|r|r|r|r||c|c|c|}
\hline
Ideal &
\multicolumn{1}{c|}{$n_1$} & \multicolumn{1}{c|}{$k_1$} & \multicolumn{1}{c|}{$l_1$} &
\multicolumn{1}{c|}{$n_2$} & \multicolumn{1}{c|}{$k_2$} & \multicolumn{1}{c||}{$l_2$} &
Method & $|B(I)|$ & $|\PSL(2,O_d/I)|$ \\ \hline \hline
$\myLangle 2,\frac{1+\sqrtMinusDSymbolTable{15}}{2}\myRangle$ &
2 & 0 & 1 &
1 & 0 & 2 &
\multicolumn{2}{c|}{6-Link} &
6 \\ \hline
$\myLangle 3,\frac{3+\sqrtMinusDSymbolTable{15}}{2}\myRangle = $ &
3 & 1 & 1 &
3 & 0 & 1 &
\multicolumn{2}{c|}{8-Link} &
12 \\
$\myLangle \frac{3+\sqrtMinusDSymbolTable{15}}{2}, \frac{3-\sqrtMinusDSymbolTable{15}}{2}\myRangle$  & & & & & & & \multicolumn{2}{c|}{}&  \\ \hline
$\myLangle \frac{1+\sqrtMinusDSymbolTable{15}}{2}\myRangle$ &
4 & 0 & 1 &
1 & 0 & 4 &
\multicolumn{2}{c|}{12-Link} &
24 \\ \hline
$\myLangle 2\myRangle$ &
2 & 0 & 2 &
2 & 0 & 2 &
\texttt{LowerBound1} & $\infty$ &
36 \\ \hline
$\myLangle 5,\frac{5+\sqrtMinusDSymbolTable{15}}{2}\myRangle = $ &
5 & 2 & 1 &
5 & 0 & 1 &
\multicolumn{2}{c|}{24-Link} &
60 \\
$\myLangle \frac{5+\sqrtMinusDSymbolTable{15}}{2}, \frac{5-\sqrtMinusDSymbolTable{15}}{2}\myRangle$  & & & & & & & \multicolumn{2}{c|}{} & \\ \hline
\markConjugate{$\myLangle \frac{3-\sqrtMinusDSymbolTable{15}}{2}\myRangle$} &
6 & -2 & 1 &
3 & 0 & 2 &
\multicolumn{2}{c|}{24-Link} &
72 \\ \hline
$\myLangle 4,1+\sqrtMinusDSymbolTable{15}\myRangle$ &
&&&&&&
$\containedInI\myLangle 2\myRangle$ & $\infty$ &
144 \\ \hline
$\myLangle 8,\frac{9+\sqrtMinusDSymbolTable{15}}{2}\myRangle = $ &
8 & 4 & 1 &
2 & 1 & 4 &
\texttt{Order} & 46656 &
192 \\
$\myLangle 1+\sqrtMinusDSymbolTable{15}, \frac{7-\sqrtMinusDSymbolTable{15}}{2}\myRangle$  & & & & & & & & &  \\ \hline
$\myLangle 3\myRangle$ &
3 & 0 & 3 &
3 & 0 & 3 &
\texttt{LowerBound1} & $\infty$ &
324 \\ \hline
$\myLangle \frac{5+\sqrtMinusDSymbolTable{15}}{2}\myRangle$ &
10 & 2 & 1 &
5 & 0 & 2 &
\texttt{LowerBound1} & $\infty$ &
360 \\ \hline
$\myLangle 6,3+\sqrtMinusDSymbolTable{15}\myRangle = $ &
&&&&&&
$\containedInI\myLangle 2\myRangle$ & $\infty$ &
432 \\
$\myLangle 3+\sqrtMinusDSymbolTable{15}, 3-\sqrtMinusDSymbolTable{15}\myRangle$  & & & & & & & & &  \\ \hline
\markConjugate{$\myLangle 12,\frac{9+\sqrtMinusDSymbolTable{15}}{2}\myRangle = $} &
12 & 4 & 1 &
3 & 0 & 4 &
\texttt{LowerBound1} & $\infty$ &
576 \\
$\myLangle 3-\sqrtMinusDSymbolTable{15}, \frac{9+\sqrtMinusDSymbolTable{15}}{2}\myRangle$  & & & & & &  & & & \\ \hline
$\myLangle 1+\sqrtMinusDSymbolTable{15}\myRangle$ &
&&&&&&
$\containedInI\myLangle 2\myRangle$ & $\infty$ &
1152 \\ \hline
$\myLangle 4\myRangle$ &
&&&&&&
$\containedInI\myLangle 2\myRangle$ & $\infty$ &
1152 \\ \hline
$\myLangle \sqrtMinusDSymbolTable{15}\myRangle$ &
15 & 7 & 1 &
15 & 0 & 1 &
\texttt{LowerBound1} & $\infty$ &
1440 \\ \hline
$\myLangle \frac{7+\sqrtMinusDSymbolTable{15}}{2}\myRangle$ &
&&&&&&
$\containedInI\myLangle 8,\frac{9-\sqrtMinusDSymbolTable{15}}{2}\myRangle$ & $\geq 46656$ &
1536 \\ \hline
$\myLangle 6,\frac{3+3 \sqrtMinusDSymbolTable{15}}{2}\myRangle$ &
&&&&&&
$\containedInI\myLangle 3\myRangle$ & $\infty$ &
1944 \\ \hline
\markConjugate{$\myLangle 17,\frac{11+\sqrtMinusDSymbolTable{15}}{2}\myRangle = $} &
17 & 5 & 1 &
17 & -8 & 1 &
\texttt{LowerBound1} & $\infty$ &
2448 \\
$\myLangle \frac{1-3 \sqrtMinusDSymbolTable{15}}{2}, \frac{11+\sqrtMinusDSymbolTable{15}}{2}\myRangle$  & & & & & & & & &  \\ \hline
$\myLangle 2+\sqrtMinusDSymbolTable{15}\myRangle$ &
19 & 10 & 1 &
19 & 4 & 1 &
\texttt{LowerBound1} & $\geq 1.370 \cdot 10^{39}$ &
3420 \\ \hline
$\myLangle 3+\sqrtMinusDSymbolTable{15}\myRangle$ &
&&&&&&
$\containedInI\myLangle 2\myRangle$ & $\infty$ &
3456 \\ \hline
$\myLangle \frac{9+\sqrtMinusDSymbolTable{15}}{2}\myRangle$ &
&&&&&&
$\containedInI\myLangle 8,\frac{9+\sqrtMinusDSymbolTable{15}}{2}\myRangle$ & $\geq $46656 &
4608 \\ \hline
$\myLangle 5\myRangle$ &
5 & 0 & 5 &
5 & 0 & 5 &
\texttt{LowerBound1(...,5)} & $\infty$ &
7500 \\ \hline
$\myLangle 6\myRangle$ &
&&&&&&
$\containedInI\myLangle 2\myRangle$ & $\infty$ &
11664 \\ \hline
$\myLangle \frac{1+3 \sqrtMinusDSymbolTable{15}}{2}\myRangle$ &
 &&&&&&
 $\containedInI\myLangle 17,\frac{11-\sqrtMinusDSymbolTable{15}}{2}\myRangle$ & $\infty$ &
14688 \\ \hline
$\myLangle \frac{11+\sqrtMinusDSymbolTable{15}}{2}\myRangle$ &
 &&&&&&
 $\containedInI\myLangle 17,\frac{11+\sqrtMinusDSymbolTable{15}}{2}\myRangle$ & $\infty$ &
14688 \\ \hline
$\myLangle 4+\sqrtMinusDSymbolTable{15}\myRangle$ &
31 & -14 & 1 &
31 & 13 & 1 &
\!\!\texttt{LowerBound2(\phantom{AAAAAA}}\!\! & $\infty$ &
14880 \\
&&&&&&& \!\!\texttt{\phantom{LowerB}...,14880,2)}\!\! && \\ \hline
$\myLangle \frac{3+3 \sqrtMinusDSymbolTable{15}}{2}\myRangle$ &
&&&&&&
$\containedInI\myLangle 3\myRangle$ & $\infty$ &
15552 \\ \hline
\end{tabular}
\end{center}
\explainConjugate{fig:disc15}
\end{table}

\clearpage

\begin{table}[h]
\caption{Magma function to produce $B(I)$ for $d=23$.\label{tbl:magma23}}
\begin{center}
\begin{minipage}{12cm}
\begin{verbatim}
function B23(n1,k1,l1, n2,k2,l2, n3,k3,l3)
    return Group<g1,g2,g3,g4,g5|
        g3^3,(g3*g2)^2,(g1,g2),(g4,g5),
        g5*g2^-1*g3^-1*g5^-1*g1^-1*g2^-1*g3^-1*g1,
        g4^-1*g5*g3*g2*g5^-1*g2*g4*g3,
        
        g2^n1,g2^k1*g1^l1,
        g4^n2,g4^k2*g5^l2,
        (g4*g3*g2)^n3,(g4*g3*g2)^k3*(g2^-1*g5*g3*g2)^l3>;
end function;
\end{verbatim}
\end{minipage}
\end{center}
\end{table}

\begin{table}[h]
\caption{Computations for $d=23$ (compare to Figure~\ref{fig:disc23}).\label{tbl:results23}}
\begin{center}
\begin{tabular}{| c || r |r|r|r|r|r|r|r|r||c|c|c|}
\hline
Ideal &
\multicolumn{1}{c|}{$n_1$} & \multicolumn{1}{c|}{$k_1$} & \multicolumn{1}{c|}{$l_1$} &
\multicolumn{1}{c|}{$n_2$} & \multicolumn{1}{c|}{$k_2$} & \multicolumn{1}{c|}{$l_2$} &
\multicolumn{1}{c|}{$n_3$} & \multicolumn{1}{c|}{$k_3$} & \multicolumn{1}{c||}{$l_3$} &
Method & $|B(I)|$ & $\!\!|\PSL\!\left(2,\!\frac{O_d}{I}\right)|\!\!$ \\ \hline \hline
$\myLangle 2,\frac{1+\sqrtMinusDSymbolTable{23}}{2}\myRangle$ &
2 & 1 & 1 &
1 & 0  & 2 &
2 & 1 & 1 & \multicolumn{2}{c|}{9-Link} &
6 \\ \hline
$\myLangle 3,\frac{1+\sqrtMinusDSymbolTable{23}}{2}\myRangle$ &
3&1&1&
3&1 &1 &
3 &0 &1 & \multicolumn{2}{c|}{12-Link} &
12 \\ \hline
\markConjugate{$\myLangle 4,\frac{3-\sqrtMinusDSymbolTable{23}}{2}\myRangle = $} &
4&1&-1& 2 & 1& -2 & 4&-1 &1 & \multicolumn{2}{c|}{18-Link} &
24 \\
$\myLangle \frac{3-\sqrtMinusDSymbolTable{23}}{2}, \frac{5+\sqrtMinusDSymbolTable{23}}{2}\myRangle$  & & & & & & & & & & \multicolumn{2}{c|}{} &  \\ \hline
$\myLangle 2\myRangle$ &
2 & 0 & 2 &
2 & 0  & 2 & 2 & 0 & 2 & \texttt{LowerBound1} & $\infty$ &
36 \\ \hline
$\myLangle \frac{1+\sqrtMinusDSymbolTable{23}}{2}\myRangle$ &
6 & 1 & 1 &
3 & -1 & 2 & 6 & -3 & 1 & \texttt{LowerBound1} & $\infty$ &
72 \\ \hline
\markConjugate{$\myLangle 6,\frac{5-\sqrtMinusDSymbolTable{23}}{2}\myRangle = $} &
6 & -2 & 1 &
6 & 1 & 1 & 3 & 0 & 2 & \texttt{Order} & 288 &
72 \\
$\myLangle \frac{5-\sqrtMinusDSymbolTable{23}}{2}, \frac{7+\sqrtMinusDSymbolTable{23}}{2}\myRangle$  & & & & & &  & & & & & &\\ \hline
$\myLangle 4,1+\sqrtMinusDSymbolTable{23}\myRangle$ &
 &  &  &
& & & & & & $\containedInI\myLangle 2\myRangle$ & $\infty$ &
144 \\ \hline
\markConjugate{$\myLangle \frac{3-\sqrtMinusDSymbolTable{23}}{2}\myRangle$} &
8 & -1 & 1 &
4 & 1 & 2 & 8 & 3 & 1
 & \texttt{LowerBound1} & $\geq 139968$ &
192 \\ \hline
$\myLangle 3\myRangle$ &
3 & 0 & 3 &
3&0 &3 &3 &0 & 3& \!\!\texttt{LowerBound1(}\!\! & $\infty$ &
288 \\
& & & & & & & & & & \!\!\texttt{\phantom{Lower}...,12)}\!\! & & \\ \hline
$\myLangle 9,\frac{7+\sqrtMinusDSymbolTable{23}}{2}\myRangle = $ &
9 & 4 & 1 &
9 & -2 & 1 & 9 & 6 & 1 
 & \texttt{LowerBound1} & $\geq 82944$ &
324 \\
$\myLangle \frac{7+\sqrtMinusDSymbolTable{23}}{2}, 2-\sqrtMinusDSymbolTable{23}\myRangle$  & & & & & & & & & & & & \\ \hline
$\myLangle 6,1+\sqrtMinusDSymbolTable{23}\myRangle$ &
 &  &  &
& & & & & & $\containedInI\myLangle 2\myRangle$ & $\infty$ &
432 \\ \hline
$\myLangle \frac{5+\sqrtMinusDSymbolTable{23}}{2}\myRangle$ &
12 & 3 & 1 &
6& 3& 2& 12& 7& 1& \texttt{LowerBound1}  & $\geq 2.533\cdot 10^{25}$ &
576 \\ \hline
$\myLangle 12,\frac{13+\sqrtMinusDSymbolTable{23}}{2}\myRangle = $ &
 &  &  &
& & & & & &$\containedInI$  & $\infty$ &
576 \\
$\myLangle 1+\sqrtMinusDSymbolTable{23}, \frac{11-\sqrtMinusDSymbolTable{23}}{2}\myRangle$  & & & & & & & & & & $\myLangle \frac{1+\sqrtMinusDSymbolTable{23}}{2}\myRangle$  & & \\ \hline
$\myLangle 13,\frac{9+\sqrtMinusDSymbolTable{23}}{2}\myRangle = $ &
13 & 5 & 1 &
13 & 4 & 1 & 13 & 2 & 1 & \texttt{LowerBound1} &\!$\geq 4.564\cdot 10^{32}$\! &
1092 \\
$\myLangle \frac{9+\sqrtMinusDSymbolTable{23}}{2}, 4-\sqrtMinusDSymbolTable{23}\myRangle$  & & & & & && & & & & &  \\ \hline
$\myLangle 4\myRangle$ &
 &  &  &
& & & & & & $\containedInI\myLangle 2\myRangle$ & $\infty$ &
1152 \\ \hline
$\myLangle 8,3+\sqrtMinusDSymbolTable{23}\myRangle = $ &
 &  &  &
& & & & & & $\containedInI\myLangle 2\myRangle$ & $\infty$ &
1152 \\
\!\!\!\!$\myLangle 3+\sqrtMinusDSymbolTable{23}, 5-\sqrtMinusDSymbolTable{23}\myRangle$\!\!\!\! & & & & & & & & & & & & \\ \hline
$\myLangle 16,\frac{19+\sqrtMinusDSymbolTable{23}}{2}\myRangle = $ &
 &  &  &
& & & & & & $\containedInI$ & $\geq 139968$ &
1536 \\
$\myLangle 3+\sqrtMinusDSymbolTable{23}, \frac{13-\sqrtMinusDSymbolTable{23}}{2}\myRangle$  & & & & & & & & & &$\myLangle \frac{3+\sqrtMinusDSymbolTable{23}}{2}\myRangle$ &
 & \\ \hline
$\myLangle 6,\frac{3+3 \sqrtMinusDSymbolTable{23}}{2}\myRangle$ &
 &  &  &
& & & & & & $\containedInI\myLangle 3\myRangle$ & $\infty$ &
1728 \\ \hline
$\myLangle \frac{7+\sqrtMinusDSymbolTable{23}}{2}\myRangle$ &
&&&
& & & & & & $\containedInI$ & $\geq 82944$ &
1944 \\
& & & & & & & & & & $\myLangle 9,\frac{7+\sqrtMinusDSymbolTable{23}}{2}\myRangle$ & & \\ \hline
$\myLangle 18,\frac{11+\sqrtMinusDSymbolTable{23}}{2}\myRangle=$ &
 &  &  &
& & & & & & $\containedInI$ & $\infty$ &
1944 \\
$\myLangle \frac{11+\sqrtMinusDSymbolTable{23}}{2}, \frac{3-3 \sqrtMinusDSymbolTable{23}}{2}\myRangle$& & & & & &  & & & & $\myLangle \frac{1-\sqrtMinusDSymbolTable{23}}{2}\myRangle$ & &  \\ \hline
$\myLangle 1+\sqrtMinusDSymbolTable{23}\myRangle$ &
 &  &  &
& & & & & & $\containedInI\myLangle 2\myRangle$ & $\infty$ &
3456 \\  \hline
$\myLangle \sqrtMinusDSymbolTable{23}\myRangle$ &
23 & -11 & 1 &
23 & -6 & 1 & 23 & -6 & 1 & \!\!\texttt{LowerBound2(}\!\! & $\infty$ &
6072 \\ 
&&&&&&&&&& \!\!\texttt{\phantom{L}...,6072,2)}\!\! && \\\hline
$\myLangle \frac{9+\sqrtMinusDSymbolTable{23}}{2}\myRangle$ &
 &  &  &
& & & & & & $\containedInI$ & \!$\geq 4.564\cdot 10^{32}$\! &
6552 \\
& & & & & & & & & & $\myLangle 13,\frac{9+\sqrtMinusDSymbolTable{23}}{2}\myRangle$ & & \\ \hline
$\myLangle 5\myRangle$ &
5 & 0 & 5 &
5 & 0 & 5 & 5 & 0 & 5 & \texttt{LowerBound1} & $\infty$ &
7800 \\ \hline
$\myLangle 2+\sqrtMinusDSymbolTable{23}\myRangle$ &
 &  &  &
& & & & & & $\containedInI$ & $\geq 82944$ &
8748 \\ 
& & & & & & & & & & $\myLangle 9,\frac{7-\sqrtMinusDSymbolTable{23}}{2}\myRangle$ & &  \\ \hline
$\myLangle 3+\sqrtMinusDSymbolTable{23}\myRangle$ &
 &  &  &
& & & & & & $\containedInI\myLangle 2\myRangle$ & $\infty$ &
9216 \\ \hline
$\myLangle 6\myRangle$ &
 &  &  &
& & & & & & $\containedInI\myLangle 2\myRangle$ & $\infty$ &
10368 \\ \hline
$\myLangle \frac{11+\sqrtMinusDSymbolTable{23}}{2}\myRangle$ &
 &  &  &
& & & & & & $\containedInI$ & $\infty$ &
15552 \\ 
& & & & & & & & & & $\myLangle \frac{1-\sqrtMinusDSymbolTable{23}}{2}\myRangle$ & &  \\ \hline
\end{tabular}
\end{center}
\explainConjugate{fig:disc23}
\end{table}

\clearpage

\begin{table}[h]
\caption{Magma function to produce $B(I)$ for $d=31$.\label{tbl:magma31}}
\begin{center}
\begin{minipage}{12cm}
\begin{verbatim}
function B31(n1,k1,l1, n2,k2,l2, n3,k3,l3)
    return Group<g1,g2,g3,g4,g5|(g1,g3),(g2)^3,(g2*g1^-1)^2,(g5,g4),
        g4*g1^-1*g3^-1*g2*g3*g4^-1*g2*g4*g3^-1*g1^-1*g2*g3*g4^-1*g2,
        g5*g3^-1*g2*g3*g4^-1*g2*g1^-1*g5^-1*g2^-1*g4*g3^-1*g2^-1*g3*g1,
        g2*g3*g4^-1*g2*g1^-1*g4*g3^-1*g2*g3*g4^-1*g1*g2^-1*g4*g3^-1,
        
        g1^n1,g1^k1*g3^l1,
        g4^n2,g4^k2*g5^l2,
        (g1*g5)^n3,(g1*g5)^k3*(g3^-1*g2*g3*g4^-1*g2*g5)^l3>;
end function;
\end{verbatim}
\end{minipage}
\end{center}
\end{table}

\begin{table}[h]
\caption{Computations for $d=31$ (compare to Figure~\ref{fig:disc31}).\label{tbl:results31}}
\begin{center}
\begin{tabular}{| c || r |r|r|r |r|r|r |r|r||c|c|c|}
\hline
Ideal &
\multicolumn{1}{c|}{$n_1$} & \multicolumn{1}{c|}{$k_1$} & \multicolumn{1}{c|}{$l_1$} &
\multicolumn{1}{c|}{$n_2$} & \multicolumn{1}{c|}{$k_2$} & \multicolumn{1}{c|}{$l_2$} &
\multicolumn{1}{c|}{$n_3$} & \multicolumn{1}{c|}{$k_3$} & \multicolumn{1}{c||}{$l_3$} &
Method & $|B(I)|$ & $\!\!|\PSL\!\left(2,\!\frac{O_d}{I}\right)|\!\!$ \\ \hline \hline
$\myLangle 2,\frac{1+\sqrtMinusDSymbolTable{31}}{2}\myRangle$ &
2 & 0 & 1 &
1&0 &2 & 1& 0&2 &
\multicolumn{2}{c|}{9-Link} &
6 \\ \hline
$\myLangle 4,\frac{1+\sqrtMinusDSymbolTable{31}}{2}\myRangle$ &
4 & 0 & 1 &
2 &1 & 2 &
1 & 0 & 4 &
\multicolumn{2}{c|}{18-Link} &
24 \\ \hline
$\myLangle 2\myRangle$ &
2 & 0 & 2 &
2 & 0 & 2 &
2 & 0 & 2 &
\texttt{LowerBound1} & $\infty$ &
36 \\ \hline
$\myLangle 5,\frac{3+\sqrtMinusDSymbolTable{31}}{2}\myRangle = $ &
5 & -1 & 1 & 5 &1 & 1 & 5 &1 & 1 &
\multicolumn{2}{c|}{36-Link} &
60 \\
$\myLangle \frac{3+\sqrtMinusDSymbolTable{31}}{2}, \frac{7-\sqrtMinusDSymbolTable{31}}{2}\myRangle$  & & & & & & & & &  & \multicolumn{2}{c|}{} & \\ \hline
$\myLangle 4,1+\sqrtMinusDSymbolTable{31}\myRangle$ &
 &  &  & & & & & & & 
$\containedInI\myLangle 2\myRangle$ & $\infty$ &
144 \\ \hline
$\myLangle 7,\frac{5+\sqrtMinusDSymbolTable{31}}{2}\myRangle = $ &
7 & -2 & 1 &
7 & -2 & 1 &
7 & 0 & 1 &
\texttt{LowerBound1} & $\infty$ &
168 \\
$\myLangle \frac{5+\sqrtMinusDSymbolTable{31}}{2}, \frac{9-\sqrtMinusDSymbolTable{31}}{2}\myRangle$  & & & & & & & & & & & &  \\ \hline
$\myLangle \frac{1+\sqrtMinusDSymbolTable{31}}{2}\myRangle$ &
8 & 0 & 1 &
4 & 1 & 2 &
2 & -1 & 4 &
\texttt{LowerBound1} & $\infty$ &
192 \\ \hline
$\myLangle 3\myRangle$ &
3 & 0 & 3 &
3 & 0 & 3 &
3 & 0 & 3 &
\!\!\texttt{LowerBound1(}\!\! & $\infty$ &
360 \\
&&&&&&&&&& \!\!\texttt{\phantom{LowerB}...,3)}\!\! & & \\
 \hline
$\myLangle \frac{3+\sqrtMinusDSymbolTable{31}}{2}\myRangle$ &
10 & -1 & 1 &
10 & -4 & 1 &
10 & 1 & 1 &
\texttt{LowerBound1} & \!$\geq 8.578\cdot 10^{28}$\! &
360 \\ \hline
$\myLangle 10,\frac{7+\sqrtMinusDSymbolTable{31}}{2}\myRangle = $ &
10  & -3  & 1  &
 10 &  2 & 1  &
  10& 3 & 1 &
\texttt{LowerBound1} & $\infty$ &
360 \\
$\myLangle \frac{7+\sqrtMinusDSymbolTable{31}}{2}, 3-\sqrtMinusDSymbolTable{31}\myRangle$  & & & & & & & & & & & &  \\ \hline
$\myLangle \frac{5+\sqrtMinusDSymbolTable{31}}{2}\myRangle$ &
 &  &  &
  &  &  &
 &  &  &
$\containedInI$ & $\infty$ &
1008 \\ 
&&&&&&&&&& $\myLangle 7,\frac{5+\sqrtMinusDSymbolTable{31}}{2}\myRangle$ & & \\
\hline
$\myLangle 14,\frac{9+\sqrtMinusDSymbolTable{31}}{2}\myRangle = $ &
 &  &  &
  &  &  &
 &  &  &
$\containedInI$ & $\infty$ &
1008 \\ 
$\myLangle \frac{9+\sqrtMinusDSymbolTable{31}}{2}, 5-\sqrtMinusDSymbolTable{31}\myRangle$  &&&&&&&&&& $\myLangle 7,\frac{5-\sqrtMinusDSymbolTable{31}}{2}\myRangle$ & & \\ \hline
$\myLangle 4\myRangle$ &
 &  &  &
  &  &  &
 &  &  &
$\containedInI\myLangle 2\myRangle$ & $\infty$ &
1152 \\ \hline
$\myLangle 8,1+\sqrtMinusDSymbolTable{31}\myRangle$ &
 &  &  &
  &  &  &
 &  &  &
$\containedInI\myLangle 2\myRangle$ & $\infty$ &
1152 \\ \hline
$\myLangle 16,\frac{17+\sqrtMinusDSymbolTable{31}}{2}\myRangle = $ &
 &  &  &
  &  &  &
 &  &  & $\containedInI$ & $\infty$ &
1536 \\
$\myLangle 1+\sqrtMinusDSymbolTable{31}, \frac{15-\sqrtMinusDSymbolTable{31}}{2}\myRangle$  & & & & & & & & & & $\myLangle \frac{1+\sqrtMinusDSymbolTable{31}}{2}\myRangle$ & &  \\ \hline
$\myLangle 6,\frac{3+3 \sqrtMinusDSymbolTable{31}}{2}\myRangle$ &
 &  &  &
  &  &  &
 &  &  & $\containedInI\myLangle 3\myRangle$ & $\infty$ &
2160 \\ \hline
$\myLangle \frac{7+\sqrtMinusDSymbolTable{31}}{2}\myRangle$ &
 &  &  &
  &  &  &
 &  &  &
$\containedInI$ & $\infty$ &
2880 \\
 & & & & & & & & & & \!\!$\myLangle 10,\frac{7+\sqrtMinusDSymbolTable{31}}{2}\myRangle$\!\! & & \\
 \hline
$\myLangle 19,\frac{11+\sqrtMinusDSymbolTable{31}}{2}\myRangle = $ &
19 & 5 & -1 &
19 & 3 & 1 &
19 & 7 & 1 &
\!\!\texttt{LowerBound2(}\!\! & \!\!\!$\geq 1.769\cdot 10^{11}$\!\!\! &
3420 \\
$\myLangle \frac{11+\sqrtMinusDSymbolTable{31}}{2}, \frac{5-3 \sqrtMinusDSymbolTable{31}}{2}\myRangle$  & & & & & &  &&&&\!\!\texttt{\phantom{LL}...,3420,1)}\!\!&&\\ \hline
$\myLangle 5\myRangle$ &
 5&  0& 5 &
 5 & 0 &5  &
 5&  0& 5 &
\!\!\texttt{LowerBound1(}\!\! & $\infty$ &
7200 \\
&&&&&&&&&& \!\!\texttt{\phantom{Lower}...,25)}\!\! & & \\ \hline
$\myLangle \frac{9+\sqrtMinusDSymbolTable{31}}{2}\myRangle$ &
 &  &  &
  &  &  &
 &  &  &
$\containedInI$ & $\infty$ &
8064 \\ 
&&&&&&&&&& $\myLangle 7,\frac{5-\sqrtMinusDSymbolTable{31}}{2}\myRangle$ & & \\ \hline
$\myLangle 1+\sqrtMinusDSymbolTable{31}\myRangle$ &
 &  &  & & & & & & & 
$\containedInI\myLangle 2\myRangle$ & $\infty$ &
9216 \\ \hline
$\myLangle 6\myRangle$ &
 &  &  & & & & & & & 
$\containedInI\myLangle 2\myRangle$ & $\infty$ &
12960 \\ \hline
$\myLangle \sqrtMinusDSymbolTable{31}\myRangle$ &
31 & 16 & 1 &
31 & 8 & 1 &
31 &5 & -1 &
\!\!\texttt{LowerBound2(}\!\! & $\infty$ &
14880 \\ 
&&&&&&&&&& \!\!\texttt{...,14880,1)}\!\! && \\ \hline
$\myLangle 2+\sqrtMinusDSymbolTable{31}\myRangle$ &
 &  &  &
  &  &  &
 &  &  &
$\containedInI$ & $\infty$ &
20160 \\
&&&&&&&&&& $\myLangle 7,\frac{5-\sqrtMinusDSymbolTable{31}}{2}\myRangle$ & & \\ \hline
$\myLangle \frac{11+\sqrtMinusDSymbolTable{31}}{2}\myRangle$ &
 &  &  &
  &  &  &
 &  &  &
$\containedInI$ & \!\!\!$\geq 1.769\cdot 10^{11}$\!\!\! &
20520 \\ 
&&&&&&&&&& \!\!\! $\myLangle 19,\frac{11+\sqrtMinusDSymbolTable{31}}{2}\myRangle$\!\!\! && \\ \hline
\end{tabular}
\end{center}
\end{table}

\clearpage

\begin{table}[h]
\caption{Magma function to produce $B(I)$ for $d=39$.\label{tbl:magma39}}
\begin{center}
\begin{minipage}{14cm}
\begin{verbatim}
function B39(n1,k1,l1, n2,k2,l2, n3,k3,l3, n4,k4,l4)
    return Group<g1,g2,g3,g4,g5,g6,g7|
        g3^3,(g4,g6),(g3*g5)^2,(g2,g1),(g1^-1*g3^-1)^2,(g3^-1,g7^-1),
        (g5^-1*g1)^3,g5^-1*g1*g6^-1*g4^-1*g5*g4*g1^-1*g6,
        g4^-1*g5*g4*g2^-1*g7*g5^-1*g7^-1*g2,(g7*g5^-1*g7^-1*g1)^3,
        g6*g1^-1*g5*g6^-1*g4^-1*g5*g4*g1^-1*g4^-1*g5*g4*g1^-1,
        
        g1^n1, g1^k1*g2^l1,
        g4^n2, g4^k2*g6^l2,
        (g5^-1*g6)^n3, (g5^-1*g6)^k3*(g4*g1^-1*g6)^l3,
        g5^n4, g5^k4*(g4*g2^-1*g7)^l4>;
end function;
\end{verbatim}
\end{minipage}
\end{center}
\end{table}

\begin{table}[h]
\caption{Computations for $d=39$ (compare to Figure~\ref{fig:disc39}).\label{tbl:results39}}
\begin{center}
\begin{tabular}{| c || r|r|r||c|c|c|}
\hline
Ideal & \multicolumn{1}{c|}{$n_i$} & \multicolumn{1}{c|}{$k_i$} & \multicolumn{1}{c||}{$l_i$} & Method & $|B(I)|$ & $|\PSL(2,O_d/I)|$ \\ \hline \hline
$\myLangle 2,\frac{1+\sqrtMinusDSymbolTable{39}}{2}\myRangle$ &
2 & 0 & 1 &
\texttt{Order} & 18 &
6 \\ 
& 2 & 0 & 1 & & & \\ 
& 2 & 0 & 1 & & & \\
& 2 & -1 & 1 & & & \\  \hline
$\myLangle 3,\frac{3+\sqrtMinusDSymbolTable{39}}{2}\myRangle$ &
3 & 1 & 1 &
\texttt{LowerBound} & $\infty$ &
12 \\
& 1 & 0 & 3 & (Orbifold) & & \\ 
& 3 & 2 & 1 & & & \\ 
& 3 & -1 & 1 & & & \\
 \hline
\markConjugate{$\myLangle 4,\frac{3-\sqrtMinusDSymbolTable{39}}{2}\myRangle$} &
4 & -2 & 1 &
\texttt{Order} & 72 &
24 \\ 
& 4 & 0 & 1 & & & \\ 
& 4 & 0 & 1 & & & \\ 
& 4 & -1 & 1 & & & \\
\hline
$\myLangle 2\myRangle$ &
2 & 0 & 2 &
\texttt{LowerBound1(...,2)} & $\infty$ &
36 \\ 
& 2 & 0 & 2 & & & \\ 
& 2 & 0 & 2 & & & \\ 
& 2 & 0 & 2 & & & \\ 
\hline
$\myLangle 5,\frac{1+\sqrtMinusDSymbolTable{39}}{2}\myRangle$ &
5 & 0 & 1 &
\texttt{LowerBound1} & $\infty$ &
60 \\
& 5 & 1 & -1 & & & \\ 
& 5 & 2 & 1 & & & \\ 
& 5 & 1 & 1 & & & \\ 
 \hline
$\myLangle 6,\frac{3+\sqrtMinusDSymbolTable{39}}{2}\myRangle = $ &
 &  &  &
$\containedInI \myLangle 3,\frac{3+\sqrtMinusDSymbolTable{39}}{2}\myRangle$ & $\infty$ &
72 \\
$\myLangle \frac{3+\sqrtMinusDSymbolTable{39}}{2}, \frac{9-\sqrtMinusDSymbolTable{39}}{2}\myRangle$  & & & & & &  \\ \hline
$\myLangle 4,1+\sqrtMinusDSymbolTable{39}\myRangle$ &
 &  &  &
$\containedInI\myLangle 2\myRangle$ & $\infty$ &
144 \\ \hline
$\myLangle 8,\frac{5+\sqrtMinusDSymbolTable{39}}{2}\myRangle = $ &
8 & 2 & 1 &
\texttt{LowerBound1} & $\geq 7.004\cdot 10^{27}$ &
192 \\
$\myLangle \frac{5+\sqrtMinusDSymbolTable{39}}{2}, \frac{11-\sqrtMinusDSymbolTable{39}}{2}\myRangle$  &8  &4 & 1&  & &  \\
& 8 & 4 & 1 & & & \\ 
& 8 & 3 & 1 &&&\\
 \hline
$\myLangle 3\myRangle$ &
 &  &  &
$\containedInI\myLangle 3,\frac{3+\sqrtMinusDSymbolTable{39}}{2}\myRangle$ & $\infty$ &
324 \\ \hline
$\myLangle \frac{1+\sqrtMinusDSymbolTable{39}}{2}\myRangle$ &
 &  &  &
$\containedInI\myLangle 5,\frac{1+\sqrtMinusDSymbolTable{39}}{2}\myRangle$ & $\infty$ &
360 \\ \hline
$\myLangle 10,\frac{9+\sqrtMinusDSymbolTable{39}}{2}\myRangle = $ &
&  &  &
$\containedInI\myLangle 5,\frac{1-\sqrtMinusDSymbolTable{39}}{2}\myRangle$ & $\infty$ &
360 \\
$\myLangle \frac{9+\sqrtMinusDSymbolTable{39}}{2}, 1-\sqrtMinusDSymbolTable{39}\myRangle$  & & & & & &  \\ \hline
$\myLangle 6,3+\sqrtMinusDSymbolTable{39}\myRangle$ &
&  &  &
$\containedInI\myLangle 2\myRangle$ & $\infty$ &
1432 \\ \hline
$\myLangle \frac{3+\sqrtMinusDSymbolTable{39}}{2}\myRangle$ &
 &  &  &
$\containedInI\myLangle 3,\frac{3+\sqrtMinusDSymbolTable{39}}{2}\myRangle$ & $\infty$ &
576 \\ \hline
$\myLangle 11,\frac{7+\sqrtMinusDSymbolTable{39}}{2}\myRangle = $ &
11 & 3 & 1 &
\texttt{LowerBound1} & $\geq 5.822\cdot 10^{75}$ &
660 \\
$\myLangle \frac{7+\sqrtMinusDSymbolTable{39}}{2}, 4-\sqrtMinusDSymbolTable{39}\myRangle$  & 11 & -2 & 1 & & &  \\
& 11 &3 & 1 &&& \\
& 11 &4 & 1 &&& \\
 \hline
$\myLangle 4\myRangle$ &
 &  & &
$\containedInI\myLangle 2\myRangle$ & $\infty$ &
1152 \\ \hline
$\myLangle 15,\frac{9+\sqrtMinusDSymbolTable{39}}{2}\myRangle = $ &
&  &  &
$\containedInI\myLangle 5,\frac{1-\sqrtMinusDSymbolTable{39}}{2}\myRangle$ & $\infty$ &1440 \\
$\myLangle \frac{9+\sqrtMinusDSymbolTable{39}}{2}, 6-\sqrtMinusDSymbolTable{39}\myRangle$  & & & & & &  \\ \hline
$\myLangle \frac{5+\sqrtMinusDSymbolTable{39}}{2}\myRangle$ &
 &  &  &
$\containedInI\myLangle 8,\frac{5+\sqrtMinusDSymbolTable{39}}{2}\myRangle$ & $\geq 7.004\cdot 10^{27}$ &
1536 \\ \hline
$\myLangle 6,\frac{3+3 \sqrtMinusDSymbolTable{39}}{2}\myRangle$ &
 &  & &
$\containedInI\myLangle 3\myRangle$ & $\infty$ &
1944 \\ \hline
$\myLangle \frac{7+\sqrtMinusDSymbolTable{39}}{2}\myRangle$ &
 &  &  &
$\containedInI\myLangle 11,\frac{7+\sqrtMinusDSymbolTable{39}}{2}\myRangle$ & $\geq 5.822\cdot 10^{75}$ &
3960 \\ \hline
$\myLangle 5\myRangle$ &
&&&
$\containedInI\myLangle 5,\frac{1+\sqrtMinusDSymbolTable{39}}{2}\myRangle$ & $\infty$ &
7200 \\ \hline
$\myLangle \frac{9+\sqrtMinusDSymbolTable{39}}{2}\myRangle$ &
&  &  &
$\containedInI\myLangle 5,\frac{1-\sqrtMinusDSymbolTable{39}}{2}\myRangle$ & $\infty$ & 
8640 \\ \hline
$\myLangle 6\myRangle$ &
&  & &
$\containedInI\myLangle 2\myRangle$ & $\infty$ &
11664 \\ \hline
\end{tabular}
\end{center}
\explainConjugate{fig:disc39}
\end{table}

\clearpage

\begin{table}[h]
\caption{Magma function to produce $B(I)$ for $d=47$.\label{tbl:magma47}}
\begin{center}
\begin{minipage}{14cm}
\begin{verbatim}
function B47(n1,k1,l1,  n2,k2,l2, n3,k3,l3, n4,k4,l4, n5,k5,l5)
    return Group<g1,g2,g3,g4,g5,g6,g7|
        g1^3,(g3,g2),(g2^-1*g1)^2,(g5,g7),g2^-1*g1*g6*g1^-1*g2*g6^-1,
        g6*g2^-1*g4^-1*g5*g3^-1*g6^-1*g4*g2*g3*g5^-1,
        g7^-1*g2^-1*g5^-1*g4*g1*g4^-1*g2*g7*g4*g1^-1*g4^-1*g5,
        g3*g5^-1*g4*g1*g4^-1*g2*g5*g3^-1*g2^-1*g4^-1*g1^-1*g4,
        g5^-1*g4*g1*g4^-1*g7^-1*g2^-1*g4*g1^-1*g4^-1*g5*g3^-1*g2*g3*g7,

        g2^n1,g2^k1*g3^l1,
        g5^n2,g5^k2*g7^l2,
        (g2*g7)^n3, (g2*g7)^k3*(g4*g1^-1*g4^-1*g5)^l3,
        (g6*g2^-1*g4^-1)^n4, (g6*g2^-1*g4^-1)^k4*(g5*g3^-1*g2^-1*g4^-1)^l4,
        (g6^-1*g1^-1*g4)^n5, (g6^-1*g1^-1*g4)^k5*(g3*g5^-1*g4*g1)^l5>;
end function;
\end{verbatim}
\end{minipage}
\end{center}
\end{table}

\begin{table}[h]
\caption{Computations for $d=47$ (compare to Figure~\ref{fig:disc47}).\label{tbl:results47}}
\begin{center}
\begin{tabular}{| c || r|r|r||c|c|c|}
\hline
Ideal & \multicolumn{1}{c|}{$n_i$} & \multicolumn{1}{c|}{$k_i$} & \multicolumn{1}{c||}{$l_i$} & Method & $|B(I)|$ & $|\PSL(2,O_d/I)|$ \\ \hline \hline
$\myLangle 2,\frac{1+\sqrtMinusDSymbolTable{47}}{2}\myRangle$ &
2 & 1 & 1 &
\multicolumn{2}{c|}{15-Link} & 
6 \\ 
& 2 & 1 & 1 & \multicolumn{2}{c|}{} & \\ 
& 1 & 0 & 2 & \multicolumn{2}{c|}{} & \\
& 1 & 0 & 2 & \multicolumn{2}{c|}{} & \\ 
& 2 & 0 & 1 & \multicolumn{2}{c|}{} & \\ \hline
$\myLangle 3,\frac{1+\sqrtMinusDSymbolTable{47}}{2}\myRangle$ &
3 & 1 & 1 &
\multicolumn{2}{c|}{20-Link} & 
12 \\ 
& 1 & 0 & 3 & \multicolumn{2}{c|}{} & \\ 
& 3 & 1 & -1 & \multicolumn{2}{c|}{} & \\ 
& 1 & 0 & 3 & \multicolumn{2}{c|}{} & \\ 
& 3 & 1 & 1 & \multicolumn{2}{c|}{} & \\  \hline
$\myLangle 4,\frac{1+\sqrtMinusDSymbolTable{47}}{2}\myRangle$ &
4 & 1 & 1 &
\multicolumn{2}{c|}{30-Link} & 
24 \\ 
& 4 & 1 & 1 & \multicolumn{2}{c|}{} & \\ 
& 1 & 0 & 4 & \multicolumn{2}{c|}{} & \\ 
& 2 & 1 & 2 & \multicolumn{2}{c|}{} & \\ 
& 4 & 2 & 1 & \multicolumn{2}{c|}{} & \\ \hline
$\myLangle 2\myRangle$ &
2 & 0 & 2 &
\texttt{LowerBound1(\dots,2)} & $\infty$ &
36 \\ 
& 2 & 0 & 2 & & & \\
& 2 & 0 & 2 & & & \\
& 2 & 0 & 2 & & & \\
& 2 & 0 & 2 & & \phantom{$\geq 6.174\cdot 10^{9}$} & \\ \hline
$\myLangle 6,\frac{1+\sqrtMinusDSymbolTable{47}}{2}\myRangle$ &
6 & 1 & 1 &
\texttt{LowerBound1} & $\infty$ &
72 \\ 
& 2 &1 & 3 && & \\
& 3 &1 & 2 && & \\
& 1 &0 & 6 && & \\
\phantom{$\myLangle \frac{9+\sqrtMinusDSymbolTable{47}}{2}, 7-\sqrtMinusDSymbolTable{47}\myRangle$}
& 6 &-2 & 1 && & \\
\hline
\end{tabular}
\end{center}
\end{table}

\begin{table}[h]
\caption{Computations for $d=47$ (continued).\label{tbl:results47b}}
\begin{center}
\begin{tabular}{| c || r|r|r||c|c|c|}
\hline
Ideal & \multicolumn{1}{c|}{$n_i$} & \multicolumn{1}{c|}{$k_i$} & \multicolumn{1}{c||}{$l_i$} & Method & $|B(I)|$ & $|\PSL(2,O_d/I)|$ \\ \hline \hline
$\myLangle 6,\frac{5+\sqrtMinusDSymbolTable{47}}{2}\myRangle = $ &
6 & 3 & 1 &
\texttt{LowerBound1} & $\geq 6.174\cdot 10^{9}$ &
72 \\
$\myLangle \frac{5+\sqrtMinusDSymbolTable{47}}{2}, \frac{7-\sqrtMinusDSymbolTable{47}}{2}\myRangle$  & 6 &-1 &1 & & &  \\ 
& 3 & 0 & 2 &&& \\
& 3 & 1 & -2 &&& \\
& 2 & 0 & 3 &&&\\
\hline
$\myLangle 4,1+\sqrtMinusDSymbolTable{47}\myRangle$ &
 &  &  &
$\containedInI\myLangle 2\myRangle$ & $\infty$ &
144 \\ \hline
$\myLangle 7,\frac{3+\sqrtMinusDSymbolTable{47}}{2}\myRangle = $ &
7 & 2 & 1 &
\texttt{LowerBound1} & $\geq 7.980\cdot 10^{11}$ &
168 \\
$\myLangle \frac{3+\sqrtMinusDSymbolTable{47}}{2}, \frac{11-\sqrtMinusDSymbolTable{47}}{2}\myRangle$  & 7 & -3 & 1 & & &  \\ 
& 7 &3 &1&&&\\
& 7 & 2 & 1 &&&\\
& 7 & -1 & 1 &&&\\ \hline
\markConjugate{$\myLangle 8,\frac{7-\sqrtMinusDSymbolTable{47}}{2}\myRangle = $} &
8 & -3 & 1 &
\texttt{LowerBound1} & $\geq 8.961\cdot 10^{26}$ &
192 \\
$\myLangle \frac{7-\sqrtMinusDSymbolTable{47}}{2}, \frac{9+\sqrtMinusDSymbolTable{47}}{2}\myRangle$  & 8 & 1 & 1 & & &  \\
& 2 & -1 & 4 &&&\\
& 4 & 1 & 2 &&&\\
& 8 & -2 & 1 &&&\\ \hline
$\myLangle 3\myRangle$ &
3 & 0 & 3 &
\texttt{LowerBound1(...,3)} & $\infty$ &
288 \\ 
& 3 & 0 & 3 &&&\\
& 3 & 0 & 3 &&&\\
& 3 & 0 & 3 &&&\\
& 3 & 0 & 3 &&&\\ \hline
$\myLangle 9,\frac{5+\sqrtMinusDSymbolTable{47}}{2}\myRangle = $ &
9 & 3 & 1 &
\texttt{LowerBound1} & $\geq 2.403\cdot 10^{43}$ &
324 \\
$\myLangle \frac{5+\sqrtMinusDSymbolTable{47}}{2}, \frac{13-\sqrtMinusDSymbolTable{47}}{2}\myRangle$  & 9 & 2 & 1 & & &  \\ 
& 9 & -3 & 1 &&&\\
&9 & 1 & 1 &&&\\
& 3 & 1 & 3 &&&\\ \hline
$\myLangle 6,1+\sqrtMinusDSymbolTable{47}\myRangle$ &
 &  &  &
$\containedInI\myLangle 2\myRangle$ & $\infty$ &
432 \\ \hline
$\myLangle \frac{1+\sqrtMinusDSymbolTable{47}}{2}\myRangle$ &
 &  &  &
$\containedInI\myLangle 6,\frac{1+\sqrtMinusDSymbolTable{47}}{2}\myRangle$ & $\infty$ &
576 \\ \hline
$\myLangle 12,\frac{7+\sqrtMinusDSymbolTable{47}}{2}\myRangle = $ &
 &  &  &
$\containedInI\myLangle 6,\frac{5-\sqrtMinusDSymbolTable{47}}{2}\myRangle$ & $\geq 6.174\cdot 10^{9}$ &
576 \\
$\myLangle \frac{7+\sqrtMinusDSymbolTable{47}}{2}, 5-\sqrtMinusDSymbolTable{47}\myRangle$  & & & & & &  \\ \hline
$\myLangle \frac{3+\sqrtMinusDSymbolTable{47}}{2}\myRangle$ &
 &  &  &
$\containedInI\myLangle 7,\frac{3+\sqrtMinusDSymbolTable{47}}{2}\myRangle$ & $\geq 7.980\cdot 10^{11}$ &
1008 \\ \hline
$\myLangle 4\myRangle$ &
 &  &  &
$\containedInI\myLangle 2\myRangle$ & $\infty$ &
1152 \\ \hline
$\myLangle 16,\frac{9+\sqrtMinusDSymbolTable{47}}{2}\myRangle = $ &
 &  &  &
$\containedInI\myLangle 8,\frac{7-\sqrtMinusDSymbolTable{47}}{2}\myRangle$ & $\geq 8.961\cdot 10^{26}$ &
1536 \\
$\myLangle \frac{9+\sqrtMinusDSymbolTable{47}}{2}, 7-\sqrtMinusDSymbolTable{47}\myRangle$  & & & & & &  \\ \hline
$\myLangle 6,\frac{3+3 \sqrtMinusDSymbolTable{47}}{2}\myRangle$ &
 &  &  &
$\containedInI\myLangle 3\myRangle$ & $\infty$ &
1728 \\ \hline
$\myLangle \frac{5+\sqrtMinusDSymbolTable{47}}{2}\myRangle$ &
 &  &  &
$\containedInI\myLangle 6,\frac{5+\sqrtMinusDSymbolTable{47}}{2}\myRangle$ & $\geq 6.174\cdot 10^{9}$ &
1944 \\ \hline
$\myLangle \frac{7+\sqrtMinusDSymbolTable{47}}{2}\myRangle$ &
 &  &  &
$\containedInI\myLangle 8,\frac{7+\sqrtMinusDSymbolTable{47}}{2}\myRangle$ & $\geq 8.961\cdot 10^{26}$ &
4608 \\ \hline
$\myLangle 5\myRangle$ &
5 & 0 & 5 &
\texttt{LowerBound1(...,5)} & $\infty$ &
7800 \\
& 5 &0 & 5 &&& \\
& 5 &0 & 5 &&& \\
& 5 &0 & 5 &&& \\
& 5 &0 & 5 &&& \\ \hline
$\myLangle 6\myRangle$ &
 &  &  &
$\containedInI\myLangle 2\myRangle$ & $\infty$ &
10368 \\ \hline
$\myLangle \frac{9+\sqrtMinusDSymbolTable{47}}{2}\myRangle$ &
 &  &  &
$\containedInI\myLangle 8,\frac{7-\sqrtMinusDSymbolTable{47}}{2}\myRangle$ & $\geq 8.961\cdot 10^{26}$ &
12288 \\ \hline
\end{tabular}
\end{center}
\explainConjugate{fig:disc47}
\end{table}

\clearpage

\begin{table}[h]
\caption{Magma function to produce $B(I)$ for $d=71$.\label{tbl:magma71}}
\begin{center}
\begin{minipage}{16cm}
\begin{verbatim}
function B71(n1,k1,l1,  n2,k2,l2, n3,k3,l3, n4,k4,l4, n5,k5,l5, n6,k6,l6, n7,k7,l7)
    return Group<g1,g2,g3,g4,g5,g6,g7,g8,g9|
        g8^3,(g8^-1,g4),(g8*g7^-1)^2,g1^-1*g3*g7*g3^-1*g1*g7^-1,
        g6*g3*g6^-1*g7*g9^-1*g3^-1*g9*g7^-1,
        g7^-1*g6*g3*g6^-1*g5^-1*g2*g7*g5*g6*g3^-1*g6^-1*g2^-1,
        g8*g7^-1*g1*g5*g6*g3^-1*g1*g5*g7*g8^-1*g5^-1*g1^-1*g3*g6^-1*g5^-1*g1^-1,
        g4^-1*g7^-1*g5^-1*g2*g1^-1*g3*g7*g9*g4*g1*g7^-1*g2^-1*g5*g7*g9^-1*g3^-1,
        (g5*g8*g7^-1*g5^-1*g1^-1*g7*g9*g6*g1*g5*g8*g7^-1*
         g5^-1*g1^-1*g3*g6^-1*g9^-1*g7^-1*g3^-1*g1),
        (g2*g6*g1*g5*g7*g8^-1*g5^-1*g1^-1*g3*g6^-1*g7*g8^-1*g5^-1*
         g2^-1*g5*g7*g8^-1*g5^-1*g1^-1*g7*g8^-1*g1*g5*g6*g3^-1*g6^-1),

        g7^n1,
        g7^k1*(g1^-1*g3)^l1,
        
        g2^n2,
        g2^k2*(g6*g1*g5*g7*g8^-1*g5^-1*g1^-1*g3*g6^-1*g7*g8^-1*g5^-1)^l2,
        
        g3^n3,
        g3^k3*(g6^-1*g7*g9^-1)^l3,
        
        (g7*g2)^n4,
        (g7*g2)^k4*(g6*g3*g6^-1*g5^-1*g7^-1)^l4,
        
        (g7*g9*g6)^n5,
        (g7*g9*g6)^k5*(g3^-1*g1*g5*g8*g7^-1*g5^-1*g1^-1)^l5,
        
        (g3*g9*g4)^n6,
        (g3*g9*g4)^k6*(g4^-1*g7^-1*g5^-1*g2*g7*g1^-1)^l6,

        (g4*g1*g7^-1*g2^-1*g5*g7*g9^-1*g3^-1*g8^-1*g4^-1)^n7,
        ((g4*g1*g7^-1*g2^-1*g5*g7*g9^-1*g3^-1*g8^-1*g4^-1)^k7 *
         (g6*g3^-1*g1*g5*g8*g7^-1*g5^-1*g1^-1*g6^-1*g9^-1*g8^-1*g4^-1)^l7)>;
end function;
\end{verbatim}
\end{minipage}
\end{center}
\end{table}

\begin{table}[h]
\caption{Computations for $d=71$ (compare to Figure~\ref{fig:disc71}).\label{tbl:results71}}
\begin{center}
\begin{tabular}{| c || r|r|r||c|c|c|}
\hline
Ideal & \multicolumn{1}{c|}{$n_i$} & \multicolumn{1}{c|}{$k_i$} & \multicolumn{1}{c||}{$l_i$} & Method & $|B(I)|$ & $|\PSL(2,O_d/I)|$ \\ \hline \hline
$\myLangle 2,\frac{1+\sqrtMinusDSymbolTable{71}}{2}\myRangle$ &
2 & 1 & 1 &
\multicolumn{2}{c|}{21-Link} &
6 \\ 
& 2 & 0 & 1 &\multicolumn{2}{c|}{}&\\
& 2 & 1 & 1 &\multicolumn{2}{c|}{}&\\
& 1 & 0 & 2 &\multicolumn{2}{c|}{}&\\
& 2 & 0 & 1 &\multicolumn{2}{c|}{}&\\
& 2 & 0 & 1 &\multicolumn{2}{c|}{}&\\
& 2 & 0 & 1 &\multicolumn{2}{c|}{}&\\ \hline
$\myLangle 3,\frac{1+\sqrtMinusDSymbolTable{71}}{2}\myRangle$ &
3 & 1 & 1 &
\texttt{LowerBound1} & $\infty$ &
12 \\
& 3 & -1 & 1 &&&\\
& 3 & 0 & 1 &&&\\
& 3 & 0 & 1 &&&\\
& 3 & 1 & 1 &&&\\
& 3 & 1 & 1 &&&\\
& 3 & 0 & 1 &&&\\ \hline
\markConjugate{$\myLangle 4,\frac{3-\sqrtMinusDSymbolTable{71}}{2}\myRangle$} &
4 & -1 & 1 &
\texttt{LowerBound1} & $\geq 862488$ &
24 \\ 
&4&2&1 &&&\\
&4&1&1&&&\\
&1&0&4&&&\\
&4&2&1 &&&\\
&4&0&1&&&\\
&4&2&-1 &&&\\ \hline
$\myLangle 2\myRangle$ &
2 & 0 & 2 &
\texttt{Lowerbound1(...,2)} & $\infty$ &
36 \\ 
&2&0&2&&&\\
&2&0&2&&&\\
&2&0&2&&&\\
&2&0&2&&&\\
&2&0&2&&&\\
&2&0&2&&&\\ \hline
$\myLangle 5,\frac{3+\sqrtMinusDSymbolTable{71}}{2}\myRangle$ &
5 & 2 & 1 &
\texttt{LowerBound1} & $\geq 2.929\cdot 10^9$ &
60 \\
&5&-1&1&&&\\
&5&2&-1&&&\\
&5&1&1&&&\\
&5&-1&1&&&\\
&5&2&1&&&\\
&5&2&1&&&\\ \hline
$\myLangle 6,\frac{1+\sqrtMinusDSymbolTable{71}}{2}\myRangle$ &
 &  &  &
$\containedInI\myLangle 3,\frac{1+\sqrtMinusDSymbolTable{71}}{2}\myRangle$ & $\infty$ &
72 \\ \hline
$\myLangle 6,\frac{5+\sqrtMinusDSymbolTable{71}}{2}\myRangle = $ &
&  &  &
$\containedInI\myLangle 3,\frac{1-\sqrtMinusDSymbolTable{71}}{2}\myRangle$ & $\infty$ &
72 \\
$\myLangle \frac{5+\sqrtMinusDSymbolTable{71}}{2}, \frac{7-\sqrtMinusDSymbolTable{71}}{2}\myRangle$  & & & & & &  \\ \hline
$\myLangle 4,1+\sqrtMinusDSymbolTable{71}\myRangle$ &
 &  &  &
$\containedInI\myLangle 2\myRangle$ & $\infty$ &
144 \\ \hline
$\myLangle 8,\frac{5+\sqrtMinusDSymbolTable{71}}{2}\myRangle = $ &
 &  &  &
$\containedInI\myLangle 4,\frac{3-\sqrtMinusDSymbolTable{71}}{2}\myRangle$ & $\geq 862488$ &
192 \\
$\myLangle \frac{5+\sqrtMinusDSymbolTable{71}}{2}, \frac{11-\sqrtMinusDSymbolTable{71}}{2}\myRangle$  & & & & & &  \\ \hline
$\myLangle 3\myRangle$ &
 &  &  &
$\containedInI\myLangle 3,\frac{1+\sqrtMinusDSymbolTable{71}}{2}\myRangle$ & $\infty$  &
288 \\ \hline
\end{tabular}
\end{center}
\explainConjugate{fig:disc71}
\end{table}

\begin{table}[h]
\caption{Computations for $d=71$ (continued).\label{tbl:results71b}}
\begin{center}
\begin{tabular}{| c || r|r|r||c|c|c|}
\hline
Ideal & \multicolumn{1}{c|}{$n_i$} & \multicolumn{1}{c|}{$k_i$} & \multicolumn{1}{c||}{$l_i$} & Method & $|B(I)|$ & $|\PSL(2,O_d/I)|$ \\ \hline \hline
$\myLangle 9,\frac{1+\sqrtMinusDSymbolTable{71}}{2}\myRangle$ &
 &  &  &
$\containedInI\myLangle 3,\frac{1+\sqrtMinusDSymbolTable{71}}{2}\myRangle$ & $\infty$ &
324 \\ \hline
$\myLangle 10,\frac{3+\sqrtMinusDSymbolTable{71}}{2}\myRangle = $ &
 &  &  &
$\containedInI\myLangle 5,\frac{3+\sqrtMinusDSymbolTable{71}}{2}\myRangle$ & $\geq 2.929\cdot 10^9$ &
360 \\
$\myLangle \frac{3+\sqrtMinusDSymbolTable{71}}{2}, \frac{17-\sqrtMinusDSymbolTable{71}}{2}\myRangle$  & & & & & &  \\ \hline
$\myLangle 10,\frac{7+\sqrtMinusDSymbolTable{71}}{2}\myRangle = $ &
 &  &  &
$\containedInI\myLangle 5,\frac{3-\sqrtMinusDSymbolTable{71}}{2}\myRangle$ & $\geq 2.929\cdot 10^9$ &
360 \\
$\myLangle \frac{7+\sqrtMinusDSymbolTable{71}}{2}, \frac{13-\sqrtMinusDSymbolTable{71}}{2}\myRangle$  & & & & & &  \\ \hline
$\myLangle 6,1+\sqrtMinusDSymbolTable{71}\myRangle$ &
&  &  &
$\containedInI\myLangle 2\myRangle$ & $\infty$ &
432 \\ \hline
$\myLangle 12,\frac{5+\sqrtMinusDSymbolTable{71}}{2}\myRangle = $ &
 &  &  &
$\containedInI\myLangle 4,\frac{3-\sqrtMinusDSymbolTable{71}}{2}\myRangle$ & $\geq 862488$ &
576 \\
$\myLangle \frac{5+\sqrtMinusDSymbolTable{71}}{2}, \frac{19-\sqrtMinusDSymbolTable{71}}{2}\myRangle$  & & & & & &  \\ \hline
$\myLangle 4\myRangle$ &
&  &  &
$\containedInI\myLangle 2\myRangle$ & $\infty$ &
1152 \\ \hline
$\myLangle 15,\frac{7+\sqrtMinusDSymbolTable{71}}{2}\myRangle = $ &
&&&
$\containedInI\myLangle 5,\frac{3-\sqrtMinusDSymbolTable{71}}{2}\myRangle$ & $\geq 2.929\cdot 10^9$ &
1440 \\
$\myLangle \frac{7+\sqrtMinusDSymbolTable{71}}{2}, 8-\sqrtMinusDSymbolTable{71}\myRangle$  & & & & & &  \\ \hline
$\myLangle 6,\frac{3+3 \sqrtMinusDSymbolTable{71}}{2}\myRangle$ &
 &  &  &
$\containedInI\myLangle 3,\frac{1+\sqrtMinusDSymbolTable{71}}{2}\myRangle$ & $\infty$  &
1728 \\ \hline
$\myLangle \frac{1+\sqrtMinusDSymbolTable{71}}{2}\myRangle$ &
 &  &  &
$\containedInI\myLangle 3,\frac{1+\sqrtMinusDSymbolTable{71}}{2}\myRangle$ & $\infty$ &
1944 \\ \hline
$\myLangle \frac{3+\sqrtMinusDSymbolTable{71}}{2}\myRangle$ &
 &  &  &
$\containedInI\myLangle 5,\frac{3+\sqrtMinusDSymbolTable{71}}{2}\myRangle$ & $\geq 2.929\cdot 10^9$ &
2880 \\ \hline
$\myLangle 19,\frac{9+\sqrtMinusDSymbolTable{71}}{2}\myRangle = $ &
19 & 5 & 1 &
\texttt{LowerBound2(...,3420,7)} & $\geq 2.782\cdot 10^{57}$ &
3420 \\
$\myLangle \frac{9+\sqrtMinusDSymbolTable{71}}{2}, 10-\sqrtMinusDSymbolTable{71}\myRangle$  & 19 & 9 & 1& & &  \\ 
& 19 & 2 & 1 &&&\\
& 19 & 3 & -1 &&&\\
& 19 & 8 & -1 &&&\\
& 19 &8 & 1 &&&\\
& 19 & 5 & -1 &&&\\ \hline
$\myLangle \frac{5+\sqrtMinusDSymbolTable{71}}{2}\myRangle$ &
&  &  &
$\containedInI\myLangle 4,\frac{3-\sqrtMinusDSymbolTable{71}}{2}\myRangle$ & $\geq 862488$ &
4608 \\ \hline
$\myLangle 5\myRangle$ &
&&&
$\containedInI\myLangle 5,\frac{3+\sqrtMinusDSymbolTable{71}}{2}\myRangle$ & $\geq 2.929\cdot 10^9$ &
7200 \\ \hline
$\myLangle \frac{7+\sqrtMinusDSymbolTable{71}}{2}\myRangle$ &
&&&
$\containedInI\myLangle 5,\frac{3-\sqrtMinusDSymbolTable{71}}{2}\myRangle$ & $\geq 2.929\cdot 10^9$ &
8640 \\ \hline
$\myLangle 6\myRangle$ &
 &  &  &
$\containedInI\myLangle 2\myRangle$ & $\infty$ &
10368 \\ \hline
$\myLangle \frac{9+\sqrtMinusDSymbolTable{71}}{2}\myRangle$ &
 &  &  &
$\containedInI\myLangle 19,\frac{9+\sqrtMinusDSymbolTable{71}}{2}\myRangle$ & $\geq 2.782\cdot 10^{57}$ &
20520 \\ \hline
\end{tabular}
\end{center}
\end{table}

\clearpage

\section{Corrections from  \cite{bakerReid:prinCong} and \cite{bakerReid:higherPrinCong}} For convenience we record some corrections to  \cite{bakerReid:prinCong} and \cite{bakerReid:higherPrinCong} that were uncovered whilst preparing this technical report.  None of these affect
the results contained in  \cite{bakerReid:prinCong} and \cite{bakerReid:higherPrinCong}.

Firstly, in  \cite{bakerReid:prinCong} we note that on page 1081, there are typos in the Magma routine for $\Gamma(\myLangle (1+i)^3 \myRangle)$. The 7th and 12th relations should read
\begin{verbatim} u^-1*a*t^-2*u^2*a*u, a*t^-2*a*t^-2*u^2*a*t^2*a\end{verbatim} respectively.

Similarly on page 1084, there is a typo in the Magma routine for $\Gamma(\myLangle (3+\sqrt{-11})/2 \myRangle)$.  In definition of the subgroup $m$, the last relation should read 
\begin{verbatim}t^2*a*t^-2*a*(t*u)*a*t^2*a*t^-2\end{verbatim}

Turning to \cite{bakerReid:higherPrinCong}, in the Magma routine for the case of $(15,\myLangle 4, \omega_{15}\myRangle)$, the group $Q$ should read:
\begin{verbatim}Q:=quo<A|t^4*u, t*a*(t^4*u)*a*t^-1, (t*a)^-1*(t^4*u)*t*a, u*c*a, t*a*(u*c*a)*(t*a)^-1, \end{verbatim}
\begin{verbatim}(t*a)^-1*(u*c*a)*t*a, t^2*(t*a)^-1*u*t*a*t^-2, t^2*(u*c*a)*t^-2>;\end{verbatim}

In Tables 4, 5 and 6 of \cite{bakerReid:higherPrinCong}, the following peripheral subgroups should be replaced by those shown in Tables 21, 23 and 25:

\noindent In Table 4, the second and third peripheral subgroups for the entry $\myLangle 13, 4+\omega_{23} \myRangle$.

\smallskip

\noindent In Table 5, the first peripheral subgroup for the entry $\myLangle 19, 5+\omega_{31} \myRangle$.

\smallskip 

\noindent In Table 6, the first peripheral subgroup for the entry $\myLangle 2, \omega_{39} \myRangle^2$ and the
the second, third and fourth peripheral subgroups for the entry $\myLangle 2, \omega_{39} \myRangle^3$. Also note that in this latter case the order of $B_d(I)$ should be recorded as $>>1$.

\smallskip

\noindent In addition, there were omissions in Tables 4 and 5 of \cite{bakerReid:higherPrinCong}, namely the ideals $\myLangle 6,-3+\omega_{23}\myRangle$, 
$\myLangle (5+\sqrt{-23})/2 \myRangle$ and $\myLangle 10, (7+\sqrt{-31})/2 \myRangle$. That these do not give link complements are recorded in Tables 21 and 23.

In Appendix A of \cite{bakerReid:higherPrinCong}, the list of matrix generators for $\PSL(2,O_{47})$ contains some typos:  the $(1,2)$-entry of $g_5$ is incorrect, it should be $-3+3\omega_{47}$. Similarly the $(1,1)$-entry of $g_7$ should be $1-2\omega_{47}$.

Finally in Appendix B, the cases of $5$ splitting were not recorded in the the cases of $d=39$ and $d=71$.

\clearpage

\section{Link complements proofs}

In this section, we explain how the functions defined in the Magma file \href{http://unhyperbolic.org/prinCong/prinCong/magma/LinkComplementHelpers.m}{\texttt{LinkComplementHelpers.m}} work and can be used to prove that a principal congruence manifold is a link complement. We provide Magma files to show this for all the 48 principal congruence manifolds in question at \cite[\href{http://unhyperbolic.org/prinCong/prinCong/magma/}{\texttt{prinCong/magma/}}]{goerner:data}. For example, to check all $d=2$ cases, one can run the following shell command:
\begin{center}
\begin{minipage}{14cm}
\begin{verbatim}
magma LinkComplementHelpers.m LinkComplement2.m
\end{verbatim}
\end{minipage}
\end{center}
which produces the following output:
\begin{center}
\begin{minipage}{14cm}
\begin{verbatim}
<1+sqrt(-2)>
4-Link
<2>
12-Link
...
\end{verbatim}
\end{minipage}
\end{center}

\subsection{Class number one}

We use $(2,\myLangle 1+\sqrt{-2}\myRangle)$ as an example how to verify a principal congruence link complement $M=\H^3/\Gamma(I)$ using the function \texttt{VerifyLink}:
\begin{center}
\begin{minipage}{13cm}
\begin{verbatim}
// Presentation of Bianchi group
Bianchi2<a,t,u> := Group<a,t,u|a^2,(t*a)^3,(a*u^-1*a*u)^2,(t,u)>;
// Parabolic elements fixing cusp of the Bianchi Bianchi orbifold.
Bianchi2P := [[ t, u ]];

VerifyLink(
    Bianchi2, Bianchi2P,
    [[3,1,1]], 12,
    [[ <Id(Bianchi2), [ 0, 1]>,
       <a,            [ 0, 1]>,
       <t*a,          [ 0, 1]>,
       <t^-1*a,       [-1, 1]>]]);
\end{verbatim}
\end{minipage}
\end{center}
This code will output \texttt{4-Link} confirming that all necessary tests have passed and $M$ is a link complement. The first two arguments to the function \texttt{VerifyLink} are the finitely presented Bianchi group and the elements generating $P_\infty$ from Section~\ref{sec:presentationBianchi}. The next argument is the triple $(n,k,l)$ from Section~\ref{sec:findBIClassNumberOne} and the size of $\PSL(2,O_d)$. Next is a list of pairs $(g, (p,q))$ of an element $g$ in the Bianchi group specifying a cusp and Dehn-filling coefficients $(p,q)$.

\newcounter{verifyLinkTests}

The function will use these arguments to perform three tests:
\begin{enumerate}
\item \texttt{VerifyLink} checks whether $|B(I)| = |\PSL(2,O_d/I)|$ using a presentation of $B(I)$ similar to the one in, e.g., Table~\ref{tbl:magma2}.
\setcounter{verifyLinkTests}{\value{enumi}}
\end{enumerate}
Recall from Section~\ref{sec:Prelims} that this implies that $N(I)$ is the fundamental group of $M=\H^3/\Gamma(I)$. By Perelman's Theorem, it suffices to find Dehn-fillings of $M$ trivializing the fundamental group. This is equivalent to finding a set of primitive elements in $N(I)$ generating $N(I)$ such that each element is conjugate to an element in $P_\infty$ and corresponds to a different cusp of $M$. These primitive elements are specified as last argument to \texttt{VerifyLink}, namely, a pair $(g, (p,q))$ presents the element $g (t^n)^p (t^ku^l)^q g^{-1}$. For example \texttt{<a, [ 0, 1]>} yields \verb|a*(t*u)*a^-1|, an element in \texttt{Bianchi2} which is also in $N(I)\subset\PSL(2,O_d)$.
\begin{enumerate}
\setcounter{enumi}{\value{verifyLinkTests}}
\item  \texttt{VerifyLink} checks that the elements specified this way are indeed generating $N(I)$ or, equivalently, the quotient of $N(I)$ by the subgroup generated by these elements is trivial. In our example, the function would perform a test equivalent to checking that the following code returns 1:
\begin{center}
\begin{minipage}{13cm}
\begin{verbatim}
N := NormalClosure(Bianchi2, sub<Bianchi2 | t^3, t*u>);
Q := quo<N | (t*u),
             a*(t*u)*a^-1,
             t*a*(t*u)*a^-1*t^-1,
             t^-1*a*(t^-2*u)*a^-1*t>;
Order(Q);
\end{verbatim}
\end{minipage}
\end{center}
\item \texttt{VerifyLink} checks that the elements $g\in \PSL(2,O_d)$ actually correspond to different cusps in $M$ where an element $g$ corresponds to the cusp of $M$ that is represented by the point in $\C P^1$ fixed by $gP_\infty g^{-1}$.\\
Consider the map $\pi:\PSL(2,O_d)\to \PSL(2,O_d/I)\cong B(I)$. If $d\not=1,3$, two elements $g$ and $g'\in \PSL(2,O_d)$ correspond to the same cusp if and only if their images $\pi(g)$ and $\pi(g')$ yield the the same left cosets $\pi(g)\pi(P_\infty)$ and $\pi(g')\pi(P_\infty)$ in $\PSL(2,O_d/I)/\pi(P_\infty)$. If $d=1,3$, we need to check whether they are the same left coset in $\PSL(2,O_d/I)/\myLangle l,t,u\myRangle$. The test in \texttt{VerifyLink} first checks  that the number of given pairs $(g, (p,q))$ is equal to $|B(I)|/|\pi(P_\infty)|$. It then checks that the union of all cosets $g\pi(P_\infty)$ has size $|B(I)|$.
\end{enumerate}

\subsection{Higher class numbers} For higher number class, verifying $M=\H^3/\Gamma(I)$ works similarly except that we need to specify a list of tuples $(n_i,k_i,l_i)$ as in Section~\ref{sec:findBIHigherClassNumber} and need to give a list of list of Dehn-fillings, each list specifying the Dehn-fillings of all the cusps of $M$ corresponding to the same Bianchi orbifold cusp. The following code illustrates this for $(15,\myLangle 2, \frac{1+\sqrt{-15}}{2}\myRangle)$, see Tables~\ref{tbl:magma15} and \ref{tbl:results15}:
\begin{center}
\begin{minipage}{15.3cm}
\begin{verbatim}
Bianchi15<a,c,t,u> := Group<a,c,t,u|(t,u),(a,c),a^2,(t*a)^3,
                                    u*c*u*a*t*u^-1*c^-1*u^-1*a*t^-1>;
Bianchi15P := [[t, u], [u*c*a, c^-1*a*u^-1*c^-1*u^-1*t*a]];

VerifyLink(
    Bianchi15, Bianchi15P,
    [[2,0,1],[1,0,2]], 6,
    [[<Id(Bianchi15), [ 1, 0]>,     // gives t^2
      <(t*a),         [ 0, 1]>,     // gives (t*a)*u*(t*a)^-1
      <(t*a)^2,       [ 0, 1]>],    // gives (t*a)^2*u*(t*a)^-2
     [<Id(Bianchi15), [ 1, 0]>,     // gives u*c*a
      <(t*a),         [ 1, 0]>,     // gives (t*a)*(u*c*a)*(t*a)^-1
      <(t*a)^2,       [ 0, 1]>]]);  // (t*a)^2*(c^-1*a*u^-1*c^-1*u^-1*t*a)^2*(t*a)^-2
\end{verbatim}
\end{minipage}
\end{center}
An element $(g,(p,q))$ in the $i$-th list of Dehn fillings now corresponds to $g (p_{(i),1}^{n_i})^{p} (p_{(i),1}^{k_i} p_{(i),2}^{l_i})^q g^{-1}\in N(I)$. The same three tests are performed. The test whether the different $g$ belong to different cusps of $M$ is done separately for each Bianchi orbifold cusp, inspecting the respective list of Dehn-fillings and using the respective $P_{(i)}$-cosets.

\subsection{The \texttt{Expand} helper}

Note that the elements $g$ in the above Magma example for $(15,\myLangle 2, \frac{1+\sqrt{-15}}{2}\myRangle)$ repeat. To abbreviate the data we need to give, we can use the \texttt{Expand} helper requiring to specify the list of $g$'s only once as first argument:
\begin{center}
\begin{minipage}{15.3cm}
\begin{verbatim}
VerifyLink(
    Bianchi15, Bianchi15P,
    [[2,0,1],[1,0,2]], 6,
    Expand([Id(Bianchi15), t*a, (t*a)^2],
           [ [ [ 1, 0], [ 0, 1], [ 0, 1] ],
             [ [ 1, 0], [ 1, 0], [ 0, 1] ] ]));
\end{verbatim}
\end{minipage}
\end{center}
The \texttt{Expand} helper can be used in most but not all higher class number cases, e.g., it cannot be used for $(47,\myLangle\frac{1+\sqrt{-47}}{2}\myRangle)$.

\clearpage

\subsection{Links with symmetries acting freely on cusps} In the following example for $(11,\myLangle \frac{3+\sqrt{-11}}{2}\myRangle)$, note that the Dehn-filling coefficients are invariant under the action of $a$:
\begin{center}
\begin{minipage}{15.3cm}
\begin{verbatim}
Bianchi11<a,t,u> := Group<a,t,u|a^2,(t*a)^3,(a*t*u^-1*a*u)^3,(t,u)>;
Bianchi11P := [[ t, u ]];

VerifyLink(
    Bianchi11, Bianchi11P,
    [[5,1,1]], 60,
    [[ <   Id(Bianchi11),    [ 0, 1]>,
       <   t*a,              [ 0, 1]>,
       <   t^2*a,            [-1, 2]>,
       <   t^-2*a^-1,        [ 0, 1]>,
       <   t*a*t^-2*a,       [ 0, 1]>,
       <   t^2*a*t^-2*a^-1,  [ 0, 1]>,
       
       <a* Id(Bianchi11),    [ 0, 1]>,
       <a* t*a,              [ 0, 1]>,
       <a* t^2*a,            [-1, 2]>,
       <a* t^-2*a^-1,        [ 0, 1]>,
       <a* t*a*t^-2*a,       [ 0, 1]>,
       <a* t^2*a*t^-2*a^-1,  [ 0, 1]>]]);
\end{verbatim}
\end{minipage}
\end{center}
This means that the order-2 symmetry $a$ of the principal congruence manifold $M$ extends to the link. The symmetry is swapping the cusps pairwise fixing none. Similarly to \texttt{Expand}, the function \texttt{Symmetrize} can be used to abbreviate the above example to:
\begin{center}
\begin{minipage}{15.3cm}
\begin{verbatim}
VerifyLink(
    Bianchi11, Bianchi11P,
    [[5,1,1]], 60,
    Symmetrize(
        a, 2,   // Symmetry and its order
        [[]],
        [[<Id(Bianchi11),    [ 0, 1]>,
          <t*a,              [ 0, 1]>,
          <t^2*a,            [-1, 2]>,
          <t^-2*a^-1,        [ 0, 1]>,
          <t*a*t^-2*a,       [ 0, 1]>,
          <t^2*a*t^-2*a^-1,  [ 0, 1]>]]));
\end{verbatim}
\end{minipage}
\end{center}

\clearpage

\subsection{Symmetric links}
Since some principal congruence links have symmetries fixing some cusps, \texttt{Symmetrize} takes two lists for specifying Dehn-fillings on cusps. The first list is for the cusps which are fixed and the second one for those which are not fixed. Here is $(15,\myLangle\frac{1+\sqrt{-15}}{2}\myRangle)$ as an example with class number $h_d=2$. The code
\begin{center}
\begin{minipage}{15.3cm}
\begin{verbatim}
Bianchi15<a,c,t,u> := Group<a,c,t,u|(t,u),(a,c),a^2,(t*a)^3,
                                    u*c*u*a*t*u^-1*c^-1*u^-1*a*t^-1>;
Bianchi15P := [[t, u], [u*c*a, c^-1*a*u^-1*c^-1*u^-1*t*a]];

VerifyLink(
    Bianchi15, Bianchi15P,
    [[4,0,1],[1,0,4]], 24,
    [[ <(t*a)^-1,                   [ 1, 1]>,  // fixed by symmetry a*t^2*a
       <t^2*(t*a)^-1,               [ 0, 1]>,  // fixed by symmetry a*t^2*a
           
       <Id(Bianchi15),              [ 1, 1]>,
       <t*a,                        [ 1, 1]>,
           
       <(a*t^2*a)*  Id(Bianchi15),  [ 1, 1]>,
       <(a*t^2*a)*  t*a,            [ 1, 1]>],
     [ <Id(Bianchi15),              [ 1, 0]>,  // fixed by symmetry a*t^2*a
       <t^2,                        [ 1, 0]>,  // fixed by symmetry a*t^2*a
           
       <t*a,                        [ 1, 0]>,
       <(t*a)^-1,                   [ 1, 0]>,
           
       <(a*t^2*a)*  t*a,            [ 1, 0]>,
       <(a*t^2*a)*  (t*a)^-1,       [ 1, 0]>]]);
\end{verbatim}
\end{minipage}
\end{center}
can be abbreviated to 
\begin{center}
\begin{minipage}{15.3cm}
\begin{verbatim}
VerifyLink(
    Bianchi15, Bianchi15P,
    [[4,0,1],[1,0,4]], 24,
    Symmetrize(
        a*t^2*a, 2,
        [[ <(t*a)^-1,       [ 1, 1]>,       // fixed by symmetry a*t^2*a
           <t^2*(t*a)^-1,   [ 0, 1]> ],
         [ <Id(Bianchi15),  [ 1, 0]>,
           <t^2,            [ 1, 0]> ]],
        [[ <Id(Bianchi15),  [ 1, 1]>,       // not fixed by symmetry a*t^2*a
           <t*a,            [ 1, 1]> ],
         [ <t*a,            [ 1, 0]>,
           <(t*a)^-1,       [ 1, 0]> ]]));
\end{verbatim}
\end{minipage}
\end{center}

\subsection{Fully symmetric links} Notice that all Dehn-filling coefficients $(p, q)$ are the same in the following example for the case $(7,\myLangle\frac{1+\sqrt{-7}}{2}\myRangle)$:
\begin{center}
\begin{minipage}{13cm}
\begin{verbatim}
Bianchi7<a,t,u> := Group<a,t,u|a^2,(t*a)^3,(a*t*u^-1*a*u)^2,(t,u)>;
Bianchi7P := [[ t, u ]];

VerifyLink(
    Bianchi7,Bianchi7P,
    [[2,0,1]], 6,
    [[<Id(Bianchi7), [ 0, 1]>,
      <t*a,          [ 0, 1]>,
      <(t*a)^2,      [ 0, 1]>]]);
\end{verbatim}
\end{minipage}
\end{center}
This means that every symmetry of the principal congruence manifold extends to the link (in fact, since the ideal has norm $\idealNorm(I)=2$, $B(I)\cong \PSL(2,O_d/I)$ is isomorphic to $A_3$ and acts as such on the three cusps of $M$). In cases where there is a link with this many symmetries, we can use the function \texttt{VerifyLink2} to verify a principal congruence link complement:
\begin{center}
\begin{minipage}{13cm}
\begin{verbatim}
VerifyLink2(Bianchi7, Bianchi7P, [[ 0, 1]], 6);
\end{verbatim}
\end{minipage}
\end{center}
This function checks that the quotient group $Q$ given by
\begin{center}
\begin{minipage}{13cm}
\begin{verbatim}
Q := quo<Bianchi7 | t^0 * u^1>;
Order(Q);
\end{verbatim}
\end{minipage}
\end{center}
has order 6 which is equal to $|\PSL(2,O_d/I)|$. Note that the quotient $Q$ has fewer relations than $B(I)$, so $|Q|=|\PSL(2,O_d/I)|$ implies $|B(I)|=|\PSL(2,O_d/I)|$ and that $N(I)$ is the fundamental group of $M=\H^3/\Gamma(I)$. Furthermore, it shows that $N(I)$ is generated by conjugates of a single parabolic element by $\PSL(2,O_d)$. Unless $d=1$ or $3$, these conjugates give a well-defined peripheral curve for each cusp of $M$ such that Dehn-filling along them trivializes the fundamental group of $M$.

For higher class numbers $h_d>1$, we need to give a list of pairs $(p_1,q_1), ..., (p_{h_d},q_{h_d})$ and the function checks the size of the quotient of $\PSL(2,O_d)$ by the normal subgroup generated by $p_{(1),1}^{p_1} p_{(1),2}^{q_1}, \dots, p_{(h_d),1}^{p_{h_d}} p_{(h_d),2}^{q_{h_d}}$.

Note that this technique does not work for $d=1$ and $d=3$ since the stabilizer of $\infty$ in $\PSL(2,O_1)$ and $\PSL(2,O_3)$ is larger than $P_\infty$ and contains torsion elements.

\clearpage

\bibliographystyle{hepMatthias}
\bibliography{technical_report}

\end{document}